\documentclass[11pt,oneside]{amsart}

\usepackage[a4paper, total={6in, 8in}]{geometry}

\usepackage{amsmath,amsthm,verbatim,amssymb,amsfonts,amscd, graphicx}
\usepackage{mathrsfs}
\usepackage[utf8]{inputenc}
\usepackage[retainorgcmds]{IEEEtrantools}
\usepackage{tikz-cd}
\usepackage{graphics}
\usepackage[all]{xy}
\usepackage{incgraph}
\usepackage{mathtools}
\usepackage{graphicx}
\usepackage{marginnote}
\usepackage{comment}
\usepackage{tikz}
\usepackage{geometry}
\usetikzlibrary{arrows.meta}
\usetikzlibrary{bending}

\usepackage{hyperref}

\theoremstyle{plain}
\newtheorem{theorem}{Theorem}[subsection]
\newtheorem{corollary}[theorem]{Corollary}
\newtheorem{lemma}[theorem]{Lemma}
\newtheorem{proposition}[theorem]{Proposition}

\newtheorem{definition}[theorem]{Definition}

\theoremstyle{remark}
\newtheorem{remark}[theorem]{Remark}
\newtheorem{example}[theorem]{Example}

\renewcommand{\O}{\mathcal O}
\usepackage{relsize}
\usepackage[bbgreekl]{mathbbol}
\usepackage{amsfonts}
\DeclareSymbolFontAlphabet{\mathbb}{AMSb} 
\DeclareSymbolFontAlphabet{\mathbbl}{bbold}
\newcommand{\Prism}{{\mathlarger{\mathbbl{\Delta}}}}
\newcommand{\wh}{\widehat}
\newcommand{\Z}{\mathbb Z}
\newcommand{\Q}{\mathbb Q}

\newcommand{\C}{\mathbb C}
\newcommand{\N}{\mathbb N}
\newcommand{\MM}{\mathfrak{M}}
\newcommand{\calR}{{\mathcal R}}
\newcommand{\Ainf}{A_{\mathrm{inf}}}
\newcommand{\cris}{\mathrm{cris}}
\newcommand{\Ker}{\operatorname{Ker}}
\newcommand{\Hom}{\operatorname{Hom}}
\newcommand{\coker}{\operatorname{coker}}
\newcommand{\st}{\mathrm{st}}
\newcommand{\id}{\operatorname{id}}
\newcommand{\Shv}{\operatorname{Shv}}
\newcommand{\Gal}{\operatorname{Gal}}
\newcommand{\Spd}{\operatorname{Spd}}
\newcommand{\Spa}{\operatorname{Spa}}
\newcommand{\Perfd}{\mathrm{Perfd}}
\newcommand{\Spf}{\operatorname{Spf}}
\newcommand{\opp}{\mathrm{opp}}
\newcommand{\perf}{\mathrm{perf}}
\newcommand{\Cont}{\operatorname{Cont}}
\newcommand{\At}{A^{(2)}}
\newcommand{\SSS}{\mathfrak{S}}
\newcommand{\ho}{\wh \otimes} 
\newcommand{\Omax}{{O}_{\mathrm{max}}} 
\newcommand{\Amax}{{\A}_{\mathrm{max}}} 
\newcommand{\deltalog}{\delta_{\log}} 
\newcommand{\z}{\mathfrak z}
\newcommand{\wt}{\widetilde}
\newcommand{\Fil}{{\textnormal{Fil}}}
\newcommand{\M}{\mathcal M}
\newcommand{\y}{\mathfrak y}
\newcommand{\BdR}{B_{\mathrm{dR}}} 
\newcommand{\A}{\mathbb A}  
\newcommand{\m}{\mathfrak m}
\newcommand{\ueps}{{\underline{\varepsilon}}}
\newcommand{\calD}{{\mathcal D}}
\newcommand{\Acris}{A_{\mathrm{cris}}}
\newcommand{\gA}{\mathfrak A}
\newcommand{\II}{\mathcal{I}}
\newcommand{\RepZp}{\mathrm{Rep}_{\Z_p}}
\newcommand{\RepQp}{\mathrm{Rep}_{\Q_p}}

\makeatletter
\newcommand{\colim@}[2]{%
  \vtop{\m@th\ialign{##\cr
    \hfil$#1\operator@font colim$\hfil\cr
    \noalign{\nointerlineskip\kern1.5\ex@}#2\cr
    \noalign{\nointerlineskip\kern-\ex@}\cr}}%
}
\newcommand{\colim}{%
  \mathop{\mathpalette\colim@{\rightarrowfill@\textstyle}}\nmlimits@
}
\makeatother

\makeatletter
\newcommand*{\da@rightarrow}{\mathchar"0\hexnumber@\symAMSa 4B }
\newcommand*{\da@leftarrow}{\mathchar"0\hexnumber@\symAMSa 4C }
\newcommand*{\xdashrightarrow}[2][]{%
  \mathrel{%
    \mathpalette{\da@xarrow{#1}{#2}{}\da@rightarrow{\,}{}}{}%
  }%
}
\newcommand{\xdashleftarrow}[2][]{%
  \mathrel{%
    \mathpalette{\da@xarrow{#1}{#2}\da@leftarrow{}{}{\,}}{}%
  }%
}
\newcommand*{\da@xarrow}[7]{%
  \sbox0{$\ifx#7\scriptstyle\scriptscriptstyle\else\scriptstyle\fi#5#1#6\m@th$}%
  \sbox2{$\ifx#7\scriptstyle\scriptscriptstyle\else\scriptstyle\fi#5#2#6\m@th$}%
  \sbox4{$#7\dabar@\m@th$}%
  \dimen@=\wd0 %
  \ifdim\wd2 >\dimen@
    \dimen@=\wd2 %
26

  \fi
  \count@=2 %
  \def\da@bars{\dabar@\dabar@}%
  \@whiledim\count@\wd4<\dimen@\do{%
    \advance\count@\@ne
    \expandafter\def\expandafter\da@bars\expandafter{%
      \da@bars
      \dabar@ 
    }%
  }%
  \mathrel{#3}%
  \mathrel{%
26

    \mathop{\da@bars}\limits
    \ifx\\#1\\%
    \else
      _{\copy0}%
    \fi
    \ifx\\#2\\%
    \else
      ^{\copy2}%
    \fi
  }%
  \mathrel{#4}%
}
\makeatother
\begin{document}
\author{Heng Du}
\address[Heng Du]{Department of Mathematics, Purdue University}
\email{du136@purdue.edu}

\author{Tong Liu}
\address[Tong Liu]{Department of Mathematics, Purdue University}
\email{tongliu@math.purdue.edu}

\title[A prismatic approach to $(\varphi, \hat G)$-modules and $F$-crystals]{A prismatic approach to $(\varphi, \hat G)$-modules and $F$-crystals}
\maketitle
\begin{abstract}
We give a new construction of $(\varphi, \hat G)$-modules using the theory of prisms developed by Bhatt and Scholze. We give two applications of our results. Firstly, we provide a new proof for the equivalence between the category of prismatic $F$-crystals in finite locally free $\O_\Prism$-modules over $(\O_K)_\Prism$ and the category of lattices in crystalline representations of $G_K$, where $K$ is a complete discretely valued field of mixed characteristic with perfect residue field. Moreover, we will generalize this result to semi-stable representations using the absolute logarithmic prismatic site defined by Koshikawa.
\end{abstract}
\tableofcontents
\section{Introduction}
\subsection{Overview and main results} Let $K$ be a complete discretely valued field of mixed characteristics with perfect residue field $k$. Fix a separable closure $\overline{K}$ of $K$ and let $G_K$ be the absolute Galois group of $K$. The study of stable lattices in crystalline representations of $G_K$ plays an important role in number theory. For example, in many modularity lifting results, one wants to understand liftings of mod $p$ representations of the Galois group of a number field $F$ to Galois representations over $\Z_p$-lattices with nice properties when restricted to the Galois groups of $F_v$ for all places $v$ of $F$. And a reasonable property at places over $p$ is that the representation of the Galois group of the local field is crystalline. There are various theories about characterizing $G_K$-stable lattices in crystalline representations, for example, the theory of strongly divisible lattices of Breuil(cf. \cite{BreuilIntegral}), Wach modules(cf. \cite{Wach96} and \cite{Berger}), Kisin modules(cf. \cite{KisinFcrystal}), Kisin-Ren's theory(cf. \cite{KisinRen}) and the theory of $(\varphi, \hat{G})$-modules(cf. \cite{liu-notelattice}). The theories above state that one can describe lattices in crystalline representations using linear algebraic data over certain commutative rings $A$. 

A recent work of Bhatt-Scholze\cite{BS2021Fcrystals} gives a different characterization of the category of lattices in crystalline representations. To explain their result, let $\O_K$ be the ring of integers in $K$, and they consider the absolute prismatic site $(\O_K)_{\Prism}$, which is defined as the opposite category of all bounded prisms over $\O_K$ and equipped with the faithfully flat topology. Let $\O_\Prism$ be the structure sheaf on $(\O_K)_{\Prism}$, and $\mathcal{I}_{\Prism} \subset \O_\Prism$ be the ideal sheaf of Hodge-Tate divisor, then $\O_\Prism$ carries a $\varphi$-action coming from the $\delta$-structures. A prismatic $F$-crystal in finite locally free $\O_\Prism$-modules over $(\O_K)_{\Prism}$ is defined as a crystal $\MM_\Prism$ over $(\O_K)_{\Prism}$ in finite locally free $\O_\Prism$-modules together with an isomorphism $(\varphi^\ast \MM_\Prism)[1/\mathcal{I}_{\Prism}] \simeq \MM_\Prism[1/\mathcal{I}_{\Prism}]$. The main result in \cite{BS2021Fcrystals} is the following:

\begin{theorem}\label{thm-intro-main-1}(\cite[Theorem 1.2]{BS2021Fcrystals} and Theorem~\ref{Thm-main-1})
There is an equivalence of the category of prismatic $F$-crystals in finite locally free $\O_\Prism$-modules over $(\O_K)_{\Prism}$ and the category of Galois stable lattices in crystalline representations of $G_K$.
\end{theorem}

It is known that the prismatic theory of Bhatt-Scholze was first developed to give a new cohomology theory in \cite{BS19} that behaves like the ``$p$-adic motivic cohomology" for varieties $X$ over $\Q_p$. More concretely, prismatic cohomology theory unifies a lot of $p$-adic cohomology theories in the sense that one can get various cohomology theories using ``evaluation maps" \cite[Example 1.9]{BS19}. And such a procedure of evaluation can be seen more clearly in the absolute prismatic cohomology theory developed in \cite{Bhatt-Lurie-Absoluteprismaticcohomology}. Theorem~\ref{thm-intro-main-1} should also be regarded as a ``unified"-integral $p$-adic Hodge theory for crystalline representations. That is one should be able to recover classical integral $p$-adic Hodge theories from the result of Bhatt-Scholze using evaluation maps. If a classical integral $p$-adic Hodge theory over $\O_K$ is defined using linear algebraic data over a commutative ring $A$, then one should first realize $A$ as certain prisms $(A,I)$ over $\O_K$, then expect that evaluating the prismatic $F$-crystals on $(A,I)$ recovers the corresponding theory. For example, in the theory of Kisin \cite{KisinFcrystal}, he uses the base ring $A=\mathfrak{S}:=W(k)[\![u]\!]$ with $\delta(u)=0$, and he needs to fix a uniformizer $\varpi$ of $\O_K$ which is a zero of an Eisenstein polynomial $E \in W(k)[u]$. Then it is well-known that $(A,(E))$ is the so-called Breuil--Kisin prism, and $(A,(E))$ is inside $(\O_K)_{\Prism}$. Kisin was able to attach any lattice $T$ in a crystalline representation of $G_K$ a finite free $A$-module together with an isomorphism $(\varphi^\ast \MM)[1/E] \simeq \MM[1/E]$. Now, if $\mathfrak{M}_{\Prism}$ is the prismatic $F$-crystal attached to $T$ under Theorem~\ref{thm-intro-main-1}, then Bhatt-Scholze show that the evaluation of $\mathfrak{M}_\Prism$ on $(A,(E))$ recovers Kisin's theory (cf. \cite[Theorem 1.3]{BS2021Fcrystals}).

\begin{figure}[h]
\tikzset{every picture/.style={line width=0.75pt}} 
\begin{tikzpicture}[x=1pt,y=1pt,yscale=-1,xscale=1]

\draw   (166.09,69.03) -- (219.52,161.29) -- (113.6,161.29) -- cycle ;
\draw    (19.22,118.5) -- (128.17,118.5) ;
\draw [shift={(130.17,118.5)}, rotate = 180] [color={rgb, 255:red, 0; green, 0; blue, 0 }  ][line width=0.75]    (10.93,-3.29) .. controls (6.95,-1.4) and (3.31,-0.3) .. (0,0) .. controls (3.31,0.3) and (6.95,1.4) .. (10.93,3.29)   ;
\draw    (221.81,94.19) -- (265.69,77.74) ;
\draw [shift={(267.56,77.04)}, rotate = 159.44] [color={rgb, 255:red, 0; green, 0; blue, 0 }  ][line width=0.75]    (10.93,-3.29) .. controls (6.95,-1.4) and (3.31,-0.3) .. (0,0) .. controls (3.31,0.3) and (6.95,1.4) .. (10.93,3.29)   ;
\draw    (199.22,116.03) -- (264.99,116.48) ;
\draw [shift={(266.99,116.5)}, rotate = 180.39] [color={rgb, 255:red, 0; green, 0; blue, 0 }  ][line width=0.75]    (10.93,-3.29) .. controls (6.95,-1.4) and (3.31,-0.3) .. (0,0) .. controls (3.31,0.3) and (6.95,1.4) .. (10.93,3.29)   ;
\draw    (221.81,138.8) -- (265.13,155.8) ;
\draw [shift={(266.99,156.53)}, rotate = 201.43] [color={rgb, 255:red, 0; green, 0; blue, 0 }  ][line width=0.75]    (10.93,-3.29) .. controls (6.95,-1.4) and (3.31,-0.3) .. (0,0) .. controls (3.31,0.3) and (6.95,1.4) .. (10.93,3.29)   ;
\draw    (119.22,161.03) -- (168.22,74.03) ;

\draw (10,97.4) node [anchor=north west][inner sep=0.75pt]   [align=left] {{\tiny Lattices in crystalline representations}};
\draw (132,147) node [anchor=north west][inner sep=0.75pt]   [align=left] {{\tiny Prismatic $F$-crystals}};
\draw (268.24,73) node [anchor=north west][inner sep=0.75pt]   [align=left] {{\tiny Breuil-Kisin theory}};
\draw (268.04,110) node [anchor=north west][inner sep=0.75pt]   [align=left] {{\tiny $(\varphi, \hat G)$-module theory}};
\draw (267.76,150) node [anchor=north west][inner sep=0.75pt]   [align=left] {{\tiny Wach theory $(K=K_0)$}\\{\tiny ...}};
\draw (199.96,103) node [anchor=north west][inner sep=0.75pt]   [align=left] {{\tiny evaluation maps }};
\end{tikzpicture}
    \caption{A cartoon that shows the expectations that prismatic $F$-crystals will reproduce many classical integral $p$-adic Hodge theories. And our prismatic $(\varphi,\hat{G})$-module theory can be regarded as reverse engineering of this procedure.}
    \label{fig:1}
    \centering
\end{figure}

\bigskip
\noindent
\textbf{Prismatic $(\varphi,\hat{G})$-modules.}
Figure~\ref{fig:1} shows an expectation of the relation between the prismatic result of Bhatt-Scholze with other works of characterizing lattices in crystalline representations. The first question that we consider here  is whether and how one can recover the theory of $(\varphi,\hat{G})$-modules from Theorem~\ref{thm-intro-main-1}. The category of $(\varphi,\hat{G})$-modules, developed by Liu in \cite{liu-Fontaine} roughly speaking, consisting of pairs $((\mathfrak{M},\varphi_{\mathfrak{M}}),\hat{G})$, where $(\mathfrak{M},\varphi_{\mathfrak{M}})$ is a Kisin module, and $\hat{G}$ is a $G_K$-action on $\mathfrak{M}\otimes_{\mathfrak{S},\varphi} \wh \calR$ that commutes with $\varphi_{\mathfrak{M}}$ and satisfying some additional properties. Here $\wh \calR$ is a subring of $\Ainf$ that is stable under $\varphi$ and $G_K$, where $\Ainf=W(\O_{\widehat{\overline{K}}}^\flat)$ is the infinitesimal period ring introduced by Fontaine. However, the period ring $\wh \calR$ introduced by Liu is not known to be $p$-adically complete, and it is even harder to determine whether it can appear as a prism. So in order to relate the category of $(\varphi,\hat{G})$-modules with the category of prismatic $F$-crystals of Bhatt-Scholze, we develop a theory of prismatic $(\varphi,\hat{G})$-modules, in which theory the ring $\wh \calR$ is replaced by $\At_{\st}$, a subring of $\Ainf$ constructed as a certain prismatic envelope in \S \ref{subsec-Ast}. 

We have the following result.

\begin{theorem}[Theorem~\ref{thm-log-main-1}]\label{thm-intro-2}
There is an equivalence between the category of prismatic $(\varphi,\hat{G})$-modules and lattices in semi-stable representations of $G_K$. Moreover, we have a necessary and sufficient condition about prismatic $(\varphi,\hat{G})$-modules to determine whether it corresponds to lattices in crystalline representations of $G_K$. 
\end{theorem}

Here a prismatic $(\varphi,\hat{G})$-module, defined similarly to the classical theory, is a pair $((\mathfrak{M},\varphi_{\mathfrak{M}}),\hat{G})$ consisting of a Kisin module $(\mathfrak{M},\varphi_{\mathfrak{M}})$, and $\hat{G}$ is a $G_K$-action on $\mathfrak{M}\otimes_{\mathfrak{S}} \At_{\st}$ satisfying several additional conditions. The ring $\At_{\st}$ here is indeed coming from a prism $(\At_{\st},(E))$ inside $(\O_K)_{\Prism}$ that lies over $(A,(E))$, and $(\At_{\st},(E))$ carries an action of $G_K$ inside $(\O_K)_\Prism$. For a $G_K$-stable lattice $T$ in a crystalline representation, if $\MM_{\Prism}$ is the prismatic $F$-crystal attached to $T$, then evaluating $\MM_{\Prism}$ on the diagram $(A,(E)) \to (\At_{\st},(E))$ recovers the prismatic $(\varphi,\hat{G})$-module attached to $T$. We can also show the map $\At_{\st} \to \Ainf \xrightarrow{\varphi} \Ainf$ factors through $\wh \calR$, so the theory of prismatic $(\varphi,\hat{G})$-modules recovers the classical theory. The ring $\At_{\st}$ is simpler than $\wh \calR$ in many ways. Although it is still very complicated and non-noetherian, it is $p$-adic complete and we can give an explicit description of $\At_{\st}$ modulo $E$. In particular, our new theory can be used to fix the gap in \cite{liu-Fontaine} indicated by \cite[Appendix B]{gao2021breuilkisin}.

{
Another benefit of having the period rings $A^{(2)}$ and $\At_{\st}$ used in the theory of $(\varphi, \hat G)$-modules in the absolute prismatic site is the ability to establish the concept of ``$(\varphi, \hat G)$-module cohomology" theory. Specifically, given a smooth formal scheme $\mathfrak{X}$ over $\O_K$, Bhatt-Scholze have demonstrated in \cite[Theorem 1.8]{BS19} that the Breuil--Kisin cohomology ${R\Gamma}_{\mathfrak{S}}(\mathfrak{X})$ previously studied in \cite{BMS2} can be realized by prismatic cohomology ${R\Gamma}_{\Prism}\big(\mathfrak{X}/(\mathfrak{S},(E))\big)$. This provides a geometric interpretation of the Breuil--Kisin modules associated with the $p$-adic \'etale cohomology of the adic generic fiber of $\mathfrak{X}$. Additionally, by making use of the base change property of prismatic cohomology, we can establish that ${R\Gamma}_{\Prism}\big(\mathfrak{X}/(\At_{\st},(E))\big)$ yields similarly a geometric interpretation of the $(\varphi, \hat G)$-modules associated with the $p$-adic \'etale cohomology of the adic generic fiber of $\mathfrak{X}$. It is worth noting that ${R\Gamma}_{\Prism}\big(\mathfrak{X}/(\At_{\st},(E))\big)$ not only possesses a Frobenius structure, but also admits the action of the full Galois group $G_K$. We anticipate that this observation will prove useful in studying integral $p$-adic cohomology theories.
}

By proving Theorem~\ref{thm-intro-2}, we actually provide reverse engineering of the procedure described in Figure~\ref{fig:1}. That is, using the known equivalence between lattices in semi-stable representations and prismatic $(\varphi, \hat G)$-modules, we can establish a functor from the category of prismatic $(\varphi, \hat G)$-modules that correspond to crystalline representations to prismatic $F$-crystals. Moreover, we show this functor is an equivalence, thus we give a different proof of the result of Bhatt-Scholze stated in Theorem~\ref{thm-intro-main-1}. 

To be more precise, let $T$ be a $G_K$-stable lattice in a crystalline representation with positive Hodge-Tate weights, let $(A, E)$ be the Breuil--Kisin prism, and let $(\At,(E))$ (resp. $(A^{(3)},(E))$) be the self-product (self-triple-product) of $(A, (E))$ in $(\O_K)_\Prism$. It is known that evaluating prismatic $F$-crystals on the diagram $(A,(E)) \xrightarrow{i_1} (\At,(E)) \xleftarrow{i_2} (A,(E))$ induces an equivalence of the category of prismatic $F$-crystals and Kisin modules with descent data, that is pairs $((\mathfrak{M},\varphi_{\mathfrak{M}}),f)$ where $(\mathfrak{M},\varphi_{\mathfrak{M}})$ is a Kisin module and 
$$
f: \mathfrak{M}\otimes_{\mathfrak{S},i_1} \At \simeq \mathfrak{M}\otimes_{\mathfrak{S},i_2} \At
$$
is an isomorphism of $\At$-modules that is compatible with $\varphi$ and satisfies cocycle condition over $A^{(3)}$. Using this, to establish an equivalence between prismatic $(\varphi, \hat G)$-modules that correspond to crystalline representations and prismatic $F$-crystals, it remains to find certain correspondence between the $\hat{G}$-action and the descent isomorphism $f$. 

We will show the descent isomorphism can be obtained by looking at the $G_K$-action of the $(\varphi,\wh{G})$-module at a specific element. To be more precise, fix a Kummer tower $K _\infty = \bigcup_{n = 1}^\infty K (\varpi _n )$ used in the theory of Kisin, where $\{\varpi _n\}_{n\geq 0}$ is a compatible system of $p^n$-th roots of $\varpi_0=\varpi$. Choose $\tilde{\tau} \in G_K$ satisfying $\tilde{\tau}(\varpi_n)=\zeta_{p^n}\varpi_n$ where $\{\zeta_{p^n}\}_{n\geq 0}$ is a compatible system of primitive $p^n$-th roots of $1$, then our slogan is that the descent isomorphism $f$ comes from the $\tilde{\tau}$-action on the Kisin module $\MM$ inside $T^{\vee}\otimes \Ainf$. We will call this the prismatic $(\varphi,\tau)$-module theory, which can be regarded as the $(\varphi,\tau)$-module version of the result of Wu \cite{wu2021galois}.

Actually we have the maps $u\mapsto [{\varpi}^\flat]$ and $u\mapsto [\tilde{\tau}({\varpi}^\flat)]$ defines two morphisms of $(A,(E))$ to $(\Ainf,E\Ainf)$. By the universal property of $(\At,(E))$, these two maps induce a morphism $(\At,(E)) \to (\Ainf,E\Ainf)$. We can show this map $\At \to \Ainf$ is injective, and it factors through $\At_{\st}$, which is the base ring used in our prismatic $(\varphi, \hat G)$-module theory. That is,  we have a chain of subrings $A\subset \At \subset \At_{\st}$ of $\Ainf$, such that $\tilde{\tau}(A)$ is also contained in $\At$. We can show a prismatic $(\varphi, \hat G)$-module corresponds to a crystalline representation if and only if the coefficients of the $\tilde{\tau}$-action on $\mathfrak{M}$ inside $T^{\vee}\otimes_{\Z_p} \Ainf$ lie in $\At$. And once this is proved, the $\tilde{\tau}$-action will induce an isomorphism:
$$
f_{\tau}: \mathfrak{M}\otimes_{\mathfrak{S},\tilde{\tau}} \At \simeq \mathfrak{M}\otimes_{\mathfrak{S}} \At.
$$
We will see $f_{\tau}$ as above gives a descent isomorphism. Consequently, this leads to a new proof for Theorem~\ref{thm-intro-main-1}.

\bigskip
\noindent
\textbf{Logarithmic prismatic $F$-crystals and lattices in semi-stable representations.} Another advantage of our approach is that our new method can be easily generalized to semi-stable representation cases. It turns out that the prism $(\At_{\st},(E))$ is isomorphic to the self-product of $(A,(E))$ if both of them are given proper log structures and realized as log prisms inside the absolute logarithmic prismatic site over $\O_K$ define by Koshikawa in \cite{Koshikawa2021log-prism}, for details one can see \S\ref{sec-logprismandsemistablereps}. Using the equivalence between prismatic $(\varphi, \hat G)$-modules and lattices in semi-stable representations of $G_K$. we will show in \S\ref{sec-logprismandsemistablereps} the following generalization of Theorem~\ref{thm-intro-main-1} for semi-stable representations. 

\begin{theorem}\label{thm-intro-log-main}(Theorem~\ref{thm-log-main-1})
There is an equivalence of the category of prismatic $F$-crystals in finite locally free $\O_{\Prism}$-modules over $(\O_K)_{\Prism_{\log}}$ and the category of Galois stable lattices in semi-stable representations of $G_K$.
\end{theorem}

\bigskip
\noindent
\textbf{Further discussions.} As illustrated in Figure~\ref{fig:1}, Bhatt-Scholze can package classical integral $p$-adic Hodge theories in prismatic $F$-crystals, and our work can be thought of as a way of finding minimal data when unpacking prismatic $F$-crystals while can still recover the whole package. 

 Integral $p$-adic Hodge theory is well-known for its contribution to the study of torsion Galois representations. It is expected in \cite{BS2021Fcrystals} that Theorem~\ref{thm-intro-main-1} should be upgraded using the stack $\Sigma''$ constructed by Drinfeld in \cite{Drinfeld-Prismatization} to study torsion representations. On the other hand, the theory of $(\varphi, \hat G)$-modules has been shown to be a powerful tool when studying torsion Galois representations, cf. \cite{Caruso-Liu2} and \cite{GLS-WeightPart}. As the theory of $(\varphi, \hat G)$-modules has been improved,  we hope the results of this paper can develop new ideas and tools to deal with torsion Galois representations. 

Another interesting and natural question one can ask is whether Theorem~\ref{thm-intro-main-1} and Theorem~\ref{thm-intro-log-main} can accommodate more general base schemes. It turns out that our strategy does allow us to treat many relative bases. See \cite{HLMS} for more details. So part of our paper, for example, \S\ref{sec-ring-strcuture} does allow specific general base rings.  

\subsection{Outline} Let us briefly overview the content in each section. We begin with some ring theoretic foundations of this paper in \S\ref{sec-ring-strcuture}. We will construct the period rings $\At$ and $\At_{\st}$ and collect some basic properties of them for our later use. Our construction of $\At$ and $\At_{\st}$ relies heavily on the theory of prisms developed by Bhatt-Scholze, we refer the reader to \cite{BS19} for the foundations of this theory. Note that we expect some of the results in this paper can allow more general base rings, so most of the discussions in \S\ref{sec-ring-strcuture} are proved for a quite general setup. In Appendix~\ref{subsec-baserings}, we will compare our setup with those used in \cite{Kim12} and \cite{Brinon}. In \S\ref{sec-prismaticphiGhatmodules}, we develop the theory of prismatic $(\varphi,\hat G)$-modules. We will state our main result about prismatic $(\varphi,\hat G)$-modules in Theorem~\ref{thm-2} and Corollary~\ref{cor-crystalline}, but both of them rely on Proposition~\ref{thm-1prime} which we will prove in \S\ref{sec-prismaticFcrystals}. In \S\ref{sec-prismaticFcrystals}, we will review the definition of prismatic $F$-crystals defined by Bhatt-Scholze and give a different proof of their prismatic description of crystalline lattices. One will see this will be achieved simultaneously as we complete the proof of Proposition~\ref{thm-1prime}. And in order to do this, we develop the prismatic $(\varphi,\tau)$-module theory in \S\ref{subsec-phi-tau}. In \S\ref{sec-logprismandsemistablereps}, we will briefly review the theory of logarithmic prisms developed by Koshikawa in \cite{Koshikawa2021log-prism}, then we will generalize the result of Bhatt-Scholze to semi-stable representations.

\bigskip
\noindent
\textbf{Acknowledgments.}
It is our pleasure to thank Hui Gao,  Wansu Kim, Teruhisa Koshikawa, Zeyu Liu, Yu Min, Yong Suk Moon, Peter Scholze, Koji Shimizu, Yupeng Wang, and Zhiyou Wu for their comments and conversations during the preparation of this paper. We also thank the anonymous referees for carefully reading our manuscript and for their many insightful comments and suggestions.

\section{Ring structures on certain prismatic envelope} \label{sec-ring-strcuture}

Recall that $K$ is a completed discrete valuation field in mix characteristic $(0 , p)$ with the ring of integers $\O_K$ and perfect residue field $k$. Let $W(k)$ be the ring of Witt vectors over $k$. Let $\varpi\in \O_K$ be a uniformizer and $E= E(u)\in W(k)[u]$ be the Eisenstein polynomial of $\varpi$. 
Let $\C_p$ be the $p$-adic completion of $\overline{K}$, and $\O_{\C_p}$ be the ring of integers. {In this paper, for a $p$-adic complete ring $T$, we write $T\langle x_1,\ldots,x_d \rangle$ to be the $p$-adic completion of $T[x_1,\ldots,x_d]$.} Let $R_0$ be a $W(k)$-algebra which admits a lifting of the $p$-th Frobenius $\varphi : R_0 \to R_0 $. Set $R: = R_0 \otimes_{W(k)}\O_K$. We make the following assumptions for $R_0$ and $R$: 
\begin{enumerate}
\item Both $R_0$ and $R$ are $p$-adically complete integral domains, and $R_0 / p R_0= R/ \varpi R$ is an integral domain;  
    \item Let $\Breve{R}_0 \coloneqq W(k)\langle t _1,  \dots , t _m \rangle$. $R_0$ is a $\Breve{R}_0 $-\emph{formally \'etale} algebra with $p$-adic topology;   
\item $\breve{R}_0$ admits a Frobenius lift such that $\breve{R_0} \to R_0$ defined in (2) is $\varphi$-equivalent. 

\item The $k$-algebra $R_0 / p R_0$ has finite $p$-basis in the sense of \cite[Definition 1.1.1]{deJong}.
\end{enumerate}
Our main example is $R_0= \breve R_0 =  W(k). $ We will not use the finite $p$-basis assumption until \S4. The following are other examples of  $R_0$:  
\begin{example}\label{Eg-1} 
\begin{enumerate}
\item $R_0 = W(k) \langle t _1^{\pm 1} ,  \dots , t _m ^{\pm 1}\rangle$ with $\varphi (t_j) = t ^p_j$
\item $R_0 = W(k) [\![t]\!]$ with $\varphi (t) = t^p$ or $(1+t)^p -1 $.
\item $R_0$ is an unramified complete DVR with imperfect field $\kappa$ with finite $p$-basis. See Appendix~\ref{subsec-baserings} for more discussions. 
\end{enumerate}
\end{example} 

We reserve $\gamma_i(\cdot)$ to denote the $i$-th divided power. 
\subsection{Construction of \texorpdfstring{$\At$}{A(2)}} \label{subsrc-construct-A2}
Let $A=\SSS=R_0[\![u]\!]$ and extend $\varphi : A \to A$ by $\varphi (u)= u^p$. It is well-known that $(A, E)$ is a prism and we can define a surjection $\theta: A \to R$ via $u\mapsto \varpi$. We have $\Ker\theta = (E(u))$.  Let $\breve A := \breve R_0 [\![u]\!]$ and define $\varphi$ and $\breve \theta:  \breve A \to \breve R : = \O_K \otimes _W \breve R_0$ similarly. 
We set \[ A ^{\ho 2}: = A [\![y -x, s_1 - t_1, \dots , s_m - t_m]\!],\]\[ A^{\ho 3}: = A[\![ y -x, w-x , \{ s_i - t _i , r_i - t_i\}_{j= 1, \dots , m}]\!].\] 
Note that $A ^{\ho 2}$  (resp. $ A^{\ho 3}$) is $\breve A \otimes_{\Z_p} \breve A $(resp. $\breve A \otimes_{\Z_p} \breve A \otimes_{\Z_p} \breve A$)-algebra by $ u \otimes 1 \mapsto x$, 
$1\otimes u \mapsto y$ and $1 \otimes t_i \mapsto s_i$ (resp. $1\otimes 1 \otimes u \mapsto w$ and $1 \otimes 1 \otimes  t_i \mapsto r_i$). So in this way, we can extend Frobenius $\varphi$ of $A$, which is compatible with that on $\breve A$ to $A ^{\ho 2}$ and $A^{\ho 3}$. 
Set $J ^{(2)}= (E, y -x , \{s_i- t _i \}_{i = 1, \dots,  m} )\subset A^{\ho 2}$ and $J ^{(3)} = (E, y-x , w-x  , \{s_i-t_i , r_i - t_i\}_{i = 1, \dots , m}) \subset A ^{\ho 3}.$ Clearly, we have $A^{\ho i}/ J ^{(i)}\simeq R$ for $i = 2, 3$. And we have $A^{\ho 2}/(p,E)$ (resp. $A^{\ho 3}/(p,E)$) is a formal power series ring over the variables $\bar{y}-\bar{x}, \{\bar{s}_i-\bar{t}_i\}_{i = 1, \dots,  m}$ (resp. $\bar{y}-\bar{x} , \bar{w}-\bar{x}  , \{\bar{s}_i-\bar{t}_i , \bar{r}_i - \bar{t}_i\}_{i = 1, \dots , m}$), so $(A,(E)) \to (A^{\ho i}, J ^{(i)})$ satisfies the requirements of  in \cite[Prop. 3.13]{BS19}, and we can construct the prismatic envelope with respect to this map, which will be denoted by $A^{(i)}$. More precisely,  $A^{(i)}\simeq A^{\ho i}\left \{\frac{J ^{(i)}}{E}\right\}_\delta^{\wedge}$, here $\{\cdot\}_\delta^{\wedge}$ means freely adjoining elements in the category of $(p, E(u))$-completed $\delta$-$A$-algebras. We will see  $A^{(i)}$, $i = 2 ,3$ are the self product and triple product of $A$ in category $R_\Prism$ in \S  \ref{subsec-pris-crystal}. 

\subsection{The ring \texorpdfstring{$A^{(2)}_{\max}$}{A2max}} Now we set $t_0 = x$, $s_0 = y$ and 
\[ z_j = \frac{s_i - t _i}{E} \textnormal{ and } z_0 = z=  \frac{y -x }{E}= \frac{s_0 - t_0}{E}. \]
Note that $A^{(i)}$ are  $A$-algebras via $u \mapsto x$ for $i=2,3$.
\begin{definition}
Let $\Omax$ be the $p$-adic completion of the $A$-subalgebra of $A[\frac{1}{p}]$ generated by $p^{-1}E$. And let $A_{\max}^{(2)}$ be the $p$-adic completion of the $A$-subalgebra of $ A [z_j ,  \frac{1}{p}; j = 0 , \dots , m ]$ generated by $p^{-1}E$ and $\{\gamma_i(z_j)\}_{i\geq 1, j = 0 , \dots , m}$. 
\end{definition}
We first note that $\At_{\max}$ is an $A^{\wh \otimes 2}$-algebra via $ (s_j - t_j ) = E z_j, j =0 ,\dots,  m$. Write $\iota: A^{\ho 2} \to \At_{\max}$ for the structure map. 
By construction, it is easy to see that $\At _{\max}\subset R_0[\frac 1 p] [\![ E, z_j, j = 0, \dots , m]\!]$. In particular, $\At_{\max}$ is a domain and    
any element $b\in A^{(2)}_{\max}$ can be \emph{uniquely} written as 
$\sum\limits_{i_0= 0}^\infty\cdots \sum\limits_{i_m= 0 }^\infty  b_{i_1 , \dots, i _m} \prod\limits_{j= 0}^m\gamma_{i_j} (z_{j})$ with $b _{i_0 , \dots , i_m} \in \Omax$ and $b_{i_0, \dots , i_m}\to 0$ $p$-adically when $i_0 + \cdots+  i_m \to \infty$.  

{Our next goal of this subsection is to show there is a natural way to extend $\varphi$ on $A^{\ho 2}$ to $A^{(2)}_{\max}$ in Lemma~\ref{lem-phionA^2max}. More importantly, we will show $A^{(2)}$ is a closed subring of $A^{(2)}_{\max}$ and stable under $\varphi$ in Proposition~\ref{prop-key-property} and Lemma~\ref{lem-subring}}
\begin{lemma}\label{lem-Omax}
$c: = \frac{\varphi(E)}{p}\in \Omax$ and $ c^{-1} \in \Omax$. 
\end{lemma}
\begin{proof}
We have $A$ is a $\delta$-ring, and $E$ is a distinguished element, so in particular
$$
\varphi(E)/p=c_0+E^p/p
$$
where $c_0=\delta(E)\in A^\times$. So $c = \varphi(E)/p\in \Omax$, and 
$c ^{-1} = c_0 ^{-1} \sum\limits_{i = 0}^\infty \frac {(- c_0^{-1} E^p ) ^i }{p ^i }\in \Omax. $
\end{proof}

{
\begin{lemma}\label{lem-phionA^2max}
If we define $\varphi(z)=\varphi(z_0)= \frac{y^p-x^p}{\varphi(E)}$ and $\varphi (z_j) = \frac{\varphi(s_j)  - \varphi(t_j) }{\varphi (E)}$, then for $0\leq j \leq m$, we have $\gamma_n(\varphi(z_j)) \in A^{(2)}_{\max}$ for $n\geq 0$. Moreover, $\varphi$ extends to a \emph{ring} map $\varphi: A^{(2)}_{\max} \to A^{(2)}_{\max}$.
\end{lemma}
\begin{proof}
We have
\begin{IEEEeqnarray*}{+rCl+x*}
\varphi(z)=\frac{y^p-x ^p}{\varphi(E)}&=&c^{-1}\frac{y^p-x^p}{p}=c^{-1}\frac{(x+ Ez)^p-x^p}{p} \\
&= & c^{-1}\sum_{i=1}^px^{p-i}(Ez)^i\binom{p}{i}/p \\&=& c^{-1}\sum_{i=1}^{p}a_iz^i,
\end{IEEEeqnarray*}
where $a_i\in W(k)[\![x]\!][\frac{E^p}{p}]\subset \Omax\subset A^{(2)}_{\max}$ and $c$ is a unit in $\Omax$. In particular, we have $ \varphi (z) \in A^{(2)}_{\max}$. Moreover 
$$
 \gamma_n(\varphi(z))=\frac{\varphi(z)^n}{n!}=\frac{z^n}{n!}(c^{-1}\sum_{i=1}^{p}a_iz^{i-1})^n
$$
is in $A^{(2)}_{\max}.$  The argument for $\varphi (z_j)$ for $j >1$ need a little more details. We claim that in $A ^{\ho  2}$, we have $\delta(s_j)-\delta (t_j) = (s_j - t_j) \lambda_j$ for some $\lambda_j \in A ^{\ho  2}$. Recall that $\delta(t_j)=\frac{\varphi(t_j)-t_j^p}{p}$ and $p,s_j-t_j$ is a regular sequence in $A ^{\ho  2}$ from our definition, so the claim follows from the fact that $\varphi(s_j)-\varphi(t_j)$ and $s_j^p-t_j^p$ are both divisible by $s_j-t_j$.
Using that $(s_j -t_j) = E z_j$, so
\begin{equation}\label{eqn-special-shape}
\varphi (z_j) = c^{-1} (\frac{s^p _j - t^p_j}{p} + E z_j \lambda_j)    
\end{equation}
The same argument as that for $\varphi (z_0)$ also shows that $\gamma_n (z_j)\in \At_{\max}$, for $j =1 , \dots , m$. \\
Since any element $b\in A^{(2)}_{\max}$ can be uniquely written as $$b=\sum\limits_{i_0= 0}^\infty\cdots \sum\limits_{i_m= 0 }^\infty  b_{i_1 , \dots, i _m} \prod\limits_{j= 0}^m\gamma_{i_j} (z_{j})$$ with $b _{i_0 , \dots , i_m} \in \Omax$ and $b_{i_0, \dots , i_m}\to 0$ $p$-adically when $i_0 + \cdots+  i_m \to \infty$. So this allows us to extend Frobenius $\varphi $ on $ A$ to a ring map $\varphi: A^{(2)}_{\max} \to A^{(2)}_{\max}$ by letting $\varphi(u)= u^p$,  $\varphi(z) = \frac{y ^p -x^p}{\varphi(E)}$,  $\varphi (z_j) = \frac{\varphi(s_j)  - \varphi(t_j)}{\varphi (E)}$,  and $\gamma_i (z_j) \mapsto \gamma_i (\varphi (z_j))$ for $i\geq 1$.
\end{proof}
}

\begin{remark}\label{rem-not-Frob-lift} The ring map $\varphi: \At_{\max}\to \At_{\max}$ is \emph{not} a Frobenius lift of $\At _{\max}/ p$ because $\varphi (E/p)- (E/p) ^p \not \in p \At_{\max}$. In particular, $\At_{\max}$ is not a $\delta$-ring. 
\end{remark}

Recall that $\At_{\max}$ is an $A^{\ho 2}$-algebra via map  $\iota : A^{\ho 2}\to \At_{\max}$. The above construction of Frobenius $\varphi$ on $\At _{\max}$ is obviously compatible with  $\iota$. 

Our next goal is to show that $\iota$ induces a map $A^{(2)}\to \At_{\max}$ so that $\At $ is a subring of $A^{(2)}_{\max}$ which is compatible with $\varphi$-structures and filtration. We need a little preparation. 
Write $\z_{n}= \delta^n (z)$ with $\delta_0(z)= z= \z_0$, and $A_0 = W(k) [\![u]\!]$. 
\begin{lemma}\label{lem-delta-n} $$\delta^n (Ez) = b_n \z_{n} + \sum_{i = 0} ^p a^{(n)}_i \z_{n -1} ^i.  $$
where $a^{(n)}_i \in A_0 [\z_0, \dots, \z_{n -2}]$ so that $a^{(n)}_p\in A_0 ^\times$ and for $0 \leq i \leq p-1$ each monomials of  $a^{(n)}_i $ contains a  factor $\z_{j}^p$  for some  $ 0 \leq j\leq n -2$. Furthermore,  $b_{n+1 } =p\delta (b_n ) + b^p_n $ and $b_1 = p \delta (E) + E^p$. 
\end{lemma}
\begin{proof} Given  $f \in A_0[x_1 , \dots , x_m]$, if each monomials of $f$  contains $x_j^l$ for some $j$ and $l \geq p$ then we call $f$  \emph{good}.  For example, $f= x_1^p x_2 + 2 x_ 1x_2^{p +3}.$ So we need to show that $a^{(n)}_i\in A_0[\z_0, \dots, \z_{n -2}]$ is good.  Before making induction on $n$, we discuss some properties of good polynomials. It is clear that the set of good polynomials is closed under addition and multiplications. Note that 
\begin{equation}\label{eqn-delta}
\delta(\z_{l}^i)= \frac{\varphi (\z_{l}^i) - \z_{l} ^{p i}}{p}= \frac{(p\z_{l +1} + \z_{l } ^p)^i - \z_{l}^{p i}}{p} = \sum \limits_{j  = 1 }^{i} \binom{i}{j}(p^{j -1} \z_{l }^{p(i-j)} ) \z_{l+1} ^{j }. 
\end{equation}
In particular, given an $f\in A_0[\z_0 , \dots , \z_{m}]$,  $\delta(\z_{m}^p f)= f^p \delta (\z_{m} ^p) +  \z_{m}^{p ^2}\delta (f) + p \delta (\z_{m} ^p)\delta (f)$  is  a good polynomial in $A[\z_0 , \dots, \z_{m+1}]$. Using the fact that 
$\delta (a+b)=\delta(a) + \delta (b) + F(a, b)$ where  $F(X, Y) = \frac 1 p ( X^p + Y^p - (X+Y)^p) = - \sum\limits_{i = 1}^{p-1} \binom{p}{i}/p X^i Y^{p-i}$, together with the above argument of $\delta(\z_l^p f)$, it is not hard to show that if $g\in A_0[\z_0 , \dots , \z_{m}]$ is good  then $\delta (g) \in A_0[\z_0 , \dots , \z_{m}, \z_{m +1}]$ is also good. 

Now we make induction on $n$. When $n =1$, we have 
$$\delta (Ez)= E^p \z_{1} + z^p \delta(E) + p \delta (E) \z_{1}= (p \delta (E) + E^p) \z_{1} + \delta(E) z^p.$$ 
Then $b_1 = p \delta (E) + E^p $,  $a^{(1)}_p= \delta(E)\in A_0 ^\times$ and $a^{(1)}_i = 0$ for $1 \leq i \leq p-1$ are required. Now assume the formula is correct for $n$, then 
\begin{IEEEeqnarray*}{+rCl+x*}
\delta ^{n +1} (Ez) &=& \delta (b_n \z_{n} + \sum_{i = 0} ^p a^{(n)}_i \z_{n -1} ^i) \\
&=& \delta (b_n \z_{n}) + \delta (\sum_{i = 0} ^p a^{(n)}_i \z_{n -1} ^i)) + F(b_n \z_{n}, \sum_{i = 0} ^p a^{(n)}_i \z_{n -1} ^i)).
\end{IEEEeqnarray*}
Clearly, $ F(b_n \z_{n}, \sum\limits_{i = 0} ^p a^{(n)}_i \z_{n -1} ^i))  = \sum\limits_{j = 1} ^{p-1} \tilde a^{(n)}_j \z_{n} ^j$ with $\tilde a^{(n)}_j$ being good. An easy induction shows $\delta (\sum\limits_{i = 0} ^p a^{(n)}_i \z_{n -1} ^i) = \sum\limits_{i = 0} ^p \delta (a^{(n)}_i \z_{n -1} ^i) + f$ with $f \in A_0[\z_0, \dots, \z_{n -1}]$ being good. Plug the formula of $\delta(\z_{n-1}^i)$ from \eqref{eqn-delta} into
$$\delta (a^{(n)}_i \z_{n -1} ^i)= (a^{(n)}_i)^p \delta (\z_{n -1}^i) + (\z_{n-1}^{pi})\delta (a_i ^{(n)}) + p \delta (\z_{n -1}^i  ) \delta (a_i^{(n)}),$$ and using the fact that $a^{(n)}_i$ is good, we have 
$\delta (a_i^{(n)}) $ is also good, we conclude that for $0 \leq i \leq p-1$,  $$\sum\limits_{i = 0} ^{p-1} \delta (a^{(n)}_i \z_{n -1} ^i) = \sum_{i =0}^{p-1} \alpha_i \z_{n} ^i$$ with $\alpha_i \in A_0 [\z_0, \dots, \z_{n-1}]$ being good polynomials. Using that $a _p^{(n)} \in A_0 ^\times$, we compute that  $\delta (a_p^{(n)}\z^p_{n-1}) = \sum \limits_{i = 0}^p \beta_i \z_{n}^i$ with $\beta_p \in p A_0$ and $\beta_j\in A_0 [\z_0 , \dots , \z_{n-1}]$ being good for $1\leq j \leq p-1$. 
 Now we only need to analyze $\delta (b_n \z_{n})$, which is 
 $\delta (b_n ) \z_{n}^p + b_n^p \z_{n+1} + p \delta (b_n)\z_{n+1}$. So $b_{n+1} = p \delta(b_n) + b_n ^p$ and $a_p^{(n+1)} = \delta (b_n) + \beta_p$. Since $\delta(b_n) \in A_0^\times$, we see that $a_p^{(n+1)}= \delta(b_n) + \beta_p \in A_0^\times$ as required. 
\end{proof}

Let $\wt \At := A^{\ho 2} [z_j]_\delta= A^{\ho2} [\delta ^n (z_j), n \geq 0, j =0 , \dots , m]$ and natural map $\alpha : \wt \At \to \wt \At [\frac 1 p]$ (we do not know $\alpha$ is injective at this moment). We need the following result for Lemma~\ref{lem-auxiliaryrings} which is crucial for our later applications. \footnote{{We want to note that there was a gap in the proof of Lemma~\ref{lem-auxiliaryrings} in our previous preprint. We thank Yong Suk Moon for pointing it out and helping us complete the following lemma.}} 

\begin{lemma}\label{lem:gamma(z)-polynomial-in-E/p}
For $i\geq 0$ and $j=0,1,\ldots,d$, there exists $f_{ij}(X) \in \wt \At [X]$ such that, as elements of $\wt \At[\frac 1 p] $ via $\alpha: \wt \At \to \wt \At [\frac 1 p]$, 
\[
\gamma_i(z_j) = f_{ij}\Bigl(\frac{E}{p}\Bigr).
\]
\end{lemma}
\begin{proof}
Write $z = z_{j}$ for simplicity, and let $\tilde{\gamma}(z)= \frac{z^p}{p}$ and $\tilde{\gamma}^n = \underbrace{\tilde{\gamma} \circ \tilde{\gamma} \cdots \circ \tilde{\gamma}}_n$. It suffices to show that for each $n \geq 1$, we have $\tilde{\gamma}^n(z) = f_n(\frac{E}{p})$ inside $\wt \At[\frac 1 p]$ for some $f_n(X) \in \wt \At [X]$. For an element $x \in A[\delta^i(z)]_{i \geq 0}$, we say that $x$ has \emph{$\delta$-order $\leq n$} if $x$ can be written as a sum of monomials such that each term is \emph{not} divisible by $\delta^j(z)$ for $j > n$, namely if $x\in \sum_{0\leq j\leq n}A[\{\delta^i(z)\}_{0 \leq i \leq n }] \delta^j(z)$.

We claim that the following two equations hold for each $n \geq 1$:
\begin{enumerate}
\item We have
\begin{equation} \label{eq:delta(z)}
    \delta^n(z) = \nu_n \tilde{\gamma}^n(z)+P_n\Bigl(\frac{E}{p}\Bigr)+\frac{E^p}{p}d_n\delta^n(z)
\end{equation}
for some $\nu_n \in A^{\times}$, $d_n \in A$, and $P_n(X) \in (A[\delta^i(z)]_{i \geq 0})[X]$ such that each coefficient of $P_n(X)$ has $\delta$-order $\leq n-1$. 

\item We have
\begin{equation} \label{eq:gamma(delta(z))}
    \tilde{\gamma}(\delta^{n-1}(z)) = \mu_{n-1}\tilde{\gamma}^n(z)+Q_{n-1}\Bigl(\frac{E}{p}\Bigr)
\end{equation}
for some $\mu_{n-1} \in A^{\times}$ and $Q_{n-1}(X) \in (A[\delta^i(z)]_{i \geq 0})[X]$ such that each coefficient of $Q_{n-1}(X)$ has $\delta$-order $\leq n-1$.
\end{enumerate}

We prove claims (1) and (2) by induction. For $n = 1$, since
\[
\delta(Ez) = z^p\delta(E)+(p\delta(E)+E^p)\delta(z)
\]
and $\delta(E) \in \mathfrak{S}^{\times}$, we have
\[
\delta(z) = -\tilde{\gamma}(z)+\delta(E)^{-1}\frac{\delta(Ez)}{p}-\delta(E)^{-1}\frac{E^p}{p}\delta(z). 
\]
By easy induction, we also have $\delta^i(Ez) \in (Ez)A$ for each $i \geq 1$. So claim (1) holds. Claim (2) holds for $n = 1$ trivially with $Q_0(X) = 0$.

Suppose that claims (1) and (2) hold for $1 \leq n \leq m$. We will verify claims (1) and (2) for $n = m+1$. We first consider claim (2). Since each coefficient of $P_m(X)$ has $\delta$-order $\leq m-1$, $\frac{E^p}{p}=p^{p-1}\bigl(\frac{E}{p}\bigr)^p$, and Equations \eqref{eq:delta(z)} and \eqref{eq:gamma(delta(z))} hold for $1\leq n \leq m$, applying $\tilde{\gamma}(\cdot)$ to Equation \eqref{eq:delta(z)} for $n = m$ yields
\[
\tilde{\gamma}(\delta^m(z)) = \nu_m^p \tilde{\gamma}^{m+1}(z)+Q_m\Bigl(\frac{E}{p}\Bigr)
\]
for some $Q_m(X) \in (\mathfrak{S}[\delta^i(z)]_{i \geq 0})[X]$ such that each coefficient of $Q_m(X)$ has $\delta$-order $\leq m$. This proves claim (2) for $n = m+1$.

We now consider claim (1) for $n = m+1$. By the above Lemma for $n = m+1$ and that $b_n= p\alpha_n +\beta_n  E^p$
for some $\alpha_n \in A^{\times}$ and $\beta_n \in  A$ (via an easy induction on $n$), we have 
\begin{IEEEeqnarray*}{+rCl+x*}
\alpha_{m+1}\delta^{m+1}(z) = \frac{\delta^{m+1}(Ez)}{p}-\beta_{m+1}\frac{E^p}{p}\delta^{m+1}(z) &-& a_p^{(m+1)}\tilde{\gamma}(\delta^m(z)) \\
&-&\frac{1}{p}\sum_{j=0}^{p-1} a_j^{(m+1)}(\delta^{m}(z))^j.
\end{IEEEeqnarray*}
As noted above, we have $\delta^{m+1}(Ez) \in (Ez)A$. Furthermore, by the condition on $a_j^{(m+1)}$, the last term $\frac{1}{p}\sum_{j=0}^{p-1} a_j^{(m+1)}(\delta^{m}(z))^j$ is a linear combination of terms involving $\tilde{\gamma}(\delta^l(z))=\frac{1}{p}(\delta^l(z))^p$ for some $0\leq l\leq m-1$.
Thus, by applying Equations~\eqref{eq:delta(z)} and \eqref{eq:gamma(delta(z))} for $1 \leq n \leq m$, we see that claim (1) also holds for $n = m+1$ with $\nu_{m+1} = -\alpha_{m+1}^{-1}a_p^{(m+1)}\mu_m$ and $d_{m+1} = -\alpha_{m+1}^{-1}\beta_{m+1}$.
This completes the induction and proves the lemma.
\end{proof}

\begin{remark}
In the above proof, by equation~\eqref{eq:gamma(delta(z))}, we even have for each $i,j\geq 0$, $\gamma_i(\delta^j(z))=f(\frac{E}{p})$ for some $f\in \wt \At [X]$. 
\end{remark}

An easy induction by \eqref{eq:delta(z)} implies $\alpha (\delta^n (z) ) \in A^{\ho 2}[\{\gamma_i (z_j )\}_{i \geq 0 , j =1 , \dots , m}, \frac E p]$ inside $\At_{\max}$, which satisfies equations in Lemma \ref{lem-delta-n} by replacing 
$\z_n$ by $\alpha (\delta^n (z))$ inside $\At_{\max}$. It is clear that $\iota$ is still Frobenius compatible (because both $A ^{\ho 2}$ and $\At_{\max}$ are domains). Since $E = p \frac E p$, $\iota$ is continuous for $(p , E)$-topology on $\wt \At $ and $p$-topology on $\At_{\max}$. Finally, we construct a ring map $\iota: \At \to \At_{\max}$ so that $\iota$ is compatible with Frobenius.  

Our next goal is to show that $\iota$ is injective. We define 
$$\Fil ^i A^{(2)}_{\max}[\frac 1 p]:  = E^i A^{(2)}_{\max}[\frac 1 p].$$ 
And for any subring $B \subset A^{(2)}_{\max}[\frac 1 p]$, set 
$$\Fil ^i B : = B \cap \Fil ^i A^{(2)}_{\max}[\frac 1 p]= B \cap E^i A^{(2)}_{\max}[\frac 1 p].$$
Let $D_z$ be the $p$-adic completion of $R [\gamma_i (z_j), i \geq 0;  j = 0, \dots , m]$. 
\begin{proposition} \label{prop-key-property}
\begin{enumerate}
    \item $\wt \At/ E = R [\gamma_i (z_j ), i \geq 0; j = 0, \dots , m]$.
    \item $ A^{(2)}/ E \simeq D_z$.
    \item $\iota$ is injective.
    \item $\Fil ^1 A^{(2)}= E A ^{(2)}$.
    \item $A^{(i)}$ are flat over $A$ for $i = 2, 3$.
\end{enumerate}
\end{proposition}
\begin{proof}
(1)  By definition, $\wt \At  = A ^{\ho  2}[ z^{(n)}_j , n\geq 0; j = 0 , \dots, m ]/ J $ where $\mod J  $ is equivalent to modulo out the following relations (note that $z_0 = z$):  $\delta  (z^{(n)}_j)= z^{(n+1)}_{j}, \delta ^n (Ez) = \delta^n (y -x) , \delta ^n (Ez_j)= \delta ^n (s_j-t _j)$ for $n\geq 0$. Since $\delta (x-y)= \frac{(x^p - y ^p)- (x-y)^p}{p}$ and $\delta(s_j - t_j) = \frac{\varphi (s_j - t_j)- (s_j - t_j)^p}{p}$, {using the fact that $p, s_i-t_j$ and $p,x-y$ are regular sequence,} one can prove that $\delta^n (x-y)$  and $\delta^n (s_j - t_j)$ always contains a factor $(x-y)$, $s_j - t_j$ and hence $\delta^n (x-y), \delta(s_j - t_j)\equiv 0 \mod E$. Therefore $\delta ^n (Ez_j) \equiv 0 \mod E$. By Lemma \ref{lem-delta-n}, we see that 
$$p \mu_n z^{(n )}_{j} = -\sum_{i = 0} ^p  \overline{a^{(n)}_i} (z^{(n -1)}_j) ^i  \mod E \text{ and } pz^{(1)}_j = z_j^p \mod E $$
where $\overline {a^{(n)}_i} = a^{(n)}_i \mod E$ and $\mu_n = \frac{\delta (b_n)}{p}\mod E \in \O_K^\times$. Using that $a_p^{(n)} \in A_0^\times$, and $a_i^{(n)}, 1 \leq i \leq p-1$ are good in the sense that they contain factor of $(z^{(l)}_j) ^p$ for some $l = 0 , \dots, n-2$, we easily see by induction that 
$\wt \At /E = R [\wt \gamma^n (z_j ), n \geq 0; j = 0 , \dots , m  ]$. But it is well-known that $R [\wt \gamma^n (z_j), n \geq 0; j = 0 , \dots , m ] = R [\gamma_n (z_j), n \geq 0; j = 0 , \dots , m]. $

Now we show that the natural map $\iota:  \wt \At\to \At_{\max}[\frac  1 p]$ induced by $\alpha (\delta ^n(z_j))$ is injective. Note that $\wt \At $ is the direct limit of $\wt{A_n^{(2)}}$ where $\wt{A_n^{(2)}}: =  A ^{\hat \otimes 2}[\{ \delta^i (z_j)\}_{i = 1 , \dots , n, j = 0 , \dots , m}]$. A argument similar as above shows that $\wt{A_n^{(2)}} / E $ injects to $\At_{\max}[\frac 1 p]/E = D_z [\frac 1 p]$. 
Since $\wt{A_n^{(2)}}$ is $E$-separate and  $\At_{\max}$ is a domain, this implies that $\wt{A_n^{(2)}}$ injects to $\At_{\max}[\frac 1 p]$. So 
$\wt \At$ injects to $\At_{\max}$ via $\iota$. 

(2) Since $A^{(2)}$ is $(p, E)$-completion of $\wt \At$ \footnote{Indeed, $A^{(2)}$ is \emph{derived} $(p, E)$-completion. Since $\wt \At/ E$  is $\Z_p$-flat, then derived completion coincides with the classical completion, which is used here.}, we have a natural map from $\bar \iota: A^{(2)}/E \to D_z$. The surjectivity of $\bar \iota$ is straightforward as $A^{(2)}$ is also $p$-complete. To see injectivity, given an sequence $f_n $ so that $f_{n +1}- f _n \in  (p , E)^n \wt \At $ and $ f_n = E g_n$ for all $n$, we have to show that $g_n$ is a convergent sequence in $A^{(2)}$.  Since $E (g_{n +1} - g_n) = \sum_{i = 0} ^n p ^i E ^{n - i} h_i$ with $h _i \in \wt \At $. Then $ E|p ^n h_n $. Since $\wt \At /E$ has no $p$-torsion,
we have $E | h_n $ and write $h_n  = E h'_n $. Since $\wt \At $ is a  domain as it is inside the fraction field of $A ^{\ho 2}$, we see that 
$ g_{n +1}- g_n = p ^n h'_n + \sum\limits_{i = 0} ^{n -1} p ^i E^{n - i - 1} h_i $. Hence $g_n$ converges in $ A^{(2)}$ as required. 
 
(3) It is clear that $A^{(2)}_{\max}[\frac 1 p]/ E \simeq D_z [\frac 1 p]$. So the map $\iota \mod E(u)$ induces an injection $D_z \hookrightarrow D_z [\frac 1 p]$. So for any $x\in \Ker (\iota)$, we see that $x= Ea$ for some $a \in A^{(2)}$. As $A^{(2)}_{\max}$ is $E$-torsion free and $A^{(2)}$ is $E$-complete, we see that $x= 0$ as required.  
 
(4) By the definition of $\Fil ^1 \At  $, we see that $E A ^{(2)} \subset \Fil ^1 A ^{(2)}$ and $ A^{(2)} / \Fil ^1 A^{(2)} $ injects to $A^{(2)}_{\max}[\frac 1 p ]/E= D_z[\frac 1 p]$. But we have seen that $A^{(2)}/E = D_z$ injects to $D_z$. Then $\Fil ^1 A^{(2)} = E A^{(2)}$. 

(5) Both $A^{(2)}$ and $A^{(3)}$ are obtained by the construction of \cite[Proposition 3.13]{BS19}, which implies that they are $(p,E)$-complete flat over $A$. Since $A$ is Noetherian, by \cite[Tag 0912]{stacks-project}, we have both $A^{(2)}$ and $A^{(3)}$ are $A$-flat.
\end{proof}
\begin{corollary}\label{cor-filtration-shape}
\begin{enumerate}
    \item $\Fil ^i A ^{(2)} = E^i A^{(2)}. $
    \item $A^{(i)}$ are bounded prisms for $i = 2, 3$. 
\end{enumerate}
\end{corollary} 
\begin{proof}These follow that $\At / E \At \simeq D_z$ which is $\Z_p$-flat. For (2), we have $A^{(2)}$ and $A^{(3)}$ are $(p,E)$-complete flat over $A$, so boundedness follows from (2) in \cite[Lemma 3.7]{BS19}.
\end{proof}

\begin{lemma}\label{lem-subring} $A^{(2)}$ is a closed subset inside $A^{(2)}_{\max}$.  
\end{lemma}
\begin{proof} We need to show the following statement:  Given $ x \in \wt \At $, if $x= p ^n y$ with $y \in A^{(2)}_{\max}$ then $x = \sum\limits_{i = 0}^n p ^{n-i} E^i x_i$ with $x_i \in \wt \At . $ Indeed, since $A^{(2)}/E \simeq A^{(2)}_{\max}/ {\Fil ^1}$, there exists $x_0, w_1 \in \wt \At $ so that $x= p ^n x_0 + E w_1$. Then $Ew_1 \in p ^n  A^{(2)}_{\max}$. Write $ E w_1= p ^n \sum \limits_{i =0} ^\infty\sum\limits_{j =0}^m f_{ij} \gamma_i (z_j)$, we see that $f_{ij}= \sum_{l \geq 1} a_{ijl}  \frac{E^l}{p^l}\in \Fil ^1 \Omax$. So it is easy to see that $p ^n E^{-1}f_{ij} \in p ^{n -1} \Omax$ and then 
$w_1 = p ^{n -1} x_1 $ with $x_1 \in \At_{\max}$. Then we may repeat the above argument to $w_1$, and finally $x= \sum\limits_{i = 0}^n p ^{n-i} E^i x_i$ with $x_i \in \wt \At$ as required. 
\end{proof}

\begin{remark}\label{rem:compare BS_part1}
This remark is prompted by feedback provided by an anonymous referee. We have seen $\iota$ induces $\iota'\colon\At\langle \frac{E}{p} \rangle \to A^{(2)}_{\max}$, where $\At\langle \frac{E}{p} \rangle$ is the $p$-adic completion of $\At[\frac{E}{p}]$. Moreover, since $E$ is divisible by $p$ in both $\At\langle \frac{E}{p} \rangle$ and $A^{(2)}_{\max}$, the $p$-adic topology is the same as the $(p,E)$-adic topology over these rings. Combining these with Lemma~\ref{lem:gamma(z)-polynomial-in-E/p} and Equation~\ref{eq:delta(z)}, one has $\iota'$ has both left and right inverses and is thus an isomorphism. Notably, we have $A^{(2)}_{\max}=\O_\Prism\langle \frac{\mathcal{I}_\Prism}{p}\rangle\big((\At,(E))\big)$, where $\O_\Prism\langle \frac{\mathcal{I}_\Prism}{p}\rangle((A,I))$ is equal to the $p$-adic completion of $A[ \frac{I}{p} ]$, as a presheaf on $R_\Prism$. We want to remark that $\O_\Prism\langle \frac{\mathcal{I}_\Prism}{p}\rangle$ bears a close relationship to the sheaf $\Prism_{\bullet}\langle \frac{{I}}{p}\rangle$ on the (small) quasi-syntomic site over $R$, which is defined in \cite[\S6]{BS2021Fcrystals}. Additionally, using similar techniques in \cite{DuLiu-NewMethod}, one can show $\O_\Prism\langle \frac{\mathcal{I}_\Prism}{p}\rangle$ is a sheaf on the absolute \emph{transversal} prismatic site and contains $\O_\Prism$ as a subsheaf. This may give a direct verification of the inclusion of $\At$ in $A^{(2)}_{\max}$. However, to further our subsequent arguments, it is useful to write down the elements of $A^{(2)}_{\max}$ in a more explicit manner.
\end{remark} 

Now we realize $A^{(2)}$ as a subring of $A^{(2)}_{\max}$ via $\iota$. We need to introduce some auxiliary rings. By the description of elements in $A^{(2)}_{\max}$, we define ${\wt S}$ be the subring of $A^{(2)}_{\max}$ as follows
$$
\widetilde{S} :=  A^{(2)} [\![ \frac {E ^p}{p} ]\!] := \{ \sum_{i \geq 0} a_i (\frac{E^p}{p})^i \mid a_i\in A^{(2)} \}.
$$
And when $p=2$, we define 
$
\widehat{S} := \At [\![\frac{E^4}{2}]\!]
$ similarly. 
We will have $\widehat{S} \subset \widetilde{S} \subset A^{(2)}_{\max}$. Viewing $\wt S$ and $\widehat{S}$ as subrings of $A^{(2)}_{\max}$, we give them the filtration induced from $A^{(2)}_{\max}$. We will use the fact that if $\{a_i\}_{i\geq 1}$ is a sequence in $A^{(2)}$ that converges to $0$, then for any sequence $\{b_i\}_{i\geq 1}$ inside $ \widetilde{S}$, one can show $\sum_i a_i b_i \in \widetilde{S}$ by rewriting it into formal power series in $\frac{E^p}{p}$.

\begin{lemma}\label{lem-auxiliaryrings}
Fix $h \in \N$, then we have
\begin{enumerate}
    \item We have $\varphi(A^{(2)}_{\max}) \subset \widetilde{S} \subset A^{(2)}_{\max}$, and when $p=2$, we have $\varphi(\widetilde{S}) \subset \widehat{S} \subset \widetilde{S}$;
    \item $x \in \Fil ^h \wt S $ if and only if $x$ can be written as
    $$
    x = \sum\limits_{i \geq h } a_i \frac {E ^i}{p ^{\lfloor \frac i p\rfloor}} 
    $$
    with $a_i\in A^{(2)}$. 
    \item when $p>2$, there is a $h_0>h$ such that $\varphi (\Fil ^m \wt S ) \subset A ^{(2)} + E^h\Fil^{m +1} \wt S$ for all $m > h_0$;
    \item when $p=2$, then $x \in \Fil ^h \widehat S $ if and only if $x$ can be written as
    $$
    x = \sum\limits_{i \geq h } a_i \frac {E ^i}{2 ^{\lfloor \frac i 4\rfloor}} 
    $$
    with $a_i\in A^{(2)}$;
    \item when $p=2$, there is a $h_0>h$ such that $\varphi (\Fil ^m \widehat{S}  ) \subset A ^{(2)} + E^h\Fil^{m +1} \widehat{S} $ for all $m > h_0$.
\end{enumerate}
\end{lemma}
\begin{proof}
For $(1)$, any $a\in \At_{\max}$, we can write 
$$
a = \sum _{i_0 = 0}^\infty \cdots \sum_{i _m = 0}^\infty \sum_{l = 0}^\infty  a_{i_0 , \dots , i_m, l } \left (\frac E p\right ) ^l \prod_{j = 0}^m \gamma_{i_j} (z_j)
$$
where $a_{i_0 , \dots , i_m, l}\in A$ and $a_{i_0 , \dots , i_m, l}\to 0$ $p$-adically when $\sum_j i _j + l \to \infty$. Thanks for Lemma \ref{lem:gamma(z)-polynomial-in-E/p}, we see that $b_{i_0 , \dots , i_m , l}: = \varphi \left  (\left (\frac E p\right ) ^l \prod_{j = 0}^m \gamma_{i_j} (z_j) \right)\in \wt S$. So $\varphi (a) = \sum a_{i_0 , \dots , i_m , l} b_{i_0 , \dots , i_m , l}$ converges in $\wt S$. 

For the claim in $(1)$ for $p=2$, we have $\varphi(\frac{E^2}{2})=(E^2+2b')^2/2= \frac{E^4}{2} + 2b$ for some $b,b'\in A$. And for $a = \sum_{i \geq 0} a_i (\frac{E^p}{p})^i \in \widetilde{S}$, we have 
\begin{IEEEeqnarray*}{+rCl+x*}
\varphi(a) = \sum_{i \geq 0} \varphi(a_i) (\frac{\varphi(E^2)}{2})^i &=& \sum_{i \geq 0} \varphi(a_i) \sum_{j=0}^{i} c_{ij}(2b)^{i -j} (\frac{E^4}{2})^j \\
&=& \sum_{j \geq 0} \left (\sum_{i=j}^{\infty} \varphi(a_i)c_{ij} (2b)^{i-j} \right ) (\frac{E^4}{2})^j
\end{IEEEeqnarray*}
for some $c_{ij} \in \Z$. We have $\varphi(a) \in \widehat{S}$ since $\sum_{i=j}^{\infty} \varphi(a_i)c_{ij} (2b)^{i-j}$ converges in $\At$.

For $(2)$, the if part is trivial. For the other direction, any $x \in \wt S $, we have 
$$
x = \sum\limits_{i \geq 0 } a_i \frac {E^{i}}{p^{\lfloor \frac i p  \rfloor}}.
$$
And if we also have $x \in \Fil ^h A^{(2)}_{\max}[\frac  1 p ] = E^h A^{(2)}_{\max}[\frac 1 p]$, Then $\tilde a_0=\sum\limits_{0\leq i \leq h} a_i \frac {E^{i}}{p^{\lfloor \frac i p \rfloor}}$ will lie inside $\Fil ^h \At [\frac 1 p]$. This implies $p^{\lfloor \frac h p\rfloor}\tilde a_0 \in \Fil ^h A^{(2)}= E^h A^{(2)}$. That is $\tilde a_0={p^{-\lfloor \frac h p\rfloor}}{E^h} b$ for some $b \in A^{(2)}$. So $x$ is of the given form. The proof for $(4)$ is similar.

For $(3)$, we have by $(2)$, $x \in \Fil ^m \wt S$, $x$ can be written as
$$
x = \sum\limits_{i \geq m } a_i \frac {E ^i}{p ^{\lfloor \frac i p\rfloor}}.
$$
And use the fact $\varphi(E)=E^p+pb$ for some $b \in \At$, we have 
\begin{IEEEeqnarray*}{+rCl+x*}
\varphi(x) &=& \sum\limits_{i \geq m } \varphi(a_i) \sum_{j=0}^{i} \frac {c_{ij}E^{p(i-j)} p^j}{p ^{\lfloor \frac i p\rfloor}} \\
&=& \sum_{i \geq m} \sum_{ j \geq \lfloor \frac i p\rfloor}^{i} \frac {b_{ij}E^{p(i-j)} p^j}{p ^{\lfloor \frac i p\rfloor}} + \sum_{i \geq m} \sum_{0 \leq j < \lfloor \frac i p\rfloor} E^h\frac {b_{ij}E^{p(i-j)-h} p^j}{p ^{\lfloor \frac i p\rfloor}}
\end{IEEEeqnarray*}
with  $b_{ij} \in A^{(2)}$.

In particular, we have $\sum_{i \geq m} \sum_{ j \geq \lfloor \frac i p\rfloor}^{i} \frac {b_{ij}E^{p(i-j)} p^j}{p ^{\lfloor \frac i p\rfloor}}$ is inside $\At$. To prove $(3)$, it is amount to find $h_0$ such that whenever $m>h_0$, $i \geq m$ and $0 \leq j < \lfloor \frac i p\rfloor$, we have 
$$
\sum_{i \geq m} \sum_{0 \leq j < \lfloor \frac i p\rfloor} \frac {b_{ij}E^{p(i-j)-h} p^j}{p ^{\lfloor \frac i p\rfloor}} \in \Fil^{m +1} \wt S.
$$
The claim follows if we can find $h_0 > h$ such that $\frac {E^{p(i-j)-h} p^j}{p ^{\lfloor \frac i p\rfloor}} \in \wt S$ and $p (i-j)- h \geq m +1$ for all $m>h_0$, $i \geq m$ and $0 \leq j < \lfloor \frac i p\rfloor$. That is $\lfloor\frac{p (i -j)-h}{p} \rfloor +j \geq \lfloor \frac i p \rfloor $ and $p(i-j)-h\geq m+1$ for all $i,j,m$ in this range. And solve this we have it is enough to choose $h_0 > \max \{h, \frac{p(h+1)+1}{p(p-2)}\}$, which is valid for $p>2$.

Statement in $(5)$ is similar to $(3)$. Any $x \in \Fil ^m \widehat S$, $x$ can be written as
$$
x = \sum\limits_{i \geq m } a_i \frac {E ^i}{2^{\lfloor \frac i 4\rfloor}}.
$$
We have $\varphi(E)=E^2+2b$ for some $b \in \At$, so 
\begin{IEEEeqnarray*}{+rCl+x*}
\varphi(x) &=& \sum\limits_{i \geq m } \varphi(a_i) \sum_{j=0}^{i} \frac {c_{ij}E^{2(i-j)} 2^j}{2 ^{\lfloor \frac i 4\rfloor}} \\
&=& \sum_{i \geq m} \sum_{ j \geq \lfloor \frac i 4\rfloor}^{i} \frac {b_{ij}E^{2(i-j)} 2^j}{2 ^{\lfloor \frac i 4\rfloor}} + \sum_{i \geq m} \sum_{0 \leq j < \lfloor \frac i 4\rfloor} E^h\frac {b_{ij}E^{2(i-j)-h} 2^j}{2 ^{\lfloor \frac i 4\rfloor}}.
\end{IEEEeqnarray*}
Similar to the argument in $(3)$, it is amount to find $h_0$ such that whenever $m>h_0$, $i \geq m$ and $0 \leq j < \lfloor \frac i 4\rfloor$, we have $\lfloor (i -j)-\frac h 2 \rfloor + j  \geq \lfloor \frac i  4 \rfloor $ and $2(i-j)-h \geq m +1$. It is enough to choose $h_0 > 2(h+2)$. 
\end{proof}

If $A$ is a ring then we denote by  ${\textrm M} _d (A)$ the set of $d\times d$-matrices with entries in $A$.

\begin{proposition}\label{prop-desecnt} Let $Y \in {\textrm M}_d (A^{(2)}_{\max})$ so that $E^ h Y = B  \varphi (Y) C$ with $B$ and $C$ in ${\textrm M}_d (A ^{(2)})$ then   $Y$ is in ${\textrm M}_d (A ^{(2)}[\frac 1 p])$. 
\end{proposition}
\begin{proof} First, we claim that there is a constant $s$ only depending on $h$, such that the entries of $p^s Y$ is in $\wt S$. By $(1)$ of Lemma~\ref{lem-auxiliaryrings}, entries of  $E ^ h Y$ are in $\wt S$. So for each entry $a$ of $Y$, we can write $E^h a = \sum \limits_{i = 0}^\infty a_i \frac{ E^{pi}}{p^i}$ with $a_i \in A ^{(2)}$. It is clear that $E^h p ^h a = a' + E^h \sum\limits_{i \geq h} a_j \frac {E^{pi-h }}{p ^i} $ so that $a' \in A^{(2)}$. Therefore, $a' \in \Fil^h A^{(2)} = E^h A ^{(2)}$ by Corollary \ref{cor-filtration-shape}. So write $a' =  E^h b$, we have $ p ^h a = b + \sum\limits_{i \geq h} a_j \frac {E^{pi-h }}{p ^i}$. In particular, we see that $p ^{2h} a \in \wt S$, this proves our claim. When $p=2$, then we may repeat the above argument and we can assume $p^s Y$ is in ${\textrm M} _d (\widehat{S})$.

Let $\mathfrak{R}=\wt S$ when $p>2$ and $\mathfrak{R}=\widehat{S}$ when $p=2$, then we may assume $Y$ is inside ${\textrm M} _d (\mathfrak{R})$. Then we claim there is another constant $r$ only depends on $h$, such that for each entry $a$ of $p^r Y$ {changed $Y$ into $p^r Y$}, there is a sequence $\{b_i\}_{i\geq 1}$ in $\At$ such that we have $a - \sum\limits_{i = 0} ^m b _i E^ i \in \Fil ^{m +1} \mathfrak{R}$. Note that once this is known, we will have $\sum\limits_{i = 0} ^m b _i E^ i$ converges to an element $b$ in $\At$, and $a-b=0$ since it is in $\Fil ^{m} \mathfrak{R}$ for all $m\in \N$. 

It remains to show our claim. When $p>2$, let $h_0$ be the integer in $(3)$ of Lemma~\ref{lem-auxiliaryrings}, then it is easy to show there is a constant $r$ only depends on $h_0$ (so only on $h$) and sequence $\{b_i\}_{i=1}^{h_0}$ such that for each entry $a$ of $Y':=p^rY$, we have
$$
a - \sum\limits_{i = 0} ^{h_0} b _i E^ i \in \Fil ^{h_0 +1} \mathfrak{R}.
$$
Now we show our claim by induction, assuming for each entry $a$ in $Y'$, there is a sequence $\{b_i\}_{i=1}^{m}$ such that,
$$
a - \sum\limits_{i = 0} ^{m} b _i E^ i \in \Fil ^{m +1} \mathfrak{R}.
$$
for some $m \geq h_0$.  So we can write $Y'$ as
$$
\sum_{i = 0}^m Y_i E^i + Z_{m+1},   
$$
with $Y_i \in {\textrm M}_d (A^{(2)})$ and $Z_{m+1} \in {\textrm M}_d (\Fil ^{m +1} \mathfrak{R})$. Writing $X_{m}= \sum_{i = 0}^m Y_i E^i$, then $E^h Y' = B \varphi (Y') C$ implies 
$$
E^hZ_{m+1} = B\varphi(X_m)C -E^hX_m + B\varphi(Z_{m+1})C.
$$
By $(3)$ in Lemma~\ref{lem-auxiliaryrings}, we have $B\varphi(Z_{m+1})C = A_{m+1} + E^h B_{m+1}$, with $A_{m+1} \in {\textrm M}_d (A^{(2)})$ and $B_{m+1} \in {\textrm M}_d (\Fil ^{m +2} \mathfrak{R})$. One has $B\varphi(X_m)C -E^hX_m + A_{m+1} \in {\textrm M}_d (\Fil^{h +m +1 } \At )$, so $B\varphi(X_m)C -E^hX_m + A_{m+1}=E^{h+m+1} Y_{m+1}$ with $Y_{m+1} \in {\textrm M}_d (A^{(2)})$. Moreover, we have $Y - \sum_{i = 0}^{m+1} Y_i E^i = B_{m+1} \in {\textrm M}_d (\Fil ^{m +2}\mathfrak{R})$ as required.

At last, when $p=2$. We know we can assume $Y$ is inside ${\textrm M}_d (\widehat{S})$. Then by repeating the above arguments by replacing $(3)$ in Lemma~\ref{lem-auxiliaryrings} with $(5)$, we can also prove our claim.
\end{proof}

\begin{remark}
The above proposition will be used to prove the ``boundedness of descent data at the boundary", similar to the results exhibited in \cite[\S 6.3]{BS2021Fcrystals}. Specifically, with $R=\O_K$, in the proof of Lemma~\ref{lem-equivalenceofthm}, we will establish that any $\varphi$-equivariant descent data of a Kisin module defined over $A^{(2)}_{\max}$ will automatically descend to $A^{(2)}[\frac 1 p]$. Alternatively, if $(A^{(2)}_{\inf},(d))$ is the initial prism over the quasi-regular semiperfectoid ring $\O_{\C_p}\wh{\otimes}\O_{\C_p}$, whose existence is ensured by \cite[Proposition 7.2]{BS19} and admits two natural maps from $(\Ainf,\ker\theta)$, and since we mentioned that $A^{(2)}$ is certain product of $A$ in the absolute prismatic site over $R$ (cf. Proposition~\ref{prop-selfproduct} for the explicit statement), we can construct a universal map $A^{(2)} \to A^{(2)}_{\inf}$ once fix a map $A\to \Ainf$. Moreover, in \cite[Proposition 6.10]{BS2021Fcrystals}, it is demonstrated that any $\varphi$-equivariant descent data of a Breuil--Kisin--Fargues module defined over $A^{(2)}_{\inf}\langle\frac{d}{p} \rangle[\frac{1}{p}]$ will automatically descend to $A^{(2)}_{\inf}[\frac 1 p]$. Additionally, if the Breuil--Kisin--Fargues module comes from the base change of Kisin module, then these two results on descents morphisms are actually equivalent using quasi-syntomic descent, providing that $A\to \Ainf$ is faithfully flat (so $\Ainf$ also applies to Proposition~\ref{prop-cover-final-object}), and the description of $A^{(2)}_{\max}$ from Remark~\ref{rem:compare BS_part1}. However, the proof provided in \cite{BS2021Fcrystals} crucially uses the Beilinson fiber square developed in \cite{AMMN22}. On the other hand, our proof is entirely explicitly algebraic and mainly uses the finite $E$-height condition of Kisin modules. 
\end{remark}

\subsection{The ring \texorpdfstring{$A^{(2)}_{\st}$}{A(2)st}}\label{subsec-Ast} We assume that $R = \O _K$ in the following two subsections. 
For our later use for semi-stable representations, we construct $A^{(2)}_{\st}$ as the following: Define $\varphi$ on $W(k)[\![x, \y]\!]$ by $\varphi(x)= x^p$ and $\varphi (\y ) = (1+\y)^p -1$ and set $w = \frac{\y}{ E}$. Set $ A^{(2)}_{\st}: = W(k)[\![x,\y ]\!]\{w\}_\delta^\wedge$ where $\wedge$ means $(p, E)$-completion. Similarly, we define $A^{(3)}_{\st}=W(k)[\![x,\y, \z ]\!]\{\frac{\y}{E},\frac{\z}{E}\}^\wedge_{\delta}$, with the $\delta$-structure on $W(k)[\![x,\y, \z ]\!]$ given by $\delta(x)=0$, $\varphi(\y)=(\y+1)^p-1$ and $\varphi(\z)=(\z+1)^p-1$. Define $A^{(2)}_{\st,\max}$ to be the $p$-adic completion of $W(k)[\![x, \y]\!][w, \frac E p , \gamma_i (w), i \geq 0].$ It is clear that for any $f \in A^{(2)}_{\st ,\max}$ can be written uniquely $a = \sum\limits_{i= 0}^\infty f_i \gamma_i (w) $ with $f_i \in \Omax$ and $f_i \to 0$ $p$-adically. For any subring $B\subset A^{(2)}_{\st , \max}[\frac 1 p]$, we set $\Fil ^i B : = B \cap E^i A^{(2)}_{\st , \max}[\frac 1 p]$ and $D_w $ the $p$-adic completion of $\O_K [\gamma_i(w), i \geq 0]$. 

It turns out that $A^{(2)}$ and $A^{(2)}_{\st}$ share almost the same properties by replacing $z$ with $w$. 
So we summarize all these properties in the following: 
\begin{proposition}\label{prop-Ast-properties}
\begin{enumerate}
 \item One can extend Frobenius from $A$ to $A^{(2)}_{\st, \max}$.  
    \item There exists an embedding $\iota : A^{(2)}_{\st} \hookrightarrow A^{(2)}_{\st , \max}$ so that $\iota$ commutes with Frobenius. 
    \item $A^{(2)}_{\st} \cap E ^ i A^{(2)}_{\st , \max}[\frac 1 p] = E A^{(2)} _{\st}$. 
    \item $A ^{(2)}_{\st}/E \simeq D_w = A ^{(2)}_{\st , \max}/ \Fil ^1 A^{(2)}_{\st , \max}.$
    \item $\At_{\st}$ is closed in $\At_{\st, \max}$. 
    \item $\At_{\st}$ and $A^{(3)}_{\st}$ are flat over $A$, and in particular they are bounded. 
    \item Proposition \ref{prop-desecnt} holds by replacing $A^{(2)}_{\max}$ and $A^{(2)}$ by 
    $A^{(2)}_{\st} $ and $A^{(2)}_{\st, \max}$ respectively. 
\end{enumerate}

\end{proposition}
\begin{proof}
All previous proof applies by noting the following difference 
$$\varphi(w)= \varphi ( \frac {\y}{E}) = c ^{-1} \frac {1}{p} \sum_{i =1}^p \binom{p}{i} \y ^i= c ^{-1 }\sum _{i =1}^{p-1} \y ^i\binom{p}{i}/ p  + c^{-1}\frac{E^p w^p}{p}.   $$
Also $\delta (\y) = \sum\limits _{i =1}^{p-1} \y ^i\binom{p}{i}/ p$ always contains $\y$-factor and this is a key input for the analogy of Lemma \ref{lem:gamma(z)-polynomial-in-E/p}. 

For the boundedness of $A^{(3)}_{\st}$, we have 
$$ 
W(k)[\![x,\y, \z]\!]/(p,E)\simeq (\O_K/p)[\![\bar{\y}, \bar{\z}]\!]
$$ so $\{\y, \z\}$ form a $(p,E)$-complete regular sequence, and by \cite[Proposition 3.13]{BS19}, $A^{(3)}_{\st}$ is also $A$-flat, and this implies $A^{(3)}_{\st}$ is bounded by (2) in Lemma 3.7 of $loc.cit.$. 
\end{proof}

Note that $A^{\ho 2}= W(k) [\![x, y]\!]\subset W(k)[\![x,
 \y]\!]$ via $y = x(\y +1)$ or equivalently $\y = \frac y x -1 $. It is clear that this inclusion is a map of $\delta$-rings. By the universal property of prismatic envelope to construct $\At$, the inclusion induces a map of prisms $\alpha: \At\to \At_{\st}$. Since $z= xw$, we easily see that $\At_{\max} \subset \At_{\st, \max}$. So $\At\subset \At _{\st}$ via $\alpha$. We will see that $\At$ (resp. $\At_{\st}$) is the self product of $A$ in category $X_{\Prism}$ (resp. $(X, M_X)_{\Prism_{\text{log}}}$) in 
 \S \ref{subsec-pris-crystal} and \S \ref{sec-logprismandsemistablereps}. Then the existence of $\alpha: \At \to \At_{\st}$ can be explained by the universal property of self-product. See \S\ref{sec-logprismandsemistablereps} for details. 

To simplify our notation, let $B^{(2)}_{\st}$ (resp. $B^{(3)}_{\st}$, $B^{(2)}$, $B ^{(3)}$) be the $p$-adic completion of $ {\At_{\st}} [\frac 1 E]$ (resp. $A^{(3)}_{\st}[\frac 1 E]$, $\At [\frac 1 E]$, $A^{(3)}[\frac 1 E]$). 
\begin{lemma}\label{lem-intersection} \begin{enumerate}
\item $ A^{(i)}_{\st}  \subset B^{(i)}_{\st}\subset B^{(i)}_{\st}[\frac 1 p]$ and $ A^{(i)} \subset B ^{(i)} \subset B ^{(i)}[\frac 1 p]$ for $i = 2, 3$. 
    \item $B^{(2)}_{\st} \cap {\At_{\st}} [\frac 1 p] = \At_{\st}$ and $B ^{(2)} \cap {\At} [\frac 1 p] = \At$. 
\end{enumerate}
\end{lemma}
\begin{proof} Here we only prove the case $\At$ while the proofs for $\At_{\st}$, $A^{(3)}$ and $A^{(3)}_{\st}$ are almost the same. 

By Proposition \ref{prop-key-property}, $\At$ is a subring of $\At _{\max}\subset K_0  [\![x, z]\!]$. So $\At$ and hence $\At [\frac 1 E]$
is an integral domain. Then $B^{(2)} $ has no $p$-torsion: Assume that $x \in B^{(2)}$ so that $p x = 0 $. Suppose that $x_n \in \At [\frac 1 E]$ so that $x \equiv x _n \mod p ^n$. Then $p x_n \equiv 0 \mod p ^n \At [\frac 1 E]$. Since $\At [\frac 1 E]$ is domain,  $x_{n}\equiv 0 \mod p ^{n -1}$. Hence  $x = 0 $. As $B^{(2)}$ has no $p$-torsion, we see that $B^{(2)}\subset B^{(2)}[\frac 1 p]$. 
To see the natural map $\At \to B^{(2)}$ is injective, it suffices to show that $\At / p \At $ injects to $ \At / p\At [\frac 1 u]= B ^{(2)}/ pB^{(2)}$. Clearly, this is equivalent to that $\At / p\At$ has no $u$-torsion. Note that $\At$ is obtained by taking prismatic envelope of $A^{\ho 2}= W(k)[\![x, z]\!]$ for the ideal $ I = (z)$. As mentioned before, we can apply \cite[Prop. 3.13]{BS19} to our situation. So $\At$ is flat over $A$ and hence $\At / p \At$ has no $u$-torsion as desired.

Now we can regard $B^{(2)}$ and $\At [\frac 1 p]$ as  subrings of $B^{(2)}[\frac 1 p]$. In particular, $ B^{(2)} \cap \At [\frac 1 p ]$ makes sense and contains $\At$. For any $x\in B^{(2)} \cap \At [\frac 1 p ]$, if $x\not \in \At$ but $p x \in \At$. Then the image of $y = px  $ inside $\At/ p \At$ is nonzero but the image of $y$ in $B ^{(2)}/ p B ^{(2)}$ is zero. This contradicts to that $\At / p \At $ injects to $B^{(2)}/ p B^{(2)}$. So such $x$ can not exist and we have $B ^{(2)} \cap {\At} [\frac 1 p] = \At$ as required. \end{proof}

By \cite[Lem. 3.9]{BS19}, any prism $(B, J)$ admits its perfection $(B,J)_{\perf}=(B_{\perf}, JB _{\perf})$. 
\begin{remark}
In \cite{BS19}, the underlying $\delta$-ring of $(B,J)_{\perf}$ is denoted by $(B_{\infty},JB_{\infty})$, and $B_{\perf}$ is defined as the direct perfection of $B$ in the category of $\delta$-rings. In this paper, we write $B_{\perf}$ as the $(p,J)$-adic completion of $\mathrm{colim}_{\varphi} B$, which also coincides with the derived $(p, I)$-completion of $\mathrm{colim}_{\varphi} B$ (cf. Lemma 3.9 of $loc.cit.$).
\end{remark}

\begin{lemma}\label{lem-perfisflat}
We have $(\At)_{\perf}$ and $(\At_{\st})_{\perf}$ are $A$-flat.
\end{lemma}
\begin{proof} We have seen that $\At$ is $A$-flat via $i_1$. And it is easy to see $\varphi$ on $A$ is flat. Since $i_1$ is a $\delta$-map, so we have $\varphi^n\circ i_1 =i_1 \circ \varphi^n$ which is flat. So $\mathrm{colim}_\varphi \At$ is flat over $A$. In particular, we will have $A_{\perf}$ is $(p , E)$-complete flat over $A$. Now since $A$ is Noetherian, by \cite[Tag 0912]{stacks-project}, we have $(\At)_{\perf}$ is $A$-flat. The proof for $(\At_{\st})_{\perf}$ is the same. 
\end{proof}

\subsection{Embedding \texorpdfstring{$A^{(2)}$}{A(2)} and \texorpdfstring{$A^{(2)}_{\st}$}{A(2)st} to \texorpdfstring{$\Ainf$}{Ainf}}\label{subsec-embedding}
Let $\Ainf=W(\O_{\C_p}^\flat)$, then there is a surjection $\theta: \Ainf \to \O_{\C_p}$ and $\Ker\theta=(E)$. And let $\BdR^+$ be the $\Ker\theta$-adic completion of $\Ainf[\frac{1}{p}]$.

\begin{definition}
Let $\A_{\max}$ be the $p$-adic completion of the $\Ainf$-subalgebra of $\BdR^+$ generated by $E/p$. 
\end{definition}
It can be easily seen that $\varphi(E/p):=\varphi(E)/p\in A_{\cris}\subset \A_{\max}$ is well-defined and it extends the Frobenius structure on $\Ainf$ to an endomorphism on $\Amax$.

Let $\{\varpi_n\}_{n\geq 0}$ be a compatible system of $p^n$-th roots of $\varpi_0=\varpi$ and $\{\zeta_n\}_{n\geq 0}$ be a compatible system of $p^n$-th roots of 1. Write $\varpi^\flat  : = (\varpi_n)_{n\geq 0}$ and $\zeta^\flat : = (\zeta_n)_{n\geq 0}$ as elements in $\O_{\C_p}^\flat$ and we let $u=[\varpi^\flat ]$, $\epsilon=[\zeta^\flat]$, $v=\epsilon u$ and $\mu=\epsilon-1$ be elements inside $\Ainf$. We can regard $W(k)[\![x,y]\!]$ as a subring of  $\Ainf$ via $x \mapsto u$ and $y\mapsto v$. Consider $z' = \frac{u-v}{E}\in \Ainf [\frac 1 E]$. Since $u -v = u (\epsilon -1)$ is clearly inside $\Ker (\theta )$ and $ \Ker (\theta) = E \Ainf$, we conclude that $ z' \in \Ainf$. Hence we have a natural map (of $\delta$-rings) $\iota_A : \wt\At \to \Ainf$ via $z\mapsto z'$,  which naturally extends to $\iota _A : A^{(2)}\to \Ainf$ because $(p, E)$-topology of $A^{(2)}$ matches with the weak topology of $\Ainf$. Similarly, we have map of $\delta$-rings $\iota_{\st}: A^{(2)}_{\st} \to A_{\inf}$ via $ x \mapsto u$ and $\y \mapsto \epsilon-1$ and $w\mapsto \frac{\epsilon-1}{E}$. 
\begin{remark}\label{rem-embedding-depend} Once we know that $\At$ is self-product of $A$ inside $X_\Prism$ with $X= \Spf (\O_K)$ as explained in \S \ref{subsec-pris-crystal}. The map $\iota_A$ can be constructed as follows: First we fix an embedding $A\to \Ainf$ by sending $x \mapsto u =  [\varpi^\flat]$. Then $A\to \Ainf$ by $x \to v= \epsilon u$ is another map of prisms. By universal property of $\At$, these two maps define a map $\iota_A : \At \to \Ainf$. Clearly, the map $\iota_A : \At \to \Ainf$ depends on choice of ${\varpi}^\flat = (\varpi_n)_{n \geq 0}$ and ${\zeta}^\flat = (\zeta_n)_{n \geq 0}$. Also $\iota_A$ is a special case of $\iota^{(2)}_\gamma$ defined by \eqref{equ-diagram-prisms} in \S\ref{subsec-pris-crystal-proof}. Indeed if $\gamma ([w^\flat])= [\zeta^\flat][w ^\flat]$ then $\iota_A = \iota^{(2)}_{\gamma}$. Similarly comment also applies to $\iota_{\st}$. 
\end{remark}

\begin{proposition}
There is a unique embedding 
$
\begin{tikzcd}
A_{\max}^{(2)} \arrow[r, hookrightarrow] & \A_{\max}
\end{tikzcd}
$ such that
$$
\begin{tikzcd}
W(k)[\![x,y]\!] \arrow[r, hookrightarrow]\arrow[d, hookrightarrow] & \Ainf \arrow[d, hookrightarrow]  & \\
A_{\max}^{(2)} \arrow[r, hookrightarrow] & \A_{\max}  \arrow[r, hookrightarrow] & \BdR^+
\end{tikzcd}
$$
commutes. Furthermore, $\Fil ^ i \BdR^+ \cap A^{(2)}_{\max}= \Fil ^i A^{(2)}_{\max}$. The same result holds when $A^{(2)}_{\max}$ is replaced by $A^{(2)}_{\st , \max}$.  
\end{proposition}
\begin{proof} In the following, we only treat the case of $A^{(2)}_{\st , \max}$ while the proof of $A^{(2)}_{ \max}$ is the same by noting that $z= u w$ in $A_{\inf}$. 

The uniqueness is apparent. To show the existence of the embedding, it is enough to show $\gamma_i(w)\in\Amax$ for all $i\geq 1$. 

It is a well-known fact that $\Amax$ is isomorphic to the $p$-adic completion of $\Ainf[\frac{u^e}{p}]$, and $\Amax[1/p]$ is a Banach $\Q_p$-algebra, which is the completion of $\Ainf[1/p]$ under the norm $\lvert \cdot \rvert_{p^{-1}}$ such that
$$
\lvert x \rvert_{p^{-1}} = \sup_{n} \{p^{-n}\lvert x_n \rvert_{\O_{C}^\flat}\}
$$
where $x=\sum_{n\gg 0} [x_n]p^n \in \Ainf[1/p]$. And we have for $x\in \Amax[1/p]$, $x\in \Amax$ if and only if $\lvert x \rvert_{p^{-1}} \leq 1$. Moreover $\lvert \cdot \rvert_{p^{-1}}$ is multiplicative. So now it is enough to show for $x=\gamma_i(w)$ considered as an element inside $ \Amax[1/p]$, we have $\lvert x^{p-1} \rvert_{p^{-1}}\leq 1$. To show this, we have by \cite[Proposition 3.17]{BMS1}, $\xi:=\mu/\varphi^{-1}(\mu)$ is a generator of $\Ker\theta$ with $\mu = \epsilon-1$. In particular, $w=\mu/E = a  \varphi^{-1} (\mu) \in \Ainf$ with $a\in \Ainf^\times$. 
And we can check $\overline{w}^{p-1} = c \overline{u}^e$ inside $\O_C^\flat=\Ainf/p\Ainf$, with $c$ a unit. So $w^{p-1}=au^e+bp$ with $a,b\in \Ainf$, and
$$
x^{p-1}=\frac{(au^e+bp)^i}{(i!)^{p-1}}.
$$
Using the fact $v_p(i!)< \frac{i}{p-1}$, one can show each term in the binomial expansion on the right-hand side of the equation has $\lvert \cdot \rvert_{p^{-1}}$-norm less or equal to $1$, so in particular, $\lvert x^{p-1} \rvert_{p^{-1}}\leq 1$.

To prove that $\Fil ^i \BdR^+ \cap  A ^{(2)}_{\st, \max} = \Fil ^i A ^{(2)}_{\st , \max}$, it suffices to show $E \BdR ^+ \cap A ^{(2)}_{\st, \max}[\frac 1 p] = E A ^{(2)}_{\st , \max}[\frac 1 p]$. By  Proposition  \ref{prop-key-property},  we reduces to prove that the map $$\theta : D_w= A ^{(2)}_{\st , \max}[\frac 1 p ]/E \to \BdR^+ / E= \C_p$$ is injective. Let $f(w) = \sum_{i \geq 0} a_i \gamma_i (w)\in \Ker \theta$ with $a_i \in \O_K$ limits to $0$ $p$-adically. Then $f(w_0) = 0$ with $w_0: = \theta (w)= \theta (\frac{\epsilon-1}{E}) \in \C_p$. Note  $v_p (w_0)\geq \frac{1}{p-1}$ since  $ \frac{\epsilon-1}{\varphi^{-1}(\epsilon) -1}$ is another generator of kernel $\theta: A_{\inf}  \to \O_{\C_p}$, and we have  
$v_p (w_0) = v_p (\theta (\varphi^{-1} (\epsilon)-1))= \frac{1}{p-1}$. 
Since we aim to show that $f= 0$, without loss of generality, we can assume that $K$ contains $p_1= {p}^\frac{1}{p-1}$. Note that 
$v_p (i !)\leq \frac{1}{p-1}$, we conclude that $\frac{w_0}{p_1}$ is a root of $f(p_1 w)$ which is in $\O_K\langle w\rangle$. By Weierstrass preparation theorem, $w_0$ is algebraic over $K$ unless $f=0$.  By lemma below,  $w_0: = \theta (w) \in \C_p$ is transcendental over $K$ and hence $f= 0$. 
\end{proof}
\begin{lemma}\label{lem-transcendental}
$w_0 = \theta (\frac{\epsilon-1}{E})$ is transcendental over $K$. 
\end{lemma}
\begin{proof}
If $w_0$ is contained in an algebraic extension $L$ over $K$, we define $L_{0,\infty}=\bigcup_n L(\varpi_n)$. For $g\in G_{L_{0,\infty}}$, we will have 
$$
\theta(g(\frac{\epsilon-1}{E}))=g(w_0)=w_0=\theta(\frac{\epsilon-1}{E}).
$$
Since $G_{L_{0,\infty}}$ fix $E$,  $\theta(\frac{g(\epsilon-1)-(\epsilon-1)}{E})=0$. This implies $g(\epsilon-1)-(\epsilon-1)\in \Fil^2\BdR^+$. Recall for $t=\log \epsilon$, $t-(\epsilon-1)\in \Fil^2\BdR^+$, so we have $g(t)-t \in \Fil^2\BdR^+$. But this can't be true. Since $L_{0,\infty}$ can only contain finitely many $p^n$-th roots of $1$, for $g\in G_{L_{0,\infty}}$, $g(t)=c(g)t$ satisfying $c(g) \in \Q_p$ and $c(g)\neq 1$. This implies $g(t)-t = (c(g)-1)t \in \Fil^1\BdR^+ \setminus \Fil^2\BdR^+$.
\end{proof}

\begin{corollary}\label{cor-inj}
The natural maps $\iota _ A : A ^{(2)} \to A_{\inf}$ and $\iota_{\st} : A^{(2)}_{\st} \to A_{\inf}$ are injective. 
\end{corollary}

To summarize, we have the following commutative diagram of rings inside $\BdR^+$:
$$
\begin{tikzcd}
\At  \arrow[d, hookrightarrow] \arrow[r, hookrightarrow] & \At_{\st}\arrow[r, hookrightarrow] \arrow[d, hookrightarrow]  & \Ainf \arrow[d, hookrightarrow] \\
\At_{\max} \arrow[r, hookrightarrow] & \At_{\st, \max} \arrow[r, hookrightarrow] & \Amax.  
\end{tikzcd}
$$

\section{Application to semi-stable Galois representations}\label{sec-prismaticphiGhatmodules}
In this section, we assume that $R= \O_K$. We explain how to use the period ring $A^{(2)}$ and $A^{(2)}_{\st}$ to understand lattices in crystalline and semi-stable representations. Roughly speaking, we are going to use $A^{(2)}$ and $A^{(2)}_{\st}$ to replace $\wh{\mathcal R}$ in the theory of $(\varphi , \hat G)$-modules developed in \cite{liu-notelattice}. 

Let $K _\infty =\bigcup_{n = 1}^\infty K (\varpi _n )$, $G_\infty: = {\textrm Gal } (\overline  K / K_\infty)$ and $G_K: = {\textrm Gal } (\overline K  / K)$.  Recall that $A= \mathfrak S = W(k)[\![u]\!]$. Let $S $  be the $p$-adic completion of $ W(k) [\![u , \frac{E^i}{i !}, i \geq 1]\!]$, which is the PD envelope of $W(k)[u]$ for the ideal $(E)$. It is clear that $S\subset \Omax$. We define $\varphi$ and $\Fil ^i$ on $S$ induced from those on $\Omax$, in particular, $\Fil ^i S = S \cap E ^i \Omax[\frac 1 p]$. Note that $A$ embeds to $\Ainf$ via $u \mapsto [\varpi^\flat ]$ that is not stable under $G_K$-action but only on $G_\infty$-action. For any $g \in G_K$, define $\ueps (g)= \frac{g(u)}{u}$. It is clear that $\ueps (g) = \epsilon ^{a(g)}$ with $a(g) \in \Z_p$. We define \emph{two}  differential operators $N_S$ and $\nabla_S$ on $S$ by $N_S(f) = \frac{d f}{du}u$ and $\nabla_S (f) = \frac{ df }{du}$. We need $\nabla_S$ to treat crystalline representations.  

\subsection{Kisin module attached to semi-stable representation}\label{subsec-Kisin-st} Fix $h \geq 0$, 
a \emph{Kisin module of height $h$} is a finite free $A$-module $\MM $ with a semi-linear endomorphism $\varphi_{\MM}: \MM \to \MM$ so that $\coker (1 \otimes \varphi_\MM)$ is killed by $E^h$, where $1 \otimes \varphi_{\MM} : \MM ^* : = A \otimes _{\varphi, A}\MM \to \MM $ is linearization of  $\varphi_\MM$. Note here we are using the classical setting of Kisin modules used in \cite{liu-notelattice} but it is good enough for this paper. The following summarizes the results on Kisin modules attached to $G_K$-stable $\Z_p$-lattices in semi-stable representations. The details and proofs of these facts can be found in \cite{liu-notelattice}. 

Let $T$ be a $G_K$-stable $\Z_p$-lattice inside a semi-stable representation $V$ of $G_K$ with Hodge-Tate weights in $\{0, \dots , h\}$. We set 
$$D: = D^*_{\st} (V)= \Hom_{\Q_p, G_K} (V , B_{\st})$$ which is the filtered $(\varphi , N)$-module attached to $V$ and $D_K : = K \otimes_{K_0} D$. Then there exists a unique Kisin module $\MM : = \MM (T) $ of height $h$ attached to $T$ so that 
\begin{enumerate}
    \item $\Hom_{\varphi , A} (\MM , \Ainf)\simeq T|_{G_\infty}$. 
    \item There exists an $S$-linear isomorphism  
    $$\iota_S : S [\frac 1 p] \otimes _{\varphi, A}\MM \simeq D \otimes _{W(k)} S $$ so that $\iota_S$ is compatible with $\varphi$ on the both sides. 
\item $\iota_S$ also induces an isomorphism $\Fil^h (S [\frac 1 p] \otimes _{\varphi, A}\MM) \simeq \Fil ^h (D\otimes_{W(k)} S) $. The filtration on both sides are defined as follows: 
\[\Fil^h (S [\frac 1 p] \otimes _{\varphi, A}\MM): =\left  \{ x \in S [\frac 1 p] \otimes _{\varphi, A}\MM| (1\otimes \varphi_\MM (x)) \in \Fil ^h S[\frac 1 p ] \otimes_A \MM \right \}.  \] 
To define 
filtration on $\calD : = S \otimes _ {W(k)} D$, we first extend the monodromy operator $N_D=N$ on $D$ to $N_\calD$ (resp. $\nabla_\calD$) on $\calD$ by letting $N_{\calD}= 1 \otimes N_D + N_S \otimes 1$ (resp. $\nabla_\calD = 1 \otimes N_D + \nabla_S \otimes 1$). Then we define $\Fil ^i \calD$ by induction: set $\Fil^0 \calD = \calD$ and 
\[ \Fil ^i\calD: = \{x \in \calD|  N_{\calD}(x) \in \Fil^{i-1}\calD, f_\varpi (x) \in \Fil ^i D_K\}\]
where $f_\varpi : \calD \to D_K$ is induced by $S\to \O_K$ via $u \mapsto \varpi. $
\end{enumerate}
\begin{remark}[Griffith transversality]\label{rem-GT} We have $N_{\calD} (\Fil  ^i\calD) \subset \Fil^{i -1}\calD$ from the construction of $\Fil ^i \calD$. This property is called Griffith transversality. 

We only use $\nabla_\calD$ when $N_D = 0$, that is, when $V$ is crystalline. In this case, it is clear that $N_\calD = u \nabla_{\calD}$. So it is clear that $\nabla_\calD (\Fil ^i\calD) \subset \Fil ^{i-1}\calD$. 
\end{remark}
For ease of notation, we will write $N = N_\calD$ and $\nabla = \nabla_\calD$ in the following. 
Let $T^\vee: = \Hom_{\Z_p} (T , \Z_p)$ and $V ^\vee : = T^\vee \otimes_{\Z_p}\Q_p$ denote the dual representations.  
Then there exists an $A_{\inf}$-linear injection 
\begin{equation}\label{eqn-iota-A}
\iota_\MM: A_{\inf} \otimes _A \MM \to  T ^\vee \otimes_{\Z_p} \Ainf, 
\end{equation}
which is compatible with $G_\infty$-actions ($G_\infty$ acts on $\MM$ trivially) and $\varphi$ on both sides. Applying $S \otimes_{\varphi, A} -$ and using $\iota _S: =S \otimes_{\varphi, A} \iota_\MM $, we obtain the following commutative diagram 
$$
\xymatrix@C=5em{ \Acris[\frac 1 p] \otimes_{\varphi , A}\MM \ar[d]_\wr  ^{\Acris \otimes_S \iota _S} \ar[r] ^{S \otimes_{\varphi, A} \iota _\MM} & V^\vee \otimes_{\Z_p} \Acris \ar@{=}[d]\\   \Acris \otimes_{W(k)} D \ar[r]^{\alpha} & V^\vee \otimes _{\Z_p} \Acris}
$$
where the second row  $\alpha$ is built by the classical comparison $$B_{\st} \otimes_{K_0} D^*_{\st}(V) \simeq   V^\vee \otimes_{\Q_p} B_{\st},   $$
and  $\alpha$ is $G_K$-stable on the both sides. The left side of $\alpha$ is defined by 
\[ \forall x \in D, \forall g \in G_K, g (x) = \sum_{i = 0}^\infty N^i (x) \gamma_i (\log (\ueps(g))) \]
Therefore, if we regard $\MM^* : = A \otimes_{\varphi, A} \MM$ as an  $A$-submodule of  
$ V^\vee \otimes_{\Z_p} \Acris$ via injection $\iota^* : = S \otimes_{\varphi, A} \iota_A$, one can show that: 
\begin{equation}\label{eqn-G-action}
\forall g \in G_K, x \in \MM^*, g(x) = \sum_{i = 0}^\infty N_\calD^i (x) \gamma_i (\log (\ueps(g))).  
\end{equation}
 When $V$ is crystalline, or equivalently, $N_D = 0$, we have (\cite[\S8.1]{LL2021comparison})
 \begin{equation}\label{eqn-G-action-2}
\forall g \in G_K, x \in \MM^*, g(x) = \sum_{i = 0}^\infty \nabla_\calD^i (x) \gamma_i (u \ueps(g)).  
\end{equation}

\subsection{Descent of the \texorpdfstring{$G_K$}{GK}-action}\label{subsec-G-image}

Let us first discuss the  $G_K$-action on $\MM \subset T ^\vee \otimes_{\Z_p}\Ainf$ via $\iota_\MM$ in \eqref{eqn-iota-A} in more details.  
We select an $A$-basis $e_1, \dots , e_d$ of $\MM$ so that $\varphi (e_1, \dots, e_d)= (e_1, \dots , e_d )\gA$ with $\gA \in {\textrm M}_d (A)$.  Then there exists a matrix $B\in {\textrm M}_d (A)$ so that $\gA B = B \gA = E^h I_d$. For any $g \in G_K, $ let $X_g$ be the matrix so that 
\[ g (e_1, \dots , e_d) = (e_1, \dots , e_d) X_g. \]
In this section, we are interested in where the entries of $X_g$ locates. 
\begin{theorem}\label{Thm-1}The entries of $X_g$ are in $A^{(2)}_{\st}$. If $V$ is crystalline and $g(u)- u = Ez$ then $X_g \in {\textrm M}_{d} (\At). $
\end{theorem}

First,  it is well-known that $ W (\C_p^\flat) \otimes_{\Ainf} \iota_\MM $ is an isomorphism. So $X_g\in {\textrm M}_d (W (\C_p^\flat))$. Since $G_K$-actions and $\varphi$ commutes, we have $$ \gA \varphi(X_g)= X_g g (\gA)  .$$  
 Define $$\Fil ^h \MM ^* : = \{ x \in \MM^* | (1 \otimes \varphi_{\MM}) (x) \in E^h \MM\}. $$
Since $\MM$ has $E$-height $h$, it is easy to show that $\Fil ^h \MM^*$ is a finite free $A$-module and $\Fil^h \calD$ is generated by $\Fil ^h \MM^*$. 

To be more precise, let $\{e^*_i : =1 \otimes e_i, i =1 , \dots , d\}$ be an $A$-basis of $\MM^*$. It is easy to check that $(\alpha_1, \dots , \alpha_d)= (e_1^* , \dots , e_d^*) B$ is an $A$-basis of $\Fil ^h \MM^*$,  and it is  also an $S[\frac 1 p]$-basis of $\Fil ^h \calD$. 
So for any $g \in G_K$, we have $g (\alpha_j) = \sum\limits_{i = 0}^\infty N ^i (\alpha_j) \gamma _i (\log (\ueps (g))) $. By Griffith transversality in Remark \ref{rem-GT}: 
$N (\Fil ^i \calD)\subset \Fil^{i-1}\calD,   $
 we have, 
 \begin{equation}\label{eqn-action-g}
 g (\alpha_j)= \sum_{i = 0}^h N ^i (\alpha_j) E^i \gamma _i (\frac {\log (\ueps (g))}{E}) + \sum_{i > h}^\infty N^i (\alpha_j) \gamma_i(E) (\frac{\log (\ueps (g))}{E})^i.      
 \end{equation}
Since $N^i (\alpha_j) E^i \in \Fil ^h \calD$,  $\gamma_i (E)$ in $\Omax$ and $\{w ^n\}$ converges to $0$ inside $A^{(2)}_{\st, \max}$,  we see that $g (\alpha_1, \dots , \alpha_d) = (\alpha_1 , \dots , \alpha_d) Y_g $ with $Y_g \in A^{(2)}_{\st , \max}[\frac 1 p ]. $

In the case that $V$ is crystalline, using \eqref{eqn-G-action-2}, we have 
$$g (\alpha_j)= \sum_{i = 0}^h \nabla ^i (\alpha_j) E^i \gamma _i (\frac {u\ueps (g)}{E}) + \sum_{i >  h}^\infty \nabla^i (\alpha_j) \gamma_i(E) (\frac{u\ueps (g)}{E})^i $$
\emph{If $g $ is chosen so that  $ g (u)-u = Ez$} then, a similar argument can shows that $g (\alpha_1, \dots , \alpha_d) = (\alpha_1 , \dots , \alpha_d) Y^\nabla_g $ with $Y^\nabla_g \in A^{(2)}_{\max}[\frac 1 p]. $

Now $g (e_1^*, \dots , e_d^*) = (e_1^*, \dots, e_d ^*) \varphi (X_g)$.  Using similar arguments, we see that $\varphi(X_g)$'s entry are in $A^{(2)}_{\st , \max}[\frac 1 p ]$ and $A^{(2)}_{\max}[\frac 1 p]$ respectively. We have $(\alpha_1 , \dots , \alpha_d) = (e_1^*, \dots, e_d^*)B$ by definition of $B$,, we conclude that $$ \varphi(X_g) g (B) = B Y_g.$$ 
Using the formula that $\gA \varphi (X_g) = X_g g (\gA)$ and $\gA B = B \gA = E^h I_d$, we conclude that 
$Y_g = (\frac{g (E)}{E}) ^h X_g$. Write  $r= \frac{g (E)}{E}$. We claim that $r$ is a unit in $\At_{\st}$. Indeed, 
$\frac{g(E)}{E}= \frac{E (u \epsilon ^{a(g)})}{E(u)}= \sum\limits_{i =0}^e E^{(i)} (u) \frac{ u ^i (\epsilon^{a(g)}-1)^i}{E i !}$ is again inside $A_{\st}^{(2)}$, where $E^{(i)}$ means the $i$-th derivative of $E$. And it is easy to show $g(E)$ is also a distinguished element $A_{\st}^{(2)}$, so by \cite[Lemma 2.24]{BS19}, $r$ is a unit. Similarly, when $g(u)-u=Ez$, we will have $r=\frac{g (E)}{E}\in (A^{(2)})^\times$. Hence  
\begin{equation}\label{eqn-key-eqn}
E^h X_g  = r ^{-h} \gA  \varphi (X_g) g (B).
\end{equation}

Now we can apply Proposition \ref{prop-desecnt} and Proposition \ref{prop-Ast-properties} (5) to the above formula, we conclude that for $g\in G_K$ (resp. $g\in G_K$ such that $g(u)-u=Ez$ and $V$ is crystalline), we have $X_g$ has entries in $A^{(2)}_{\st}[\frac 1 p]$ (resp. $A^{(2)}[\frac 1 p]$). 

To complete the proof of Theorem \ref{Thm-1}, it suffices to show that  entries $X_g$ are in $A^{(2)}_{\st}$ (resp. $A^{(2)}$). Unfortunately, the proof to remove ``$\frac 1 p$" is much harder, which needs \S \ref{subsec-phi-tau} and \S \ref{subsec-pris-crystal-proof}. For the remaining of this subsection, we only show that the proof of Theorem \ref{Thm-1} can be reduced to the case that $g = \tilde \tau$ for a special selected $\tilde \tau \in G_K$.  

Let $L = \bigcup\limits_{n =1}^\infty K_{\infty} (\zeta_{p ^n})$, $K_{1^\infty}: = \bigcup_{n =1}^\infty K (\zeta_{p ^n})$,  $\hat G : = \Gal(L /K) $ and $H_K : = \Gal (L / K _\infty)$. 
If $p > 2$ then it is known that $\hat G \simeq \Gal (L/ K_{1 ^\infty}) \rtimes H_K $ with 
$\Gal (L/ K_{1 ^\infty}) \simeq \Z_p$. Let $\tau$ be a topological generator of $\Gal (L/ K_{1 ^\infty}) $. We have $\tau (u) = \epsilon^a u$ with $a\in \Z_p ^\times$. Without loss of generality, we may assume that $\tau (u) = \epsilon u$. If $p=2$ then we can still select $\tau \in \hat G $ so that $\tau (u)= \epsilon u$ and $\tau, H_K$ topologically generate $\hat G$. Pick $\tilde \tau \in G_K$ a lift of $\tau$. Clearly, we have $\tilde \tau (u ) - u = E z$. 

\begin{proposition}\label{thm-1prime} 
For $g = \tilde \tau, $ the entries of $X_g$ are in $A^{(2)}_{\st}$, and if further  $V$ is crystalline, then $X_g \in {\textrm M}_{d} (\At).$
\end{proposition}
\begin{lemma}\label{lem-equivalenceofthm}
Proposition ~\ref{thm-1prime} is equivalent to Theorem~\ref{Thm-1}.
\end{lemma}
\begin{proof}
Since  $\hat{G}$ is topologically generated by $\tau$ and $H_K$. So $G_K$ is topologically generated by $G_\infty$ and $\tilde{\tau}$. And we have $\tau(u)-u=(\epsilon-1)u=Ez$. Now if $X_{\tilde{\tau}}$ has coefficient in $A^{(2)}_{\st}$ and $X_g=I_d$ for all $g\in G_\infty$  then to show that $X_g \in \At_{\st}$ for all $g \in G_K$, it suffices to show that $X_{\tilde \tau^{p ^n }}$ converges to $I_d$ inside ${\text M}_d (\At_{\st})$. Since $\At_{\st}$ is closed in $\At_{\st, \max}$ by Proposition \ref{prop-Ast-properties} (5), it suffices to show that $X_{\tilde \tau^{p ^n}}$ converges inside $\At_{\st, \max}$. Since $X_g= (\frac{E}{g (E)})^r Y_g$ and $Y_g$ is defined by \eqref{eqn-action-g}, we easily check that $X_{\tilde \tau^{p ^n}}$ converges to $I_d$ in $\At_{\st,\max}$ by using that $\ueps(\tilde \tau^{p ^n})$ converges to $0$ in $(p , \epsilon-1)$-topology. The proof for the crystalline case is similar by replacing $\At _{\st}$ with $\At$.
\end{proof}
So it remains to prove Proposition ~\ref{thm-1prime} to complete the proof of  Theorem~\ref{Thm-1}. We will prove Proposition~\ref{thm-1prime} in \S\ref{subsec-pris-crystal-proof}. Briefly speaking, for $g = \tilde \tau$, we have shown that the linearization of the $g$-action defines a $\varphi$-equivariant isomorphism:
$$
f_g: \MM\otimes_{A,\iota_g} A_{\st}^{(2)}[\frac{1}{p}] \simeq \MM\otimes_{A} A_{\st}^{(2)}[\frac{1}{p}]
$$
of $A_{\st}^{(2)}[\frac{1}{p}]$-modules. Since $g(u)- u = Ez$ and $V$ is crystalline, $f_g$ defines a $\varphi$-equivariant isomorphism:
$$
f_g: \MM\otimes_{A,\iota_g} A^{(2)}[\frac{1}{p}] \simeq \MM\otimes_{A} A^{(2)}[\frac{1}{p}]
$$
of $A^{(2)}[\frac{1}{p}]$-modules. Here $\iota_g: A \to A^{(2)}_{\st}$ (resp. $\iota_g: A \to A^{(2)})$) is defined by $u \to g(u)$. On the other hand, by \cite[Theorem 5.6]{wu2021galois}, we will see the $g$-action on $T^\vee \otimes W(\C_p^\flat)$ also descent to a $\varphi$-equivariant morphism $c$ of $B^{(2)}$-modules, and recall that $B^{(2)}$ the is $p$-adic completion of $\At [\frac 1 E]$. Then by comparing $c$ and $f_g$ using the technique developed in \S\ref{subsec-phi-tau}, we will deduce Proposition~\ref{thm-1prime} from Lemma~\ref{lem-intersection}.
 
\begin{remark}\label{rem-inputofWu} Our original strategy to prove Theorem \ref{Thm-1} is to show $\At_{\st} [\frac 1 p] \cap W ({\C^\flat_p}) = \At_{\st}$ (resp. $\At [\frac 1 p] \cap W (\O^\flat_{\C_p}) = \At$).  This is equivalent to that $ \At/ p , \At_{\st}/ p$ injects in $\C_p^\flat$. Unfortunately, it does not work out though we can show  $ \wt \At/ p , \wt {\At_{\st}}/ p$ injects in $\C_p^\flat$.

\end{remark}

\subsection{Prismatic \texorpdfstring{$(\varphi, \hat G)$}{(phi,Ghat)}-modules}\label{subsec-phiGhatmodules} In this subsection, we show that the base ring $\wh{\calR}$ for the theory of $(\varphi, \hat G)$-modules can be replaced by $\At_{\st}$. To this end, this builds a new theory of $(\varphi, \hat G)$-modules  with new base ring $\At_{\st}$. Since the idea of this new theory is almost the same as that of the old one, We will use \emph{classical} to indicate we are using the theory over $\wh \calR$. For example, when we say classical $(\varphi, \hat G)$-module, it means a $(\varphi , \hat G)$-module over $\wh \calR$. 
Recall $L = \bigcup\limits_{n =1}^\infty K_{\infty} (\zeta_{p ^n})$, $\hat G : = \Gal(L /K) $ and $H_K : = \Gal (L / K _\infty)$. Let $\m $  be the maximal ideal of $\O_{\C_p}^\flat$ and set $I_+ = W(\m)$ so that $\Ainf/ I _+ = W(\bar k)$. For any subring $B\subset \Ainf$ set $ I_+ B = B \cap I_+$. Let $t = \log \epsilon$,  $t ^{(i)} = t ^{r(i)} \gamma_{\tilde q(i)}(\frac{t ^{p-1}}{p})$ where $ i = (p-1) \tilde q (i) + r(i)$ with $ 0 \leq r(i)<p-1. $
 Recall that $\wh \calR : = \Ainf \cap \calR_{K_0 }$ where 
$$\calR_{K_0}:=\left  \{ \sum_{i = 0}^\infty f_i t ^{(i)}, f_i\in S[\frac 1 p], f_i \to 0\  p{\text{-adically}}  \right \}.$$

\begin{lemma} \label{lem-properties-Ast}
\begin{enumerate}
    \item As a subring of $\Ainf$, $\At_{\st}$ is stable under $G_K$-action and the $G_K$-action factors through $\hat G$. 
\item $\At_{\st}/ I_+\At_{\st} = W(k)$. 
\item $I_+ \At \subset u \At_{\st}$. 
\item $\varphi (\At_{\st} ) \subset \wh \calR$. 
\end{enumerate}
\end{lemma}
\begin{proof}(1) It is clear that the $G_K$-action is stable on $W(k)[\![u , \epsilon -1]\!]$. Since $\At_{\st}$ is $(p , E)$-completion of $W(k)[\![u , \epsilon -1]\!] [\delta^i (w), i \geq 0]$, to show that $\At_{\st}$ is $G_K$-stable, it suffices to show that $g (w)\in \At_{\st}$ (because $g$ and $\delta$ commutes, if $g(x)\in \At _{\st}$ then so is $g (\delta (x))$). Now $E w = \epsilon-1$, we have $g (E) g (w) = g (\epsilon -1)= \epsilon ^{a(g)} -1 $. Then $g (w) = \frac{E}{g(E)} \frac {\epsilon^{a(g)} -1}{E}$. By \cite[Lemma 2.24]{BS19}, $E/g(E)$ is a unit in $\At_{\st}$, then $g (w)\in \At_{\st}$. 

(2) It is clear that both $u , \epsilon-1 $ are in $I_+$. Hence $w \in I_+$ because $E w = \epsilon-1 \in I_+$ and $E \equiv p \mod I_+$. For any $x=\sum\limits_{i = 0}^\infty  p ^i[x_i] \in \Ainf$, $x \in I_+$ if and only of $x_i \in \m$. Then it is easy to check that $\delta (I_+)\subset I_+$, and consequently  all $\delta^i (w) \in I_+$. So $I_+ \At_{\st}$ is topologically generated by $u, y=\epsilon-1 , \delta ^i (w), i \geq 0$ and hence $\At_{\st}/ I_+\At_{\st} = W(k) $ as required. 

(3) $I_+ \At$ is topologically generated by $u, v=\epsilon u, \{\delta ^i (z)\}, i \geq 0$. And (3) follows from $z= u w$ and $\delta^n (z)= u^{p ^n} \delta^n (w)$.  

(4) Since $\At_{\st} \subset \At_{\st , \max}$, it suffices to show that 
$\varphi (\At_{\st , \max}) \subset \calR_{K_0}$. Since $\varphi(\Omax) \subset A[\![\frac{E^p }{p}]\!]\subset S$, it suffices to show that $\varphi (\gamma_n (w) ) \in \calR_{K_0}$. Note that $\varphi (E) = p \nu$ with $\nu \in A[\![\frac{E^p }{p}]\!] ^\times$ and $\gamma _i (\epsilon -1) \in \calR_{K_0}$. And we have
$$
\varphi(w)=\varphi (\frac{(\epsilon -1)}{E})= \nu^{-1} (\epsilon-1) \sum_{i = 1}^p \left ( \binom{p}{i} / p \right ) (\epsilon -1)^{i -1} 
$$
which is a polynomial with coefficients in $\Z$ and in variables $\nu^{-1}$ and $\gamma_i(\epsilon -1)$'s. In particular $\varphi (\gamma_n (w) ) \in \calR_{K_0}$ by basic properties of divided powers. 
\end{proof}

\begin{definition} A (finite free) $(\varphi, \hat G)$-module of height $h$ is a (finite free) Kisin module $(\MM , \varphi_{\MM})$ of height $h$ together with an $\At_{\st}$-semi-linear $\hat G$-action on $\wh \MM : = \At_{\st} \otimes _A \MM$ so that 
\begin{enumerate}
    \item The actions of $\varphi$ and $\hat G$ on $\wh \MM$ commutes; 
    \item $\MM \subset \wh \MM ^{H_K}$; 
    \item $\hat G$-acts on $\wh \MM / I_+ \At_{\st}$ trivially. 
\end{enumerate}
\end{definition}

The category of $(\varphi, \hat G)$-modules consists of the above objects and a morphism between two $(\varphi, \hat G)$-modules is a morphism of Kisn modules that commutes with actions of $\hat G$. Given a $(\varphi, \hat G)$-modules $\wh \MM : = (\MM , \varphi, \hat G)$, we define a $\Z_p$-representation of $G_K$, 
$$\wh T (\wh \MM) : = \Hom_{\At_{\st}, \varphi} (\At_{\st} \otimes _A \MM, \Ainf).$$

Since $\varphi (\At_{\st})\subset \wh \calR$, given a $(\varphi, \hat G)$-module $\wh \MM : = (\MM, \varphi, \hat G)$-defined as the above, $(\MM, \varphi)$ together $\hat G$-action on $\wh\calR \otimes_{\varphi, A}\MM$ is a \emph{classical} $(\varphi, \hat G)$-modules $\wh \MM_c$. It is easy to check that $\wh T (\wh \MM ) = \wh T(\wh \MM_c): =\Hom_{\wh \calR, \varphi} (\wh \calR \otimes _{\varphi , A} \MM, \Ainf). $

\begin{theorem}\label{thm-2}
The functor $\wh T$ from the category of $(\varphi, \hat G)$-modules of height $h$ to the category of $G_K$-stable $\Z_p$-lattices in semi-stable representations with Hodge-Tate weights in $[0 , \dots , h]$ is an anti-equivalence. 
\end{theorem}
\begin{proof}
Given an $(\varphi, \hat G)$-module $\wh \MM= (\MM, \varphi, \hat G)$, $\wh\MM_c $ is a classical $(\varphi, \hat G)$-module. So $\wh T (\wh \MM)= \wh T (\wh \MM_c)$ is a lattice inside semi-stable representation with Hodge-Tate weights in $[0, \dots , h]$. Conversely, given a lattice in semi-stable representation $T$ with Hodge-Tate weights in $[0, \dots , h]$, following the proof for the existence of classical $(\varphi, \hat G)$-module $\wh \MM $ so that $\wh T (\wh \MM ) = T$, it suffices to show that for any $g \in G_K$, 
$g (\MM) \subset \At_{\st} \otimes _A \MM$, here $\MM $ and $ \At_{\st} \otimes _A \MM$ are regarded as submodules of $T^\vee \otimes _{\Z_p} \Ainf$ via $\iota_\MM$ in \eqref{eqn-iota-A} and uses the $G_K$-action on $T^\vee \otimes _{\Z_p} \Ainf$. This follows Theorem \ref{Thm-1}. 
\end{proof}

Now let us discuss when $\wh T (\wh \MM)$ becomes a crystalline representation. Recall that $\tau$ is a selected topological generator of $\Gal (L/ K_{1 ^\infty}) $, and we have $\tau (u) = \epsilon u$ and $\tau, H_K$ topologically generate $\hat G$.

\begin{corollary}\label{cor-crystalline} Select $\tau \in \hat G$ as the above. Then $\wh T (\wh \MM )$ is crystalline if and only if $\tau (\MM) \subset \At \otimes _{A} \MM$. 
\end{corollary}
\begin{proof} Clearly for the selected $\tau$, we have $\tau (u) - u = Ez$. If $T: = \wh T (\wh \MM)$ is crystalline then Theorem \ref{Thm-1} proves that $\tau (\MM) \subset \At \otimes _{A} \MM$. Conversely,  Suppose $\tau (\MM) \subset \At \otimes _{A} \MM$. 
Then we see that $(\tau-1)\MM \subset u \Ainf \otimes_A \MM$ by Lemma \ref{lem-properties-Ast} (3). And we have this is enough to show that $\wh T (\wh \MM)$ is crystalline. For example, We will have $\MM\otimes_{A}(\Ainf[\frac{1}{p}]/\mathfrak{p})$ has a $G_K$-fixed basis given by a basis of $\MM$, where the ideal $\mathfrak{p}$ is defined as $\mathfrak{p}:=\cup_{n\in \N} \varphi^{-n}(u)\Ainf[\frac{1}{p}] \subset \Ainf[\frac{1}{p}]$. Then one can prove by the same method in \cite[Thm. 3.8]{Ozeki} or directly use \cite[Theorem 4.2.1]{Duthesis} that $T$ is crystalline.  
\end{proof}
\begin{remark}\label{rem-Ast-is-good} Though $\At_{\st}$ is still complicated, for example, it is not Noetherian, $\At_{\st}$ is still better than old $\wh \calR$: at least it has explicit topological generators. Furthermore, $\At_{\st}$ is $p$-adic complete. This can help to close the gap in \cite{liu-Fontaine} mentioned in \cite[Appendix B]{gao2021breuilkisin}. Indeed, as indicated by Remark B.0.5 \emph{loc.cit.}, if $\wh \calR$ can be shown to be $p$-adic complete then the gap in \cite{liu-Fontaine} can be closed. So by replacing $\wh \calR$ by $\At_{\st}$, we close the gap of \cite{liu-Fontaine} (\cite{gao2021breuilkisin} provides another similar way to close the gap). 
\end{remark}

\section{Crystalline representations and prismatic \texorpdfstring{$F$}{F}-crystals}\label{sec-prismaticFcrystals}
In this section, we reprove the theorem of Bhatt and Scholze on the equivalence of prismatic $F$-crystal and lattices in crystalline representations of $G_K$ and complete the proof of Theorem \ref{Thm-1}. We start with some general facts on the absolute prismatic site (which allows general base rings). 

\subsection{Prismatic \texorpdfstring{$F$}{F}-crystals in finite projective modules}\label{subsec-pris-crystal}
Let $R=R_0\otimes_W \O_K=R_0[u]/E$ as in the beginning of \S2 and $X=\Spf(R)$ with the $p$-adic topology.

\begin{definition}
The (absolute) prismatic site $X_\Prism$ of $X$ is the opposite of the category of bounded prisms $(A, I)$ that are $(p,I)$-completed together with a map $R \to A/I$, and a morphism of prisms $(A, I) \to (B, J)$ is a covering if and only if $A \to B$ is $(p, I)$-completely faithfully flat. 

Define the following functors: 
$$
\O_\Prism: (A,I)\mapsto A,
$$ 
and for all $h\in \N$, let
$$
\II_\Prism^h: (A,I)\mapsto I^h.
$$
It is known in \cite{BS19} that these are sheaves on $X_\Prism$, and $\O_\Prism$ admits an endomorphism $\varphi_\Prism$ coming from the $\delta$-structure. We will also use $\O_\Prism[1/\II_{\Prism}]^\wedge_p$ to denote functor assign $(A,I)$ to the $p$-adic completion of $A$ with $I$ inverted.
\end{definition} 

Now we verify $\At $(resp.  $A^{(3)})$  constructed in \S\ref{subsrc-construct-A2} is indeed self (resp. triple) product of $A$ in $X_\Prism$. We mainly discuss the situation of $\At$ while the proof of $A^{(3)}$ is almost the same. Recall that $\breve A = \breve R_0 [\![u]\!]= W\langle t_1, \dots , t_m \rangle[\![u]\!]$. 

First, we want to remark on the existence of nonempty self-coproduct in the category of prisms over $R$. We thank Peter Scholze for answering our question on Mathoverflow. And we will repeat his answer here. Let $(A_i,I_i)$ for $i=1,2$ that are prisms over $R$, let $A_0=A_1\hat{\otimes}_{\Z_p} A_2$ where the completion is taken for the $(p,I_1,I_2)$-adic topology. Let $J$ be the kernel of the map:
$$
A_0 \to A_1/I_1\otimes_{R} A_2/I_2.
$$
Let $(A,I)$ be the prismatic envelope of $(A_1,I_1)\to (A_0,J)$, one can  check this is the initial object in the category of prisms over $R$ that admits maps from $(A_i,I_i)$ such that the two $R \to A_i/I_i \to A/I$ agree. Also we want to note that in general, we do not know if the boundedness of $(A_1,I_1)$ and $(A_2,I_2)$ will imply the boundedness of their coproduct. But we have seen $A^{(2)}$ and $A^{(3)}$ are indeed bounded by Corollary \ref{cor-filtration-shape}.
 
To start, note that there exists a $W$-linear map $\breve i_2: \breve A \to A^{\ho 2}$ induced by $u \mapsto y$ and $t_i \mapsto s_i$. We claim that $\breve i_2$ uniquely extends to $ i_2 : A \to A^{\ho 2}$ which is compatible with $\delta$-structures. Indeed, consider the following commutative diagram
$$\xymatrix@C=45pt{ A \ar[r] ^-{\overline i_2}\ar@{-->}[rd]^{i_{2, n}} & A^{\ho 2} / (p , J^{(2)} )  \\ {\breve A}  \ar[u]\ar[r]^-{\breve i_{2, n} } & A^{\ho 2} / (p , J ^{(2)})^n\ar[u]}  $$
Here $\breve i_{2, n}= \breve i_2 \mod (p , J ^{(2)})^n$ and $\overline{i}_2:A \to A/(p, E)  \simeq A^{\ho 2}/ (p , J ^{(2)}) $. Since $\breve i_2 (u) = y = x+ (y -x)$ and $\breve i_2 (t_i) = s_i = t _i + (s_i - t_i)$, we see that the above (outer) diagram commutes. Since $A$ is formally \'etale over $\breve A $ by $(p, u)$-adic topology, we conclude that there exists a unique map $i_{2, n} : A \to A^{\ho 2} / (p , J^{(2)})^n$ so that the above diagram commutes. Since $A ^{\ho 2}$ is $(p, J ^{(2)}  )$-complete,  there uniquely exists  $i _2 : A \to A^{\ho 2}$ which extends $\breve i_2$. To see $i_2$ is compatible with $\delta$-structures. it suffices to show that $\varphi \circ i _2 = i_2 \circ \varphi$. But both of $\varphi \circ i _2$  and $ i_2 \circ \varphi$ extend $ \breve A \overset \varphi \to \breve A \to A^{\ho 2}$.  Again by formally \'etaleness of $A$ over $\breve A$, we see that $\varphi \circ i _2 = i_2 \circ \varphi$. Hence we obtain a map $ 1 \otimes i_2: A \otimes _{\Z_p }A \to A^{\ho 2}$. Define $\theta^{\otimes 2}: A\otimes_{\Z_p} A \to R$ via $\theta^{\otimes 2} (a \otimes b)= \theta (a) \theta (b)$. By the construction of $i _2$, we have the following commutative diagram
\[ \xymatrix{ A \otimes _{\Z_p} A \ar[r] ^{1 \otimes i _2} \ar[d]^{\theta ^{\otimes 2 }} & A ^{\ho 2}\ar[d]  \\ R \ar[r]^- \sim & A^{\ho 2}/ J ^{(2)}  }\]
Let $\widehat {A ^{\otimes 2}} $ be the $(p , \ker (\theta ^{\otimes 2}))$-completion of $A^{\otimes 2}: = A \otimes_{\Z_p} A$. 
Hence $1 \otimes i_2$ induces a map $\hat i_2$ from the $\widehat {A ^{\otimes 2}}$ to $ A ^{\ho 2}$ because $A ^{\ho 2}$ is clearly $(p , J ^{(2)})$-complete. To treat $A^{\ho 3}$, we construct $ i _3: A \to A ^{\ho 3}$ by extending $\breve i _3: A \to A^{\ho 3}$ by sending $u \mapsto w $ and $t _j \mapsto r_j$. The same method shows that $i_3$ is compatible with $\delta$-structure and we obtain a map $1 \otimes i _2 \otimes i_3 : A^{\otimes 3} \to A^{\ho 3}$ with $A ^{\otimes 3}: A \otimes_{\Z_p} A \otimes_{\Z_p}A$. Similarly, we obtain a natural map $\hat i _3 : \widehat {A ^{\otimes 3}} \to A ^{\ho 3} $. 
\begin{lemma} For $s= 2, 3$, $\hat i_s : \widehat {A ^{\otimes s}} \to A ^{\ho s}$ are isomorphisms. 
\end{lemma}
\begin{proof}
We need to construct an inverse of $\hat i_s$. We only show for $\hat i _2$ and the proof for $\hat i _3$ is the same. 
Let $g: A^{\ho 2} \to \widehat {A ^{\otimes 2}}$ be the $A$-linear map by sending $y -x \mapsto 1 \otimes u - u \otimes 1 $ and $s_j - t _j  \mapsto  1 \otimes t_j  - t_j \otimes 1$.  Clearly $g$ is well-defined because $ 1 \otimes u - u \otimes 1$ and $1 \otimes t_j  - t_j \otimes 1$ are in $\Ker (\theta ^{\otimes 2})$. Since $i_2 (u) = y$ and $i_2 (t_j) = s_j$,  $\hat i _2 \circ g $ is identity on $A ^{\ho 2}$. Now it suffices to show that $h : = g \circ \hat i_2 $ is identity. Write $K = (p , \Ker (\theta ^{\otimes 2}))$. Note that we have map $ A \otimes_{\Z_p} \breve A \to  \widehat {A ^{\otimes 2}} \overset h \to   \widehat {A ^{\otimes 2}}$ induced by $h  $ which we still call it $\breve h $ for simplicity. 
Now  we  have the following commutative diagram
$$
\begin{tikzcd}[column sep=huge]
A\otimes_{\Z_p} A \arrow[r, "\mod K"] \arrow[dr, dashed, shift right=0.5ex, "h_n", swap] \arrow[dr, dashed, shift left=0.5ex, "\pi_n"] & (A\otimes_{\Z_p} A ) / K \\
A \otimes_{\Z_p}{\breve A} \arrow[u] \arrow[r,"\breve h \mod K^n", swap] & (A \otimes_{\Z_p} A )/ K ^n \arrow[u]
\end{tikzcd}
$$
where $h_n$ is induced by $h \mod K^n$, and $\pi_n$ is the modulo $K^n$ map.
We see that both $ h_n $ and $\pi_n$ on the dashed arrows can make the diagram commute. Then by the formal \'etaleness of $A$ over $\breve A$, we conclude that $h_n  = \pi_n$ and $h$ is the identity map. 
\end{proof}
\begin{proposition}\label{prop-selfproduct} $\At$ and $A^{(3)}$ is self-product  and triple product of  $A$ in $X_\Prism$. 
\end{proposition}
\begin{proof} 
In the following, we only treat the case of $\At$ while the proof for $A^{(3)}$ is the same. 
We need to prove that  for any $B = (B,J)$ in $X_\Prism$, 
$$
\Hom_{X_\Prism^{\opp}}(A^{(2)},B)=\Hom_{X_\Prism^{\opp}}(A, B ) \times \Hom_{X_\Prism^{\opp}}(A,B).
$$
By the above lemma, we have natural maps $A \otimes_{\Z_p} A \to \wh{A^{\otimes 2}} \simeq A^{\ho 2}$. Combined with natural map $A^{\ho 2}\to \At$ as $\At$ is the prismatic envelope of $ A^{\ho 2}$ for the ideal $J^{(2)}$, we have map $\alpha : A \otimes_{\Z_p} A \to \At$ which is compatible with $\delta$-structures. Then $\alpha$ induces map $$
\beta: \Hom_{X_\Prism^{\opp}}(A^{(2)},B)\to \Hom_{X_\Prism^{\opp}}(A, B ) \times \Hom_{X_\Prism^{\opp}}(A,B).
$$
To prove the surjectivity of $\beta$, given $f_i \in  \Hom_{X_\Prism}(A, B )$ for $i = 1,2$, we obtain a map  $f _1 \otimes f _2: A \otimes_{\Z_p} A \to B$. It is clear that $(f_1 \otimes f_2) (\Ker (\theta ^{\otimes 2})) \subset J$. Since $B$ is $(p, J)$-derived complete, $f \otimes f_2$ extends to a map $ f _1 \ho f_2 : \wh{A ^{\otimes 2}}\simeq A ^{\ho 2} \to B$ which is compatible with $\delta$-structures, Hence $f_1 \ho f _2$ is a morphism of $\delta$-algebra. Finally, by the universal properties of prismatic envelope, $f _1 \ho f_2 $ extends to   a map of prisms $ f_1 \ho_{\Prism} f_2: \At \to B$ as required. 

Finally, we need to show that $\beta$ is injective. It suffices to show that $A$-algebra structure map $i _1 : A\to \At $ and $i'_2: A \overset{i_2}{\to } A^{\ho 2} \to \At$ both are  injective. 
Since all rings here are $(p, E)$-complete integral domains, it suffices to check that $i_1 , i_2' \mod (p, E)$ are injective. By Proposition \ref{prop-key-property}, we see that $i_1 \mod (p, E)$ is $R/ pR \to R/pR [\{\gamma _i (z_j)\}] $, so it is injective. By the construction $i'_2$ and $i_2$, we see that $i'_2 \mod (p, E)$ is the same as $A/(p, E) \to A ^{\ho 2}/ (p , J ^{(2)}) \to \At /(p, E)$, which is same as $R/ pR \to R/pR [\{\gamma _i (z_j)\}] $.  So it is injective. 
\end{proof}

\begin{remark}\label{rem-Astprelog}
    When $R=\O_K$ is a complete DVR with perfect residue field $k$, we know a priori, the self-product $A^{(2)}$ of $(A,(E))$ in $X_\Prism$ can be constructed as the prismatic envelope of $(A,(E))\to (B,I)$, where $B$ is the $(p,E(u),E(v))$-adic completion of $W(k)[\![u]\!] \otimes_{\Z_p} W(k)[\![v]\!]$ and $I$ is the kernel of the map:
    $$
    B \to A/(E)\otimes_{R} A/(E)=R.
    $$
    On the other hand,  $W(k)$ is formally \'etale over $\Z_p$ for the $p$-adic topology, so for all $(C,J)\in X_\Prism$, the map $W(k)\to R \to C/J$ lifts uniquely to a map $W(k) \to C$. In particular, for all $(C,J)\in X_\Prism$, $C$ has a natural $W(k)$-algebra structure. So when we construct the self-product, we can also consider $A^{(2)}$ as the prismatic envelope of $(A,(E))\to (C,J)$, where $C$ is the $(p,E(u),E(v))$-adic completion of $A\otimes_{W(k)} A$ and $J$ is the kernel of the map:
    $$
    C \to A/(E)\otimes_{R} A/(E)=R.
    $$
    We  have $C\simeq W(k)[\![u,v]\!]$, $J=(E(u),u-v)$ and $A^{(2)}=W(k)[\![u,v]\!]\{\frac{u-v}{E}\}^\wedge_\delta$.
\end{remark}

\begin{definition}\label{def-Fcrystal}
\begin{enumerate}
    \item 
A prismatic crystal over $X_\Prism$ in finite locally free $\O_\Prism$-modules (resp. $\O_\Prism[1/I]^\wedge_p$-modules) is a finite locally free $\O_\Prism$-module (resp. $\O_\Prism[1/I]^\wedge_p$-module) $\MM_{\Prism}$ such that for all morphisms $f: (A, I) \to (B, J)$ of prisms, it induces an isomorphism:
$$
f^\ast \MM_{\Prism,A} := \MM_{\Prism}((A,I))\otimes_A B \simeq \MM_{\Prism,B} := \MM_{\Prism}((B,J))
$$
$$
(resp.\quad f^\ast \MM_{\Prism,A} := \MM_{\Prism}((A,I))\otimes_{A[1/I]^\wedge_p} B[1/I]^\wedge_p \simeq \MM_{\Prism,B} := \MM_{\Prism}((B,J))).
$$

\item A prismatic $F$-crystal over $X_\Prism$ of height $h$ in finite locally free $\O_\Prism$-modules is a prismatic crystal $\MM_{\Prism}$ in finite locally free $\O_\Prism$-modules together with a $\varphi_{\Prism}$-semilinear endomorphism $\varphi_{\MM_{\Prism}}$ of the $\O_\Prism$-module $\MM_{\Prism}: \MM_{\Prism} \to \MM_{\Prism}$ such that the cokernel of the linearization $\varphi_\Prism^\ast \MM_{\Prism} \to \MM_{\Prism}$ is killed by $\II^h$.
\end{enumerate}

\end{definition}

\begin{proposition}\label{prop-cover-final-object}
If the sheaf represented by $(B,I)$ in $\Shv(X_\Prism)$ covers the final object $\ast$ in $\Shv(X_\Prism)$, i.e., for any $(C,J)$ in $X_\Prism$, there is a $(P, J)$ lies over $(B,I)$ and covers $(C,J)$. Also assume that the self-coproduct $B^{(2)}$ and self-triple-coproduct $B ^{(3)}$ of $(B,I)$ are inside $X_{\Prism}$, i.e., they are bounded. Then a prismatic crystal $\MM_{\Prism}$ over $X$ in finite locally free $\O_\Prism$-modules is the same as a finite projective module $\MM$ over $B$ together with a descent data $\psi: \MM\otimes_{i_1,B} B^{(2)}\simeq \MM\otimes_{i_2,B} B^{(2)}$ satisfies the cocycle condition. Here  $i _j : B \to B^{(2)}$ $(j=1,2)$ are the two natural maps.
\end{proposition}

\begin{proof}
Let $\MM$ be a prismatic crystal in finite projective modules. Define $\MM= \MM_{\Prism}((B,I))$, and the descent data comes from the crystal property:
$$
\psi:\MM\otimes_{i_1,B} B^{(2)}\simeq  \MM_{\Prism}((B^{(2)},I)) \simeq \MM\otimes_{i_2,B} B^{(2)}.
$$
Now  given $(\MM, \psi)$, then for any $(C,J)$ in $X_\Prism$, we need to construct a finite projective module over $C$. We choose the $(P, J)$ as in the assumption, let $\MM_P=\MM  \otimes_B P$, and consider the following diagram:
$$
\begin{tikzcd}
C \arrow[r] & P \arrow[rr,"f_1"] & & P^{(2)}_C \\
 & B \arrow[u] \arrow[r,"i_1"] & B^{(2)} \arrow[ur,"f"] & \\
 & & B \arrow[u,"i_2"] \arrow[r] & P \arrow[uu,"f_2"] \\
 & & & C \arrow[u]
\end{tikzcd}
$$
Here $(P^{(2)}_C,J)$ is the self-coproduct of $(P,J)$ in the category of prisms over $(C,J)$, and the existence of $(P^{(2)}_C,J)$ is from \cite[Corollary 3.12]{BS19}, where they also show that $P^{(2)}_C$ is the derived $(p, J)$-completion of $P\otimes^\mathbb L_C P$ and $(P^{(2)}_C,J)$ is bounded. As a bounded prism over $(C,J)$, $(P^{(2)}_C,J)$ is naturally inside $X_\Prism$, so $f$ exists by the universal property of $B^{(2)}$. So if we take the base change of $\psi$ along $f$, we get 
$$
f^\ast\psi: (\MM\otimes_{i_1,B} B^{(2)})\otimes_{B^{(2)},f}P^{(2)}_C \simeq (\MM\otimes_{i_2,B} B^{(2)})\otimes_{B^{(2)},f}P^{(2)}_C
$$
which is the same as an isomorphism:
$$
\psi_C: \MM_{P}\otimes_{P,f_1}P^{(2)}_C \simeq \MM_{P}\otimes_{P,f_2}P^{(2)}_C.
$$
Similar arguments will show $\psi_C$ satisfies the cocycle condition. And $\MM_{P}$ descents to a finite projective module over $C$ by \cite[Proposition A.3]{ALB}.
\end{proof}

\begin{remark}
We want to note that the structures of finite nonempty coproducts in the category of bounded prisms over a prism $(A,I)$ is much simpler compared with the structure of finite nonempty products in the category $(R/A)_{\Prism}$ (cf. \cite[Lecture V, Corollary 5.2]{BhaNotes18}).
\end{remark}

\begin{lemma}\label{lem-AEcoversfinal}
The prism $(A,(E))$ defined in \S\ref{subsrc-construct-A2} covers the final object $\ast$ in $\Shv(X_\Prism)$ in the sense of Proposition~\ref{prop-cover-final-object}. And $A^{(2)}$ and $A^{(3)}$ are bounded.
\end{lemma}
\begin{proof}
The proof is similar to \cite[Lemma 5.14]{ALB}, we need to show for $R$ defined as in \S\ref{subsrc-construct-A2}, there exists a quasi-syntomic perfectoid cover of $R$. We will construct this perfectoid cover similar to \cite[\S 7.1]{Kim12}.

First recall we have $R=\O_K\otimes_{W}R_0$, and we fix a compatible system $\{\varpi_n\}_{n\geq 0}$ of $p^n$-th roots of a uniformizer $\varpi_0$ of $\O_K$ inside $E$. Let $\wh K_\infty$ be the $p$-adic completion of $\cup_n K(\varpi_n)$, we know $\wh K_\infty$ is perfectoid. Use $\overline{R}_0[\![u]\!]$ to denote $A/(p)=R/(\varpi)=R_0/(p)[\![u]\!]$, and let $\overline{R}_0[\![u]\!]_{\mathrm{perf}}^\wedge$ to be the $u$-adic completion of the direct perfection of $\overline{R}_0[\![u]\!]$, it can be checked directly that $(\overline{R}_0[\![u]\!]_{\mathrm{perf}}^\wedge[1/u],\overline{R}_0[\![u]\!]_{\mathrm{perf}}^\wedge)$ is a perfectoid affinoid $\wh K_\infty^\flat$-algebra, by tilt equivalence, there is a corresponded perfectoid affinoid $\wh K_\infty$-algebra. More explicitly, let $\tilde{R}_\infty = W(\overline{R}_0[\![u]\!]_{\mathrm{perf}}^\wedge)\otimes_{W(\O_{\wh K_\infty}^\flat),\theta} \O_{\wh K_\infty}$. Then  $\tilde{R}_\infty$ is naturally an  $R$-algebra, and we claim it is a quasi-syntomic cover of $R$.

To show this, by \cite[\S 7.1.2]{Kim12}, we have
$$
\tilde{R}_\infty = (R_0\wh {\otimes}_{W}\O_{\wh K_\infty})\wh {\otimes}_{\Z_p} \Z_p\langle T_i^{-p^\infty}\rangle
$$
where $T_i \in R_0$ is any lift of a $p$-basis of $R_0/(p)$. We have $\O_K\to \O_{\wh K_\infty}$ is a quasi-syntomic cover so by (2) of \cite[Lemma 4.16]{BMS2}, $R \to R_0\wh {\otimes}_{W}\O_{\wh K_\infty}$ is also a quasi-syntomic cover. And we have $S=\Z_p\langle T_i^{-p^\infty}\rangle$ is a quasi-syntomic ring, this can be seen by constructing  a perfectoid quasi-syntomic covering of it, so by Lemma 4.34 of $loc.cit.$, we have the complex $\mathbb{L}_{S/\Z_p} \in D(S)$ has $p$-complete Tor amplitude in $[-1,0]$. In particular, $\Z_p \to \Z_p\langle T_i^{-p^\infty}\rangle$ is also a quasi-syntomic cover, so applying (1) in Lemma 4.16 of $loc. cit.$,  $R \to \tilde{R}_\infty$ is a quasi-syntomic perfectoid cover.

The  boundedness of $\At$ and $A ^{(3)}$ is from (2) in Corollary~\ref{cor-filtration-shape}.
\end{proof}

\begin{corollary}\label{cor-crystal-descentdata}
Assume the the base $X=\Spf(R)$ satisfies the condition in \S\ref{sec-ring-strcuture}, and let $A$, $A^{(2)}$ and $A^{(3)}$ be  defined as in \S\ref{subsrc-construct-A2}, then a prismatic $F$-crystal $(\MM_{\Prism}, \varphi_{\MM_{\Prism}})$ in finite locally free $\O_\Prism$-modules of height $h$ over $X$ is the same as a Kisin module $(\MM,\varphi_\MM)$ of height $h$ over $A$ with a descent datum
$$
f: \MM \otimes_{A,i_1} A^{(2)} \simeq \MM \otimes_{A,i_2} A^{(2)}
$$
that compatible with the $\varphi$-structure and satisfies the cocycle condition over $A^{(3)}$.
\end{corollary}

\begin{theorem}(\cite[Theorem 1.2]{BS2021Fcrystals})\label{Thm-main-1}  
Let $T$ be a crystalline representation of $G_K$ over a $\Z_p$-lattice of Hodge-Tate weights in $[0,h]$, then there is a prismatic $F$-crystal $\MM_{\Prism}(T)$ over $X_\Prism$ of height $h$ over $X$ such that $\MM_{\Prism}((A,E))$ is the Kisin module associated to $T$. Moreover, the association of $T\mapsto \MM_{\Prism}(T)$ induces an equivalence of the above two categories. 
\end{theorem}
We will prove this theorem in \S \ref{subsec-pris-crystal-proof}. 

\begin{remark}
Theorem~\ref{Thm-main-1} was first established by Bhatt-Scholze in \cite[Theorem 1.2]{BS2021Fcrystals}. The harder direction of \cite[Theorem 1.2]{BS2021Fcrystals} is to show for all $\Z_p$-lattices inside crystalline representations of $G_K$, one can attach a prismatic $F$-crystal. Using the theory of $(\varphi,\hat{G})$-modules, we have shown in \S\ref{subsec-G-image}, given a crystalline representation of $G_K$ over a $\Z_p$-lattices $T$, we can attach a Kisin module $\MM$ and a descent data\footnote{Strictly speaking, \S\ref{subsec-G-image} only constructs an isomorphism but have not checked that  it satisfies cocycle condition, which will be proved in \S \ref{subsec-pris-crystal-proof}.} 
$$
f_{\tilde{\tau}}: \MM\otimes_{A,i_1} A^{(2)}[\frac{1}{p}] \simeq \MM\otimes_{A,i_2} A^{(2)}[\frac{1}{p}]
$$
comes from the $\tau$-action. We just show this is a $\varphi$-equivariant isomorphism, and we need to show it gives rise to a descent data over $A^{(2)}$. As we have mentioned in Remark~\ref{rem-inputofWu}, we can not find a direct ring theoretic proof of this. Our idea is to use result of \cite{wu2021galois} or \cite[Corollary 3.7]{BS2021Fcrystals}: the underlying Galois representation $T$ gives a descent data over $A^{(2)}[\frac{1}{E}]^\wedge_p$. To finish our proof, we need to compare this descent data with $f_{\tilde{\tau}}$ over $A^{(2)}[\frac{1}{E}]^\wedge_p[\frac{1}{p}]$. This leads us to develop a ``prismatic" $(\varphi,\tau)$-module theory in the next subsection, where we will have Lemma~\ref{lem-evaluation-1} and Lemma~\ref{lem-evaluation-2} to help us compare descent data over $A^{(2)}[\frac{1}{E}]^\wedge_p$ and $A^{(2)}[\frac{1}{E}]^\wedge_p[\frac{1}{p}]$ via an evaluation map to $W(\O_{\hat{L}}^\flat)$.
\end{remark}

\subsection{Prismatic \texorpdfstring{$(\varphi,\tau)$}{(phi,tau)}-module theory}\label{subsec-phi-tau} In this subsection, we make some preparations to prove Proposition \ref{thm-1prime} and Theorem \ref{Thm-main-1}. So we restrict to the case that $R=\O_K$ is a complete DVR with perfect residue field. 
\begin{definition}
An \'etale $\varphi$-module over $A[1/E]^\wedge_p$ is a pair $(\M, \varphi_\M)$ such that $\M$ is a finite free module over $A[1/E]^\wedge_p$, and $\varphi_\M$ is an isomorphism
$$
\varphi_\M: \varphi^\ast \M: = A[1/E]^\wedge_p\otimes_{\varphi , A[1/E]^\wedge_p} \M  \simeq \M
$$
of $A[1/E]^\wedge_p$-modules. And we define an \'etale $\varphi$-module over $A[1/E]^\wedge_p[1/p]$ to be a $\varphi$-module over $A[1/E]^\wedge_p[1/p]$ such that it is obtained from an \'etale $\varphi$-module over $A[1/E]^\wedge_p$ by base change.

An \'etale $\varphi$-module over $A[1/E]^\wedge_p$ (resp. $A[1/E]^\wedge_p[1/p]$) with descent data is a triple $(\M, \varphi_\M, c)$, with $(\M, \varphi_\M)$ an \'etale $\varphi$-module over $A[1/E]^\wedge_p$ (resp. $A[1/E]^\wedge_p[1/p]$), and $c$ an isomorphism
$$
c: \M \otimes_{A[1/E]^\wedge_p,i_1} B^{(2)} \simeq \M \otimes_{A[1/E]^\wedge_p,i_2} B^{(2)}
$$
$$
(\text{resp. }c: \M \otimes_{A[1/E]^\wedge_p[1/p],i_1} B^{(2)}[1/p] \simeq \M \otimes_{A[1/E]^\wedge_p[1/p],i_2} B^{(2)}[1/p])
$$
which is compatible with the $\varphi$-structure and satisfies the cocycle condition over $B^{(3)}$ (resp. $B^{(3)}[\frac 1 p]$). Here for $j=1,2$, $i_j: A[1/E]^\wedge_p \to B^{(2)}$ is the map induced from $i_j: (A,(E)) \to (A^{(2)},(E))$. 
\end{definition}

\begin{remark}\label{rmk-Wuandevaluation}
It is the main result in \cite{wu2021galois} and \cite[\S2]{BS2021Fcrystals} that there is an equivalence of the category of lattices in representations of $G_K$ and the category of prismatic $F$-crystals in finite locally free $\O_\Prism[1/I]^\wedge_p$-modules over $\O_K$. Also by \cite[Proposition 2.7]{BS2021Fcrystals}, one can show prismatic $F$-crystals in finite locally free $\O_\Prism[1/I]^\wedge_p$-modules is the same as \'etale $\varphi$-modules over $A[1/E]^\wedge_p$ with descent data. 
\end{remark}

The aim of this subsection is to use the ideas in \cite{wu2021galois} and \cite[\S 5.5]{KedlayaLiu-relativeII} show that \'etale $\varphi$-modules over $A[1/E]^\wedge_p$ (resp. $A[1/E]^\wedge_p[1/p]$) with descent data are equivalence to $\RepZp(G_K)$ (resp. $\RepQp(G_K)$). More importantly, for all $\gamma \in \hat{G}$, we will construct an evaluation at $\gamma$ map
$$
e_\gamma: B^{(2)} \to W(\hat{L}^\flat)
$$
and use it to study $\varphi$-equivariant morphisms between finite free $B^{(2)}$ and $B^{(2)}[1/p]$-modules. We will see the evaluation at $\tau$ map will play a crucial role in our proof of Proposition~\ref{thm-1prime} and the Theorem~\ref{Thm-main-1} below.

Recall in \S\ref{subsec-phiGhatmodules}, we define $L = \bigcup\limits_{n =1}^\infty K_{\infty} (\zeta_{p ^n})$, $\hat G : = \Gal(L /K) $ and $H_K : = \Gal (L / K _\infty)$. Moreover, we define $\wh K_{1^\infty}$ to be the $p$-adic completion of $\cup_{n\geq 0} K(\zeta_{p^n})$, and we let $\hat{L}$ to be the $p$-adic completion of $L$. It is clear that $A[1/E]_p^\wedge\subset W(\hat L ^\flat)^{H_K}$.
Recall the following definition and theorem in \cite{Caruso-phitau}:

\begin{theorem}\label{thm-caruso}
An \'etale $(\varphi,\tau)$-module is a triple $(\M, \varphi_{\M}, \hat{G})$ where
\begin{itemize}
    \item $(\M, \varphi_{\M})$ is an \'etale $\varphi$-module over $A[1/E]^\wedge_p$;
    \item $\hat{G}$ is a continuous $W(\hat{L}^\flat)$-semi-linear $\hat{G}$-action on $$\hat{\M}:=W(\hat{L}^\flat)\otimes_{A[1/E]^\wedge_p}\M$$ such that $\hat{G}$ commutes with $\varphi_{\M}$;
    \item regarding $\M$ as an $A[1/E]^\wedge_p$-submodule of $\hat{\M}$, we have $\M\subset \hat{\M}^{H_K}$.
\end{itemize} 
Then there is an anti-equivalence of the category of \'etale $(\varphi,\tau)$-modules and $\RepZp(G_K)$, such that if $T$ corresponds to $(\M, \varphi_{\M}, \hat{G})$, then
$$
T^\vee = (\hat{\M}\otimes_{W(\hat{L}^\flat)}W(\C_p^\flat))^{\varphi=1}.
$$
\end{theorem}

One of the basic facts used in the theory of \'etale $(\varphi,\tau)$-modules developed in \cite{Caruso-phitau} is that $\Gal(\hat{L}/\wh K_{1^\infty})\simeq \Z_p$, and we write $\tau$ to be a topological generator of $\Gal(\hat{L}/K_{1^\infty})$ determined by $\tau(\varpi_n)=\zeta_{p^n}\varpi_n$ as the discussion before Corollary \ref{cor-crystalline}. Also  $\hat{G}$ is topologically generated by $\tau$ and $H_K$, so in particular,  the $\hat{G}$-action on $\hat{\M}$ is determined by the action of $\tau$ on $\M$ inside $\hat{\M}$. As discussed before, we will provide a direct correspondence of the category of \'etale $(\varphi,\tau)$-modules and the category of \'etale $\varphi$-modules over $A[1/E]^\wedge_p$ with descent data. Moreover, we will construct an evaluation at $\tau$ map:
$$
e_\tau: B^{(2)} \to W(\hat{L}^\flat),
$$
and show that the $\tau$-action on $\M$ inside $\hat{\M}$ is given by the base change of the descent data along $e_\tau$. 

\begin{remark}
In \cite[Theorem 5.2]{wu2021galois}, they prove a similar equivalence but for \'etale $(\varphi,\Gamma)$-modules. The theory of \'etale $(\varphi,\Gamma)$-module is defined for the cyclotomic tower $K_{1^\infty}$ over $K$ while the theory of \'etale $(\varphi,\tau)$-modules is defined using the Kummer tower $K_{\infty}$. We will use a lot of ideas and results developed in \cite{wu2021galois} when proving our claims in this subsection. The main difficulty in our situation is that the Kummer tower $K_\infty$ is not a Galois tower over $K$. To deal with this, we have to use the idea in \cite[\S 5.5]{KedlayaLiu-relativeII}. Roughly speaking, we will take the Galois closure $L$ of $K_\infty$,  then prove results over ${L}$, then descent back to $K_\infty$ using $K_\infty={L}^{H_K}$. 

One should be able to construct the evaluation map in the content of \cite{wu2021galois} the same way as we define in this subsection. This map will give a more direct correspondence of the descent data and the $\Gamma$-actions on \'etale $(\varphi,\Gamma)$-modules.
\end{remark}

By \cite[Lem 3.9]{BS19}, any prism $(B , J)$ admits a map into its perfection $(B_{\perf}, JB_{\perf})$. The following theorem (\cite[Thm 3.10]{BS19}) is the key to understanding perfect prisms.  
\begin{theorem}\label{thm-perfectprismandperfectoidring}
$(A,I)\to A/I$ induces an equivalence of the category of perfect prisms over $\O_K$ with the category of integral perfectoid rings over $\O_K$.
\end{theorem}
 
Let $(A,(E))$ be the Breuil--Kisin prism defined in \S\ref{subsrc-construct-A2}, we have
\begin{lemma}\label{lem-perfectionofA} $A_{\perf}\simeq W(\O_{\wh K_\infty}^\flat)$.
\end{lemma}
\begin{proof}
The same as the proof of \cite[Lemma 2.17]{wu2021galois}
\end{proof}

\begin{lemma}
Let $\Perfd_K$ be the category of perfectoid $K$-algebras, then $\Perfd_{K}$ admits finite non-empty coproducts. 
\end{lemma}
\begin{proof}
Let $R$ and $S$ be two perfectoid $K$-algebras, it follows from \cite[Corollary 3.6.18]{KedlayaLiu-relativeI} that the uniform completion $(R\otimes_K S)^u$ of the tensor product $(R\otimes_K S)$ is again a perfectoid $K$-algebra, and it is easy to show this is the coproduct of $R$ and $S$ in the category of perfectoid $K$-algebras.
\end{proof}

For $i\in \N_{>0}$, let $(A^{(i)},(E))$ (resp. $(\Ainf(\O_{\hat{L}})^{(i)},(E))$) denote the $i$-th self-coproduct of $(A,(E))$ (resp. $(\Ainf(\O_{\hat{L}}),(E))$) in the category of prisms over $\O_K$, where $\Ainf(\O_{\hat{L}}):=W(\O_{\hat{L}}^\flat)$. The following is a description of $(A^{(i)})_{\perf}[1/E]^\wedge_p$ and $(\Ainf(\O_{\hat{L}})^{(i)})_{\perf}[1/E]^\wedge_p$. 

\begin{lemma}\label{lem-Aiperf}
Let $\wh K_\infty^{(i)}$ (resp. $\hat{L}^{(i)}$) be the $i$-th self-coproduct of $\wh K_\infty$ (resp. $\hat{L}$) in $\Perfd_K$, then $(A^{(i)})_{\perf}[1/E]^\wedge_p \simeq W((\wh K_\infty^{(i)})^\flat)$ and $(\Ainf(\O_{\hat{L}})^{(i)})_{\perf}[1/E]^\wedge_p \simeq W((\hat{L}^{(i)})^\flat)$.
\end{lemma}
\begin{proof}
We will only prove the lemma for $(A^{(i)})_{\perf}[1/E]^\wedge_p$, the case for $\hat{L}^{(i)}$ is similar.

We use similar arguments as in \cite[Lemma 5.3]{wu2021galois}. Fix $i$, first we can show $(A^{(i)})_{\perf}$ is the $i$-th self-coproduct of $(A_{\perf}, (E))$ in the category of perfect prisms over $\O_K$, i.e. $(A^{(i)})_{\perf}=(A_{\perf})^{(i)}_{\perf}$. By Theorem~\ref{thm-perfectprismandperfectoidring}, Lemma~\ref{lem-perfectionofA} and \cite[Proposition 2.15]{wu2021galois}, if we let $S=(A^{(i)})_{\perf}/E$, then $S[1/p]$ is the $i$-th self-coproduct of $\wh K_\infty$ in the category of perfectoid $K$-algebras. Now we have 
$$
(A^{(i)})_{\perf}[1/E]^\wedge_p\simeq W(S^\flat)[1/[\varpi^\flat ]]^\wedge_p=W(S^\flat[1/\varpi^\flat ])\simeq W((\wh K_\infty^{(i)})^\flat).
$$
\end{proof}

\begin{remark}\label{rem-diamonds}
There is another way to view $\wh K_\infty^{(i)}$ in terms of diamonds over $\Spd(K,\O_K)$ which is used in the proof of \cite[Lemma 5.3]{wu2021galois}, that there exist a ring of integral elements $\wh K_\infty^{(i),+}$ in $\wh K_\infty^{(i)}$, such that we have 
\begin{equation}\label{eq-diamondKi}
    \Spa(\wh K_\infty^{(i)},\wh K_\infty^{(i),+})^\diamond \simeq \underbrace{\Spa(\wh K_\infty,\wh K_\infty^{+})^\diamond \times_{\Spd(K,\O_K)} \ldots \times\Spa(\wh K_\infty,\wh K_\infty^{+})^\diamond}_{i\text{-copies of } \Spa(K_\infty,K_\infty^{+})^\diamond}.
\end{equation}
And similar results hold for $\hat{L}$. Using this description and the fact that functor from perfectoid spaces over $\Spa(K, \O_K)$ to diamonds over $\Spd(K, \O_K)$ is an equivalence, we have $\hat{L}^{(i)}$ has a natural action of $\hat{G}^i$ coming from the action on the diamond spectrum. Since $\hat{L}^{H_K}=\wh K_\infty$, we have
\begin{IEEEeqnarray*}{+rCl+x*}
\Spa(\wh K_\infty^{(i)},\wh K_\infty^{(i),+})^\diamond &\simeq& \left (\Spa(\hat{L},\O_{\hat{L}})^\diamond \times \ldots \times_{\Spd(K,\O_K)}\Spa(\hat{L},\O_{\hat{L}})^\diamond \right )^{H_K^i}\\ 
&\simeq& (\Spa(\hat{L}^{(i)},\hat{L}^{(i),+})^\diamond  )^{H_K^i}.
\end{IEEEeqnarray*}
That is, $(\hat{L}^{(i)})^{H_K^i}=\wh K_\infty^{(i)}$.
\end{remark}

Now we use ideas in \cite{wu2021galois} and \cite[\S 5.5]{KedlayaLiu-relativeII} to study \'etale $\varphi$-modules over $A[1/E]^\wedge_p$ with descent data. We will show this category is the same as generalized $(\varphi,\Gamma)$-modules in the work of Kedlaya-Liu. The following is a quick review of \cite[Example 5.5.6 and 5.5.7]{KedlayaLiu-relativeII}.

Firstly, one has $\hat{L}^{(i)}\simeq \Cont(\hat{G}^{i-1}, \hat{L})$, here $\Cont$ means the set of continuous functions. One can see this fact from the proof of \cite[Theorem 5.6]{wu2021galois}. When $i=2$, we choose the two canonical maps $i_1,i_2:\hat{L} \to \hat{L}^{(2)}$, corresponds to $j_1,j_2: \hat{L} \to \Cont(\hat{G}, \hat{L})$ given by 
\begin{equation}\label{eq-j1j2}
   j_1(x): \gamma \mapsto \gamma (x) \quad \text{ and } \quad j_2(x): \gamma \mapsto x.
\end{equation}

From Remark~\ref{rem-diamonds}, there is a natural action of $\hat{G}^2$ on $\hat{L}^{(2)}$. One can check this corresponding to the $\hat{G}^2$-action on $\Cont(\hat{G},\hat{L})$ given by:
$$
(\sigma_1,\sigma_2)(f)(\gamma)=\sigma_2 f(\sigma_2^{-1}\gamma\sigma_1).
$$

\begin{remark}
We interchange the roles of $j_1$ and $j_2$ comparing with the isomorphism defined in \cite[Example 5.5.6]{KedlayaLiu-relativeII}, so the $\hat{G}^2$-action is different from that in Example 5.5.7 of $loc. cit.$, we will see this definition is more convenient when relating the descent data with the semilinear group actions. 
\end{remark}

One can show $\Cont(\hat{G},-)$ commutes with tilting and the Witt vector functor, as been discussed in \cite[Lemma 5.3]{wu2021galois}, so in particular, we have 
$$
W((\hat{L}^{(i)})^\flat) \simeq \Cont(\hat{G}^{i-1}, W(\hat{L}^\flat)).
$$
For $i=2$, we still use $j_1$ and $j_2$ to represent the two canonical maps from $W(\hat{L}^\flat)$ to $\Cont(\hat{G}, W(\hat{L}^\flat))$ that comes from \eqref{eq-j1j2}. The above isomorphism also is compatible with the action of $\hat{G}^2$, so we have
\begin{equation}\label{eq-K(2)andhatGaction}
W((\wh K_\infty^{(2)})^\flat) \simeq \Cont(\hat{G}, W(\hat{L}^\flat))^{H_K^2}.
\end{equation}

Now let $\M$ be an \'etale $\varphi$-module over $W(\wh K_{\infty}^\flat)$ with a descent data: 
$$
\psi: \M \otimes_{W(\wh K_\infty^\flat),j_1} W((\wh K_\infty^{(2)})^\flat) \simeq \M \otimes_{W(\wh K_\infty^\flat),j_2} W((\wh K_\infty^{(2)})^\flat)
$$
as \'etale $\varphi$-modules over $W((\wh K_\infty^{(2)})^\flat)$ and satisfies cocycle condition over $W((\wh K_\infty^{(3)})^\flat)$. Using \eqref{eq-K(2)andhatGaction}, we have $\psi$ is the same as a descent data:
\begin{equation}\label{eq-descentdata-1}
\hat{\psi}: {\M} \otimes_{W(\wh K_\infty^\flat),j_1} \Cont(\hat{G}, W(\hat{L}^\flat))^{H_K^2} \simeq {\M} \otimes_{W(\wh K_\infty^\flat),j_2} \Cont(\hat{G}, W(\hat{L}^\flat))^{H_K^2}.
\end{equation}

For each $\gamma \in \hat{G}$, we have an evaluation map $\tilde{e}_\gamma: \Cont(\hat{G}, W(\hat{L}^\flat)) \to W(\hat{L}^\flat)$ given by evaluating at $\gamma$. Using \eqref{eq-j1j2}, one can check $\tilde{e}_\gamma \circ j_2: W(\wh K_\infty^\flat) \to W(\hat{L}^\flat)$ is given by the natural embedding and $\tilde{e}_\gamma \circ j_1: W(\wh K_\infty^\flat) \to W(\hat{L}^\flat)$ is given by $x\mapsto \gamma(x)$. So for each $\gamma \in \hat{G}$, if we tensor \eqref{eq-descentdata-1} against the evaluation map $\tilde{e}_\gamma$, we get an isomorphism:
$$
\psi_\gamma: {\M}\otimes_{W(\wh K_\infty^\flat),\gamma} W(\hat{L}^\flat) \simeq {\M}\otimes_{W(\wh K_\infty^\flat)} W(\hat{L}^\flat). 
$$
And similar to the classical Galois descent theory, the cocycle condition for $\psi$ implies $\{\psi_\gamma\}_\gamma$ satisfies 
$$
\psi_{\sigma \gamma} = \psi_\sigma \circ \sigma^\ast \psi_\gamma.
$$
Hence $\{\psi_\gamma\}_\gamma$ defines a continuous semilinear action of $\hat{G}$ on $\hat{\M}:=\M\otimes_{W(\wh K_\infty^\flat)} W(\hat{L}^\flat)$. One can check for $\gamma \in H_K$, we have the composition
$$
W(\wh K_\infty^\flat) \xrightarrow{j_k} W((\wh K_\infty^{(2)})^\flat) \to \Cont(\hat{G}, W(\hat{L}^\flat)) \xrightarrow{\tilde{e}_\gamma} W(\hat{L}^\flat)
$$
is the natural embedding $W(\wh K_\infty^\flat) \hookrightarrow W(\hat{L}^\flat)$ for $k=1,2$. And using the cocycle condition, one can show $\psi_\gamma=\id$ for $\gamma \in H_K$, so in particular, $\M \subset \hat{\M}^{H_K}$. Conversely, given a semilinear action of $\hat{G}$ on $\hat{\M}$ such that $\M \subset \hat{\M}^{H_K}$, $\{\psi_\gamma\}_\gamma$ defines a descent data $\psi$ over $\Cont(\hat{G}, W(\hat{L}^\flat))^{H_K^{2}}$ if and only if the semilinear action is continuous. In summary, we have

\begin{theorem}\label{thm-evaluation-1}
\begin{enumerate}
    \item The category of \'etale $\varphi$-modules over $A[1/E]^\wedge_p$ with descent data over $A^{(2)}[1/E]^\wedge_p$ is equivalent to the category of \'etale $(\varphi,\tau)$-modules over $A[1/E]^\wedge_p$;
    \item Given a descent data $f$ of an \'etale $\varphi$-module $\M$ over $A[1/E]^\wedge_p$, and $\gamma\in \hat{G}$, we can define the evaluation $f_\gamma$ of $f$ at $\gamma$, defined by the base change of $f$ along 
$$
e_{\gamma}: A^{(2)}[1/E]^\wedge_p \to (A^{(2)})_{\perf}[1/E]^\wedge_p \xrightarrow{\tilde{e}_\gamma} W(\hat{L}^\flat),
$$
which defines an isomorphism:
$$
f_\gamma: \M \otimes_{A[1/E]^\wedge_p,\tilde{\iota}_\gamma} W(\hat{L}^\flat) \simeq \M \otimes_{A[1/E]^\wedge_p} W(\hat{L}^\flat)
$$
where $\tilde{\iota}_\gamma: A[1/E]^\wedge_p \to W(\hat{L}^\flat) \xrightarrow{\gamma} W(\hat{L}^\flat)$. Suppose that  $(\M,f)$ corresponds to a $\Z_p$-representation $T$ of $G_K$, then $f_\gamma$ corresponds to the semilinear action of $\gamma$ on $\M$ inside $\M \otimes_{A[1/E]^\wedge_p} W(\C_p^\flat)\simeq T^\vee \otimes W(\C_p^\flat)$. Moreover, two descent data $f , g$ are equal if and only if $f_{\tau} = g_{\tau}$.
\end{enumerate} 
\end{theorem}
\begin{proof} The discussion above the theorem establishes the equivalence between the category of \'etale $\varphi$-modules over $A_{\perf}[1/E]^\wedge_p$ with descent data over $(A^{(2)})_{\perf}[1/E]^\wedge_p$ is equivalent to the category of \'etale $(\varphi,\tau)$-modules over $A[1/E]^\wedge_p$. Now  (1) follows \cite[Theorem 4.6]{wu2021galois} which shows that the category of \'etale $\varphi$-modules over $B[\frac 1 I]^\wedge_p$ is equivalent to the category of \'etale $\varphi$-modules over $B_{\perf}[\frac 1 I]^\wedge_{p}$ for bounded prism $(B, I)$ satisfying $\varphi (I) \mod p$ is generated by a non-zero
divisor in $B/p$. Then it just remains to prove the last statement in (2). Actually one can check (2) by chasing all the functors used in (1), and use the fact that for any \'etale $(\varphi,\tau)$-module, the $\hat{G}$-action on $\hat{\M}$ is determined by the $\tau$-action on $\M$. However, this can also be seen directly from the following lemma.
\end{proof}

\begin{lemma}\label{lem-evaluation-1}
Given two finite free \'etale $\varphi$-modules $\M,\mathcal{N}$ over $A^{(2)}[1/E]^\wedge_p$ and two morphisms $f, g: \M \to \mathcal{N}$ of \'etale $\varphi$-modules over $A^{(2)}[1/E]^\wedge_p$. Let $f_\tau,g_\tau$ be the base changes of $f,g$ along the map 
$$
e_{\tau}: A^{(2)}[1/E]^\wedge_p \to (A^{(2)})_{\perf}[1/E]^\wedge_p \simeq \Cont\Big(\hat{G}, W\big((\hat{L}^{(2)})^\flat\big)\Big)^{H_K^{2}} \xrightarrow{\tilde{e}_\tau} W(\hat{L}^\flat).
$$
Then $f=g$ if and only if $f_\tau=g_\tau$.
\end{lemma}

\begin{proof}
We take the natural base change of $f$ and $g$ along $A^{(2)}[1/E]^\wedge_p \to (A^{(2)})_{\perf}[1/E]^\wedge_p$, we get two morphisms $\psi$ and $\psi'$ between \'etale $\varphi$-modules over $(A^{(2)})_{\perf}[1/E]^\wedge_p$. Since the base change functor between \'etale $\varphi$-modules over $A^{(2)}[1/E]^\wedge_p$ and $(A^{(2)})_{\perf}[1/E]^\wedge_p$ is an equivalence of categories, it reduces to show that $\psi=\psi'$ if and only if their base change along 
$$
\tilde{e}_{\tau}: (A^{(2)})_{\perf}[1/E]^\wedge_p \simeq \Cont\Big(\hat{G}, W\big((\hat{L}^{(2)})^\flat\big)\Big)^{H_K^{2}} \xrightarrow{} W(\hat{L}^\flat)
$$
is equal. Since $\M$ and $\mathcal{N}$ are finite free, it is enough to show the evaluation map:
$$
\tilde{e}_{\tau}: \Cont\Big(\hat{G}, W\big((\hat{L}^{(2)})^\flat\big)\Big)^{H_K^{2}} \to W\big((\hat{L}^{(2)})^\flat\big)
$$
is injective. Suppose $h\in \Cont\Big(\hat{G}, W\big((\hat{L}^{(2)})^\flat\big)\Big)^{H_K^{2}}$ satisfies $h(\tau)=0$, then 
$$
(\sigma_1,\sigma_2)(h)(\tau)=\sigma_2 h(\sigma_2^{-1}\tau\sigma_1)=0
$$
for $(\sigma_1,\sigma_2)\in H_K^{2}$. Since $\hat{G}$ is topologically generated by $H_K$ and $\tau$, we get $h\equiv 0$.
\end{proof}

Now we give the $\Q$-isogeny versions of Theorem \ref{thm-evaluation-1} and Lemma \ref{lem-evaluation-1}. 
Recall that  the \'etale $(\varphi,\tau)$-modules over $A[1/E]^\wedge_p[\frac{1}{p}]$ is equivalent to the category $\RepQp(G_K)$, and recall the following definition of \'etale $(\varphi,\tau)$-modules over $B[1/J]^\wedge_p[\frac{1}{p}]$ for a prism $(B,J)\in X_\Prism$.

\begin{definition}\label{def-etalephimodule-2}
An (globally) \'etale $\varphi$-module $\M$ over $B[1/J]^\wedge_p[\frac{1}{p}]$ is a (finite projective) $\varphi$-module over $B[1/J]^\wedge_p[\frac{1}{p}]$ that arises by base extension from an \'etale $\varphi$-module $B[1/J]^\wedge_p$.
\end{definition}

From this definition, we immediately deduce the following result from \cite[Theorem 4.6]{wu2021galois}
\begin{proposition}
For any prism $(B,J)\in X_\Prism$ satisfying $\varphi (J) \mod p$ is generated by a non-zero
divisor in $B/p$, base change defined by $B[1/J]^\wedge_p[\frac{1}{p}]\to B_{\perf} [1/J]^\wedge_p[\frac{1}{p}]$ induces an equivalence between the category of \'etale $\varphi$-modules over $B[1/J]^\wedge_p[\frac{1}{p}]$ and the category of \'etale $\varphi$-modules over $B_{\perf}[1/J]^\wedge_p[\frac{1}{p}]$.
\end{proposition}

And similar to Theorem~\ref{thm-evaluation-1} and Lemma~\ref{lem-evaluation-1}, we have
\begin{theorem}\label{thm-evaluation-2}
The category of \'etale $\varphi$-modules over $A[\frac{1}{E}]^\wedge_p[\frac{1}{p}]$ with descent data over $A^{(2)}[\frac{1}{E}]^\wedge_p[\frac{1}{p}]$ is equivalent to the category of \'etale $(\varphi,\tau)$-modules over $A[\frac{1}{E}]^\wedge_p[\frac{1}{p}]$. Moreover, 
$$
\Cont\big(\hat{G}, W(\hat{L}^\flat)[\frac{1}{p}]\big)^{H_K^{2}} \simeq W(\wh K_\infty^{(2)})^\flat[\frac{1}{p}]. 
$$
For  $\gamma\in \hat{G}$, we can define the evaluation map
$$
\tilde{e}_\gamma: \Cont\big(\hat{G}, W(\hat{L}^\flat)[\frac{1}{p}]\big) \to W(\hat{L}^\flat)[\frac{1}{p}].
$$
And given a descent data $f$ of an \'etale $\varphi$-module $\M$ over $A[\frac{1}{E}]^\wedge_p[\frac{1}{p}]$, and $\gamma\in \hat{G}$, we can define the evaluation $f_\gamma$ of $f$ at $\gamma$, defined by the base change of $f$ along 
$$
e_\gamma: A^{(2)}[\frac{1}{E}]^\wedge_p[\frac{1}{p}] \to (A^{(2)})_{\perf}[\frac{1}{E}]^\wedge_p[\frac{1}{p}] \xrightarrow{\tilde{e}_\gamma} W(\hat{L}^\flat)[\frac{1}{p}],
$$
which defines an isomorphism:
$$
\M \otimes_{A[\frac{1}{E}]^\wedge_p[1/p],\tilde{\iota}_\gamma} W(\hat{L}^\flat)[\frac{1}{p}] \simeq \M \otimes_{A[\frac{1}{E}]^\wedge_p[1/p]} W(\hat{L}^\flat)[\frac{1}{p}]
$$
where $\tilde{\iota}_\gamma: A[\frac{1}{E}]^\wedge_p[\frac{1}{p}] \to W(\hat{L}^\flat)[\frac{1}{p}] \xrightarrow{\gamma} W(\hat{L}^\flat)[\frac{1}{p}]$. If $V$ corresponds $(\M,f)$ with in $V$ in $\RepQp(G_K)$, then $f_\gamma$ is the semilinear action of $\gamma$ on $\M $ inside $ V^\vee \otimes W(\C_p^\flat)[1/p]$. Moreover, two descent data $f , g$ are equal if and only if $f_{\tau} = g_{\tau}$.
\end{theorem}

\begin{lemma}\label{lem-evaluation-2}
Given two finite free \'etale $\varphi$-modules $\M,\mathcal{N}$ over $A^{(2)}[1/E]^\wedge_p[\frac{1}{p}]$ and two morphisms $f, g: \M \to \mathcal{N}$ of \'etale $\varphi$-modules over $A^{(2)}[1/E]^\wedge_p[\frac{1}{p}]$. Let $f_\tau,g_\tau$ be the base changes of $f,g$ along the map 
$$
e_\tau:A^{(2)}[1/E]^\wedge_p[\frac{1}{p}] \to (A^{(2)})_{\perf}[1/E]^\wedge_p[\frac{1}{p}] \xrightarrow{\tilde{e}_\tau} W(\hat{L}^\flat)[\frac{1}{p}].
$$
Then $f=g$ if and only if $f_\tau=g_\tau$.
\end{lemma}

\begin{proof}
The proofs are the same as the proof of Theorem~\ref{thm-evaluation-1} and Lemma~\ref{lem-evaluation-1}, plus the following fact that
$$
\Cont\big(\hat{G}, W(\hat{L}^\flat)[\frac{1}{p}]\big) = \Cont\big(\hat{G}, W(\hat{L}^\flat)\big)[\frac{1}{p}],
$$
which can be shown by the compactness of $\hat{G}$. 
\end{proof}

\subsection{Proofs of Proposition~\ref{thm-1prime} and Theorem \ref{Thm-main-1}}\label{subsec-pris-crystal-proof} We keep the assumption that  $R=\O_K$ is a mixed characteristic complete DVR with perfect residue field in this subsection, and keep our notations in \S 2.1.

Let us first prove Proposition~\ref{thm-1prime} using Lemma~\ref{lem-intersection} and results in \S\ref{subsec-phi-tau}. First, we give a different interpretation of the ``evaluation map":
$$
e_\gamma: A^{(2)}[1/E]^\wedge_p \to (A^{(2)})_{\perf}[1/E]^\wedge_p \simeq \Cont\Big(\hat{G}, W\big((\hat{L}^{(2)})^\flat\big)\Big)^{H_K^{2}} \xrightarrow{\tilde{e}_\gamma} W(\hat{L}^\flat)
$$
in Theorem~\ref{thm-evaluation-1} when restricted on $A^{(2)}$ . Recall that we fix a compatible system $\{\varpi_n\}_n$ of $p^n$-th roots of a uniformizer $\varpi \in \O_K$, this defines a map of prisms $\iota: (A,(E)) \to (\Ainf,(E))$ maps $u$ to $[{\varpi}^\flat ]$, and given a $\gamma \in G_K$, we define $\iota_{\gamma}$ to be the composition of $\iota$ with $\gamma: (\Ainf,(E)) \to (\Ainf,(E))$ where the second map is defined as $a \mapsto \gamma(a)$. Since $(E)\subset \Ainf$ is equal to $\Ker(\theta)$ and $\theta$ is  $G_K$-equivariant,  $\gamma$ is a well-defined map of $\delta$-pairs. By the universal property of $\At$, we can define a map of prisms $\iota_{\gamma}^{(2)} : (\At,(E)) \to (\Ainf,(E))$ so that the following diagram commutes:  
\begin{equation}\label{equ-diagram-prisms}
\begin{tikzcd}
(A,(E)) \arrow[rr,"i_1"] \arrow[ddrr,"\iota_\gamma",swap] & & (\At, (E) )\arrow[dd,"\iota^{(2)}_\gamma"] & & (A,(E)) \arrow[ll,"i_2",swap] \arrow[ddll,"\iota"]\\
& & & &\\
& & (\Ainf, (E)) & &
\end{tikzcd}  
\end{equation}
We have $\iota^{(2)}_{\gamma}$ induces a morphism $\tilde{\iota}^{(2)}_{\gamma}: A^{(2)}[1/E]^\wedge_p \to W(\C_p^\flat)$. We claim for all $\gamma \in G_K$, $\tilde{\iota}^{(2)}_{\gamma}$ is the same as the 
$$
A^{(2)}[1/E]^\wedge_p \to (A^{(2)})_{\perf}[1/E]^\wedge_p \simeq \Cont\Big(\hat{G}, W\big((\hat{L}^{(2)})^\flat\big)\Big)^{H_K^{2}} \xrightarrow{\tilde{e}_\gamma} W(\hat{L}^\flat) \hookrightarrow W(\C_p^\flat).
$$
To see this, by the universal property of direct perfection, we have \eqref{equ-diagram-prisms} factorizes as:
$$
\begin{tikzcd}
(A,(E)) \arrow[d]\arrow[rr,"i_1"]  & & (\At, (E) )\arrow[d] & & (A,(E)) \arrow[d] \arrow[ll,"i_2",swap]\\
(A_{\perf},(E)) \arrow[rr,"i'_1"] \arrow[ddrr,"\iota'_\gamma",swap] & & ((\At)_{\perf}, (E) )\arrow[dd,"\iota'^{(2)}_\gamma"] & & (A_{\perf},(E)) \arrow[ll,"i'_2",swap] \arrow[ddll,"\iota'"]\\
& & & &\\
& & (\Ainf, (E)) & &
\end{tikzcd}  
$$
So $\tilde{\iota}^{(2)}_\gamma$ has a factorization
$$
A^{(2)}[1/E]^\wedge_p \to (A^{(2)})_{\perf}[1/E]^\wedge_p \to W(\C_p^\flat).
$$
We need to check $\iota'^{(2)}_\tau$ induces the evaluation map 
$$
(A^{(2)})_{\perf}[1/E]^\wedge_p \simeq \Cont\Big(\hat{G}, W\big((\hat{L}^{(2)})^\flat\big)\Big)^{H_K^{2}} \xrightarrow{\tilde{e}_\tau} W(\hat{L}^\flat) \xhookrightarrow{} W(\C_p^\flat).
$$
And this follows from the isomorphism of $(A^{(2)})_{\perf}[1/E]^\wedge_p\simeq W((K^{(2)}_\infty)^\flat)$, then one check directly for $j_1,j_2$ defined in \eqref{eq-j1j2}, $\tilde{e}_\gamma\circ j_1: A_{\perf}[1/E]^\wedge_p \to W(\hat{L}^\flat)$ is equal to the map induced from $\iota'_\gamma$ and $\tilde{e}_\gamma\circ j_2: A_{\perf}[1/E]^\wedge_p \to W(\hat{L}^\flat)$ is equal to the map induced from $\iota'$. In particular, we have a commutative diagram:
\begin{equation}\label{eq-iotaandevaluation}
\begin{tikzcd}
A^{(2)} \arrow[d, hook] \arrow[rrr, "\iota^{(2)}_\gamma"] &&& \Ainf \arrow[d, hook]\\
A^{(2)}[1/E]^\wedge_p \arrow[r] & (A^{(2)})_{\perf}[1/E]^\wedge_p  \arrow[r,"\tilde{e}_\gamma",hook] & W(\hat{L}^\flat) \arrow[r,hook] & W(\C_p^\flat).
\end{tikzcd}  
\end{equation}
Now we can prove Proposition~\ref{thm-1prime}.

\begin{proof}[Proof of Proposition~\ref{thm-1prime}]
First we pick $\gamma=\tilde{\tau}$ that is a preimage of $\tau$ under the map $G_K \to \hat{G}$, we have $\gamma(u)-u=Ez$ and $\iota^{(2)}_{\gamma}$ defined as above is the embedding defined in \S\ref{subsec-embedding} by Remark~\ref{rem-embedding-depend}. In particular, composing the embedding $A^{(2)} \hookrightarrow \Ainf$ defined in \S\ref{subsec-embedding} with $\Ainf \hookrightarrow W(\C_p^\flat)$, one get the evaluation map 
$$
(A^{(2)})_{\perf}[1/E]^\wedge_p \simeq \Cont\Big(\hat{G}, W\big((\hat{L}^{(2)})^\flat\big)\Big)^{H_K^{2}} \xrightarrow{\tilde{e}_\tau} W(\hat{L}^\flat) \xhookrightarrow{} W(\C_p^\flat).
$$
restricted on $A^{(2)}$.

Keep the notations as in \S\ref{subsec-G-image}, let $\M_{\Ainf}=W(\C_p^\flat)\otimes_A \MM$ and $\M_A \simeq \MM\otimes_A A[1/E]^\wedge_p$. By  Theorem~\ref{thm-evaluation-1} and Theorem~\ref{thm-caruso}, recall we use $B^{(2)}= \At [\frac 1 E]^\wedge_p$ and $B^{(2)}_{\st}= \At_{\st} [\frac 1 E]^\wedge_p$ to simplify our notations, we have there is a descent data 
$$
c: \M_A \otimes _{A[1/E]^\wedge_p, \tilde{i}_1} B^{(2)} \to \M_A \otimes_{A[1/E]^\wedge_p, \tilde{i}_2} B^{(2)} 
$$
of $\M_A$ over $B^{(2)}$ that corresponds to the representation $T$. And the semilinear action of $\gamma=\tilde{\tau}$ on $\M_{\Ainf}$ is given by the evaluation $c_\tau$, that is,  we have the linearization of the $\tilde{\tau}$-action is defined by
$$
c_\tau: W(\C_p^\flat) \otimes_{\tilde{\iota}_\gamma,A[1/E]^\wedge_p} \M_A \simeq W(\C_p^\flat) \otimes_{\tilde{\iota} , A[1/E]^\wedge_p} \M_A.
$$
By  base change $c$ along $B^{(2)} \to B^{(2)}[\frac{1}{p}]$, we get a $B^{(2)}[\frac{1}{p}]$-linear $\varphi$-equivariant morphism:
$$
c': \M_A \otimes _{A[1/E]^\wedge_p, \tilde{i}_1} B^{(2)}[\frac{1}{p}] \to \M_A \otimes_{A[1/E]^\wedge_p, \tilde{i}_2} B^{(2)}[\frac{1}{p}]. 
$$
On the other hand, from the discussions after Proposition~\ref{thm-1prime}, $\tilde{\tau}$-action also defines a $\varphi$-equivariant morphism 
$$
f_{\tilde{\tau}}: \MM\otimes_{A,\iota_{\tilde{\tau}}} A_{\st}^{(2)}[\frac{1}{p}] \simeq \MM\otimes_{A} A_{\st}^{(2)}[\frac{1}{p}].
$$
We will see in Proposition~\ref{prop-descentBsttoB2} below that $f_{\tilde{\tau}}$ actually descents to a $B^{(2)}[1/p]$-linear morphism. Assuming this fact, then if we base change $f_{\tilde{\tau}}$ along $A^{(2)}[\frac{1}{p}] \to W(\C_p^\flat)[\frac{1}{p}]$, we will have $f_{\tilde{\tau}}\otimes W(\C_p^\flat)[\frac{1}{p}]=c_\tau$ since the way we define $f_{\tilde{\tau}}$ is by taking the ${\tilde{\tau}}$-action. From the discussion at the beginning of the proof and Lemma~\ref{lem-evaluation-2}, we have $f_{\tilde{\tau}}=c'$ as a $B^{(2)}[\frac{1}{p}]$-linear isomorphism between $\M_A \otimes _{A[1/E]^\wedge_p, \tilde{i}_1} B^{(2)}[\frac{1}{p}] $ and $ \M_A \otimes_{A[1/E]^\wedge_p, \tilde{i}_2} B^{(2)}[\frac{1}{p}]$.

We fix a basis $\{e_i\}$ of $\MM$, for $j=1,2$ let $\{e^j_i\}$ be the basis of $\M_A \otimes _{A, \tilde{i'}_j} B^{(2)}[\frac{1}{p}] $ defined by $e^j_i=e_i\otimes 1$ and the tensor is via $A \to A[1/E]^\wedge_p \xrightarrow{\tilde{i}_j} B^{(2)}[1/p]$. So we can interpret $f_{\tilde{\tau}}=c'$ as matrix using this two basis, this matrix is $X_{\tilde{\tau}}$ from this definition, so it has coefficients inside $A_{\st}^{(2)}[\frac{1}{p}]$ by the discussion before 
Proposition~\ref{thm-1prime}. On the other hand, $X_{\tilde{\tau}}$ has coefficients in $B^{(2)}\subset B_{\st}^{(2)}$ since $c'$ is defined by the $B^{(2)}$-linear map $c$. So by Lemma~\ref{lem-intersection}, we have $X_{\tilde{\tau}}$ has coefficients inside $A_{\st}^{(2)}$. The same argument shows when $T$ is crystalline, then $X_{\tilde{\tau}}$ has coefficients inside $A^{(2)}$.
\end{proof}

\begin{proposition}\label{prop-descentBsttoB2}
Base change along $B^{(2)} \to A^{(2)}_{\st}[1/E]^\wedge_p$ defines an equivalence of categories of \'etale $\varphi$-modules over $B^{(2)}$ and $A^{(2)}_{\st}[1/E]^\wedge_p$ and an equivalence of categories of \'etale $\varphi$-modules over $B^{(2)}[1/p]$ and $A^{(2)}_{\st}[1/E]^\wedge_p[1/p]$.
\end{proposition}
\begin{proof}
By \cite[Theorem 4.6]{wu2021galois}, we just need to show the same result after perfections, we will show $(\At)_{\perf}= (\At_{\st})_{\perf}$ in Lemma~\ref{lem-perfectionofA2andAst2} using the logarithmic prismatic site.
\end{proof}

Now, let us prove Theorem~\ref{Thm-main-1} by first producing a functor $\mathcal T$ from prismatic $F$-crystals in finite $\O_\Prism$-modules to lattices inside a crystalline representation. For prism $A$, we use $i_k: A \to \At$ or $A^{(3)}$ for natural map from $A$ to $k$-th factor of $\At$ or $A^{(3)}$. The notation $i_{kl}: A ^{(2)} \to A ^{(3)}$ has the similar meaning. 

By Corollary \ref{cor-crystal-descentdata}, given a prismatic $F$-crystal ${\MM}_\Prism$, we obtain a Kisin module $(\MM , \varphi _{\MM})$ of height $h$ together with descent data
$f: \MM \otimes _{A, i_1} \At \to \MM \otimes_{A, i_2}\At $ so that $f$ satisfies the following cocycle condition $ i _{13} \otimes f = (i _{23} \otimes f) \circ (i _{12} \otimes f ) $, where $i_{kl} \otimes f$ is the base change of $f$ along $i_{kl}$, and $f$ also compatible with the $\varphi$-structure on the both sides of $f$. Note that the existence of $f$ follows from the crystal property of   $\MM_\Prism$: 
\begin{equation}\label{eqn-cocyclefromcrystal}
f: \MM \otimes _{A, i_1} \At \simeq \MM_\Prism((A^{(2)},(E))) \simeq \MM \otimes_{A, i_2}\At 
\end{equation}

We let $\M=\MM\otimes_A A[1/E]^\wedge_p$ and  $c=f\otimes_{\At} B^{(2)}$, then  $(\M,c)$ is an \'etale $\varphi$-module with descent data, which corresponds to a $\Z_p$-representation of $G_K$. Moreover the semilinear action of $G_K$ on $\MM\otimes_A W(\C_p^\flat)$ comes from $\{c_\gamma\}_{\gamma\in G_K}$ using the evaluation maps. If we define 
$$
f_\gamma: \Ainf \otimes_{\iota_\gamma,A} \MM \to \Ainf \otimes_{\iota , A} \MM
$$
as the base change of $f$ along $\iota_{\gamma}^{(2)}$, then by \eqref{eq-iotaandevaluation}, we have $c_\gamma=f_\gamma$. The $G_K$-semilinear action commutes with $\varphi$ as $f$ does. For any $\gamma \in G_K$, we have $\gamma (A) \subset W(k)[\![u , \epsilon-1]\!] \subset \At_{\st} \subset \Ainf$. Therefore, the $G_K$-action on the $\Ainf \otimes_A \MM$ defined the above factors through $\At_{\st} \otimes_A \MM$. We claim that $G_K$-action on $\wh \MM : = \At_{\st}\otimes_A \MM$ defines a $(\varphi, \hat G)$-module which corresponds to a crystalline representation.

First, for $\gamma \in G_\infty$, $\gamma(A) = A$ in $\Ainf$, we conclude $\iota^{(2)}_\gamma : \At \to \Ainf$ satisfies $\iota^{(2)}_{\gamma}\circ i_1=\iota^{(2)}_{\gamma}\circ i_2$. In particular, for any $\gamma \in G_\infty$ and $j=1,2$, using \eqref{eqn-cocyclefromcrystal} and the crystal property of $\MM_\Prism$,  $f_\gamma$ comes from the base change of \eqref{eqn-cocyclefromcrystal} along $\iota^{(2)}_\gamma : \At \to \Ainf$, in particular, we have 
$$
f_\gamma: \MM \otimes _{A, \iota^{(2)}_{\gamma}\circ i_1} \Ainf \simeq \MM_\Prism((\Ainf,\Ker\theta)) \simeq \MM \otimes_{A, \iota^{(2)}_{\gamma}\circ i_2}\Ainf. 
$$
Since $\iota^{(2)}_{\gamma}\circ i_1=\iota^{(2)}_{\gamma}\circ i_2$, we have $f_\gamma={\mathrm{id}}$ which means $\MM \subset (\wh \MM) ^{G_\infty}$. Similarly, $G_K$ acts on $\wh \MM/ I_+$ corresponds the base change of $f$ along 
$$A^{(2)} \xrightarrow{\iota^{(2)}_\gamma} \Ainf \to W(\bar{k}) $$
where the last arrow is the reduction modulo $W(\m)$ ($\m$ is the maximal ideal of $\O_{\C_p}^\flat$). One can check for all $\gamma\in G_K$ and $j=1,2$, we have 
$$
A \xrightarrow{i_j} A^{(2)} \xrightarrow{\iota^{(2)}_\gamma} \Ainf \to W(\bar{k})
$$
are all equal to $A \to W(k) \hookrightarrow W(\overline{k})$ with the first arrow given by $u \mapsto 0$. The above map induces a morphism of prisms $(A,(E)) \to (W(k),(p))$, then using \eqref{eqn-cocyclefromcrystal} and the crystal condition of $\MM_\Prism$, we can similarly prove that $G_K$ acts on $\wh \MM/ I_+$-trivially, so $(\MM, \varphi_{\MM}, G_K)$ is a $(\varphi, \hat G)$-module. Furthermore, $\wh T (\wh \MM)$ is crystalline by Corollary \ref{cor-crystalline} and Theorem~\ref{Thm-1}. 

\begin{remark}
In \S\ref{sec-logprismandsemistablereps}, we will consider a category consisting of modules with descent data, and similar arguments about the triviality of the Galois actions can be shown directly using the cocycle condition of the descent data. We summarize this fact in the following easy fact.
\end{remark}
\begin{lemma}
Let $q:(A^{(2)},(E)) \to (B,J)$ be a map of prisms satisfying $q\circ i_1 =q\circ i_2$, then for any descent data $f$ over $A^{(2)}$, the base change of $f$ along $q$ is the identity map.
\end{lemma}

To show the fully faithfulness of this functor, let $(\MM, f)$, $(\MM', f')$ be two Kisin modules with descent data, and assume there exists a map $ \alpha: \mathcal T ((\MM , f)) \to \mathcal T ((\MM' , f'))$ of lattices in crystalline representations, then from our construction of $\mathcal{T}$ and Theorem~\ref{thm-2}, $\alpha$ is induced from a map $\hat{\alpha}: (\MM , \varphi_{\MM}, \hat G_{\MM}) \to (\MM' , \varphi _{\MM'}, \hat G_{\MM'})$ between $(\varphi, \hat G)$-modules. The faithfulness of $\mathcal{T}$ follows from the fact that $A \to A[1/E]^\wedge_p$ induces a fully faithful functor between Kisin modules over $A$ and \'etale $\varphi$-modules over $A[1/E]^\wedge_p$ from \cite[Proposition 2.1.12]{KisinFcrystal}. On the other hand,  $\hat{\alpha}$ gives morphisms $\hat{\alpha}_1: \MM \otimes_{A,i_1} A^{(2)} \to \MM' \otimes_{A,i_1}A^{(2)}$ and $\hat{\alpha}_2: \MM \otimes_{A,i_2} A^{(2)} \to \MM' \otimes_{A,i_2}A^{(2)}$. If we view $A$ and $A^{(2)}$ as subrings of $\Ainf$ using diagram \eqref{equ-diagram-prisms}, then the following diagram commutes by the fact that $\hat{\alpha}: \wh{\MM} \to \wh{\MM'}$ is compatible with $\tau$-action.
$$
\begin{tikzcd}
\MM \otimes_{A,i_1} A^{(2)} \arrow[r,"f"] \arrow[d,"\hat{\alpha}_1"] & \MM \otimes_{A,i_2} A^{(2)} \arrow[d,"\hat{\alpha}_2"] \\
\MM' \otimes_{A,i_1} A^{(2)} \arrow[r,"f'"] & \MM' \otimes_{A,i_2} A^{(2)}  \\
\end{tikzcd}
$$
Thus we produces a morphism between $(\MM, f)$ and $(\MM', f')$, i.e. $\mathcal{T}$ is also full. 

It remains to show the functor $\mathcal{T}$ is essential surjective. Given a lattice $T$ in a crystalline representation of $G_K$, let $\MM$ be the corresponding Kisin module, it suffices to construct a descent data of $\MM$ over $A^{(2)}$. We have shown in our proof of Proposition~\ref{thm-1prime} that if we view $A^{(2)}$ as a subring of $\Ainf$ via $\iota^{(2)}_{\tilde{\tau}}$, then $X_{\tilde{\tau}}$ defines a $\varphi$-equivariant isomorphism $f: \MM\otimes_{A,i_1} A^{(2)} \simeq \MM\otimes_{A,i_2} A^{(2)}$ of $A^{(2)}$-modules. We also show the base change of $f$ along $A^{(2)} \to B^{(2)}$ is equal to the descent data  $c$ of the \'etale $\varphi$-module $\M_A=\MM\otimes_A A[1/E]^\wedge_p$ that corresponds to $G_K$-action on $T$. In particular, $c: \MM\otimes_{A,i_1} B^{(2)} \simeq \MM\otimes_{A,i_2} B^{(2)}$ satisfies the cocycle condition. By Lemma \ref{lem-intersection}, $A^{(2)}$ (resp. $A^{(3)}$) injects into $B^{(2)}$ (resp. $B^{(3)}$), so we have $f$ also satisfies the cocycle condition. In particular, $(\MM,f)$ together produce a primatic $F$-crystals in finite free $\O_\Prism$-module by Corollary~\ref{cor-crystal-descentdata}.

\begin{remark}
Given an \'etale $\varphi$-module $(\M_A,\varphi_{\M_A}, c)$ over $A[1/E]^\wedge_p$ with descent datum $c$, we call $(\M_A,\varphi_{\M_A}, c)$ is \emph{of finite $E$-height} if $\M_A$ is of finite $E$-height, i.e., if there is a finite free Kisin module $(\MM,\varphi_{\MM})$ of finite height and defined over $A$ such that $\MM\otimes_A A[1/E]^\wedge_p \simeq \M_A$ as $\varphi$-modules. Since $(\M_A, \varphi_{\M_A})$ is the \'etale $\varphi$-module for $T|_{G_\infty}$, our definition of finite $E$-height is compatible with the one given by Kisin under the equivalence in (1) of Theorem~\ref{thm-evaluation-1}. 

We expect the same arguments in the proof of Proposition~\ref{thm-1prime} will be used to study representations of finite $E$-height. A similar result has been studied using the theory of $(\varphi,\tau)$-modules by Caruso. For example, in the proof of \cite[Lemma 2.23]{Caruso-phitau}, Caruso shows for representations of finite $E$-height, the $\tau$-actions descents to $\SSS_{u\text{-np},\tau}$, which is a subring of $\Ainf$ closely related to $\tilde{\iota}^{(2)}_{\tilde\tau}(B^{(2)})\cap \Ainf$, where $\tilde{\tau}$ is a preimage of $\tau$ in $G_K$. 
\end{remark}

\begin{remark}
We can also establish the compatibility of our Theorem~\ref{Thm-main-1}, the theory of Kisin and \cite[Theorem 1.2]{BS2021Fcrystals}. Given a lattice $T$ in a crystalline representation of $G_K$ with non-negative Hodge-Tate weight, and let $\MM$ be the Kisin module corresponds to $T$ in \cite{KisinFcrystal}, and let $\MM_{\Prism}$ (reso. $\MM'_{\Prism}$) be the prismatic $F$-crystal corresponds to $T^\vee$ under \cite[Theorem 1.2]{BS2021Fcrystals} (resp. $T$ under Theorem~\ref{Thm-main-1}). Note that we need to take $T^\vee$ since in the work of Bhatt-Scholze, the equivalence is covariant. By our construction of $\MM'_{\Prism}$, we have $\MM'_{\Prism}((A,(E)))\simeq \MM$. By \cite[Remark 7.11]{BS2021Fcrystals}, $\MM_{\Prism}((A,(E)))\simeq \MM$. Next we need to show the descent data over $A^{(2)}$ constructed respectively are the same. By Corollary~\ref{cor-inj}, we just need to show they are the same as descent data of \'etale $\varphi$-modules over $A^{(2)}[1/E]^\wedge_p$, but they are the same by our $\tau$-evaluation criteria in Lemma~\ref{lem-evaluation-1}.  
\end{remark}

\section{Logarithmic prismatic \texorpdfstring{$F$}{F}-crystals and semi-stable representations}\label{sec-logprismandsemistablereps}
In this section, we will propose a possible generalization of Theorem~\ref{Thm-main-1} to semi-stable representations using the absolute logarithmic prismatic site. The main reference of this subsection is \cite{Koshikawa2021log-prism}. We will restrict ourselves to the base ring $R=\O_K$, a complete DVR with perfect residue field. And we give $R$ the log structure associated to the prelog structure $\alpha: \N \to R$ such that $\alpha(1)=\varpi$ is a uniformizer in $R$, i.e., let $D=\{\varpi=0\}$, then the log structure on $X=\Spf(R)$ is defined by 
$$
M_X=M_D \hookrightarrow \O_X \text{ where } M_D(U):=\{f\in \O_X(U) \,|\, f|_{U\backslash D}\in \O^\times(U\backslash D) \}.
$$
Let us introduce the absolute logarithmic site over $(X,M_X)$.
\begin{definition}\cite[Definition 2.2 and Definition 3.3]{Koshikawa2021log-prism}
\begin{enumerate}
    \item A $\delta_{\log}$-ring is a tuple $(A,\delta, \alpha:M\to A, \delta_{\log}:M\to A)$, where $(A,\delta)$ is a $\delta$-pair and $\alpha$ is a prelog-structure on $A$. And $\delta_{\log}$ satisfies:
    \begin{itemize}
        \item $\delta_{\log}(e)=0$,
        \item $\delta(\alpha(m))=\alpha(m)^p\deltalog(m)$,
        \item $\deltalog(mn)=\deltalog(m)+\deltalog(n)+p\deltalog(m)\deltalog(n)$
        for all $m,n\in M$. And we will simply denote it by $(A,M)$ if this is no confusion. Morphisms are morphisms of $\delta$-pairs that are compatible with the perlog structure and $\deltalog$-stucture.
    \end{itemize}
    \item A $\delta_{\log}$-triple is $(A,I,M)$ such that $(A,I)$ is a $\delta$-pair and $(A,M)$ is a $\delta_{\log}$-ring.
    \item A $\delta_{\log}$-triple $(A,I,M)$ is a prelog prism if $(A,I)$ is a prism, and it is bounded if $(A,I)$ is bounded.
    \item A bounded prelog prism is a log prism if it is $(p, I )$-adically log-affine (cf. \cite[Definition 3.3]{Koshikawa2021log-prism}). 
    \item A bounded (pre)log prism is integral if $M$ is an integral monoid.
    \item A $\delta_{\log}$-triple $(A,I,M)$ is said to be over $(R,\N)$ if $A/I$ is an $R$-algebra and there is a map $M\to \mathbb N$ of monoids such that the following diagram commutes.
    $$
    \begin{tikzcd}
    M \arrow[rr] \arrow[d] & & A \arrow[d] \\
    \N \arrow[r] & R \arrow[r] & A/I
    \end{tikzcd}
    $$
    All $\delta_{\log}$-triples over $(R,\N)$ form a category. Similarly, we can define the category of prelog prisms over $(R,\N)$ and the category of bounded log prisms over $(R,\N)^a$.
\end{enumerate}
\end{definition}

\begin{remark}
If $A$ is an integral domain, or more generally if $\alpha(M)$ consists of non-zero divisors, then $\deltalog$ is uniquely determined by $\delta$ if exists. In particular, morphisms between such $\delta_{\log}$-rings are just morphisms of $\delta$-rings.
\end{remark}

\begin{remark}
Note that in this paper, for a $\delta$-pair $(A,I)$, we always assume $A$ is $(p,I)$-adic complete, but in \cite{Koshikawa2021log-prism}, non-$(p,I)$-adic completed $\delta_{\log}$-triples are also been studied. By Lemma 2.10 of loc.cit., we can always take the $(p,I)$-adic completion of the $\delta$-pair $(A,I)$ and the $\delta_{\log}$-structure will be inherited. 
\end{remark}

\begin{proposition}\cite[Corollary 2.15]{Koshikawa2021log-prism}
Given a bounded prelog prism $(A,I,M)$, one can associate it with a log prism
$$
(A,I,M)^a=(A,I,M^a)
$$
\end{proposition}
\begin{remark}
When we deal with log prisms in this paper, we will always take it as the log prism associated with some prelog prism. And by the above proposition, we know taking the associated log prism does not change the underlying $\delta$-pair. Moreover, it is a general fact that $(A,I,M)^a$ is integral if $(A,I,M)$ is integral.
\end{remark}

\begin{definition}\label{def-logprism}
The absolute logarithmic prismatic site $(X,M_X)_{\Prism_{\log}}$ is the opposite of the category whose objects are 
\begin{enumerate}
    \item bounded log prisms $(A,I,M_A)$ with \textit{integral} log structure,
    \item maps of formal schemes $f_A: \Spf(A/IA) \to X$,
    \item the map $f_A$ satisfies
$$
(\Spf(A/IA),f_A^\ast M_X) \to (\Spf(A),M_A)^a
$$
defines an exact closed immersion of log formal schemes.
\end{enumerate}
A morphism  $(A,I,M_A) \to (B,I,M_B)$ is a cover if and only if $A \to B$ is $(p,I)$-complete faithfully flat and the pullback induces an isomorphism on log structure. We will also use $\O_{\Prism}$ (resp. $\mathcal{I}_\Prism$) to denote the structure sheaf (resp. ideal sheaf of Hodge-Tate divisor) on $(X,M_X)_{\Prism_{\log}}$ by $(A,I,M_A) \mapsto A$ (resp. $(A,I,M_A) \mapsto I$).
\end{definition}

There is a variant of the about definition that we will also use in this subsection, we define $(X,M_X)_{\Prism_{\log}}^{\perf}$ be the full subcategory of $(X,M_X)_{\Prism_{\log}}$ whose  objects are $(A,I,M_A)$ with $A$ perfect.

\begin{remark}
Our definition of the absolute logarithmic prismatic site is different from \cite[Definition 4.1]{Koshikawa2021log-prism}. First, we need to consider the absolute prismatic site, not the relative one. Furthermore, we use the $(p,I)$-complete faithfully flat topology compared with the $(p,I)$-complete \'etale topology. Also we require the log structures to be integral. 
\end{remark}

\begin{proposition}\label{prop-logprismisasite}
$(X,M_X)_{\Prism_{\log}}$ forms a site.
\end{proposition}
\begin{proof}
Similar to \cite[Corollary 3.12]{BS19}, we need to show for a given diagram 
$$
\begin{tikzcd}
(C,I,M_C) & (A,I,M_A) \arrow[l,"c",swap] \arrow[r,"b"] & (B,I,M_B)
\end{tikzcd}
$$
in $(X,M_X)_{\Prism_{\log}}$ such that $b$ is a cover, then the pushout of $b$ along $c$ is a covering. From the argument in $loc. cit.$, we known for the underlying prisms, the pushout of $b$ along $c$ is the $(p,I)$-completed tensor product $D=C\wh {\otimes}_A B$, and $(D,I)$ is a bounded prism covers $(C,I)$ in the $(p,I)$-complete faithful flat topology. And we give $D$ the log structure $M_D$ defined by viewing $\Spf(D)$ as the fiber product via \cite[Proposition 2.1.2]{Ogus_logbook}, then $(C,M_C)\to (D,M_D)$ is strict morphism by Remark 2.1.3 of $loc.cit.$,  so in particular, $M_D$ is integral since $M_C$ is. For the same reason, 
$$
(\Spf(D/ID),f_D^\ast M_X) \to (\Spf(D),M_D)^a
$$ is strict since it is the base change of a strict morphism. It is an exact closed immersion since pushout of a surjective map of monoids is again surjective.
\end{proof}

\begin{example}\cite[Example 3.4]{Koshikawa2021log-prism}\label{exa-logprism}
\begin{enumerate}
    \item Let $(A,(E))$ be the Breuil--Kisin prism, then we can define a perlog structure on $(A,(E))$ by $\N \to A; n\mapsto u^n$, one has $(A,(E),\N)^a$ is in $(X,M_X)_{\Prism_{\log}}$, where (3) in Definition~\ref{def-logprism} follows from the prelog structures $\N \to R \to A/(E)$ and $\N \to A \to A/(E)$ induce the same log structure.
    \item For any prism $(B,J)$ over $(A,(E))$, it has a natural prelog structure $\N \to A \to B$, and similar to $(1)$, $(B,J,\N)^a$ is in $(X,M_X)_{\Prism_{\log}}$.
    \item A special case of (2) is that $(B,J)=(A_{\perf},(E))$, the perfection of $(A,(E))$. One has the prelog structure in (2) can be directly defined as $1\mapsto [\varpi^\flat]$. And $(A,(E),\N)^a \to(B,J,\N)^a$ is a covering of log prisms in $(X,M_X)_{\Prism_{\log}}$.
    \end{enumerate}
\end{example}

Actually, all logarithmic structures of log prisms in $(X,M_X)_{\Prism_{\log}}$ is the log structure associated with a prelog structure defined by $\N$. We thank Teruhisa Koshikawa for letting us know the following lemma.

\begin{lemma}\label{lem-Nchart}
For any log prism $(B,J,M_B)$ inside $(X,M_X)_{\Prism_{\log}}$, $(B,M_B)^a$ admits a chart $\N \to B$ defined by $n \mapsto u_B^n$ for some $u_B\in B$ satisfying $u_B \equiv \varpi \mod J$. 
\end{lemma}
\begin{proof}
For any log prism $(B,J,M_B)$ inside $(X,M_X)_{\Prism_{\log}}$, we have 
$$
(\Spf(B/J),f_B^\ast M_X) \to (\Spf(B),M_B)^a
$$
defines an exact closed immersion of log formal schemes. So by the proof of \cite[Proposition 3.7]{Koshikawa2021log-prism}, if we let $N^a_{B/J}:=\Gamma(\Spf(B/J),\underline{\N}^a)$ for the prelog structure $\N \to \O_K \to B/J$ induced from the given prelog structure on $\O_K$, then the fiber product $M_B \times_{N^a_{B/J}} \N$ is a chart for $(B,M_B)^a$. Moreover, since we assume $M_B$ to be integral, we have $(\Spf(B/J),f_B^\ast M_X) \to (\Spf(B),M_B)^a$ is a log thickening with ideal $J$ in the sense of \cite[Definition 2.1.1.]{Ogus_logbook}, and one can show $M_B \times_{N^a_{B/J}} \N \simeq \N \times (1+J)$. Now $(1+J) =(1+J)^\times$, so 
$$
\N \to \N \times (1+J) \simeq M_B \times_{N^a_{B/J}} \N \to B
$$
is also a chart for $(B,M_B)^a$. And the prelog structure given by $n \mapsto u_B^n$ for some $u_B\in B$ satisfying the image of $u_B$ in $B/J$ coincides with the image of $\varpi$ under $\O_K \to B/J$.
\end{proof}

In the rest of this subsection, we will try to generalize results we proved in \S\ref{subsec-pris-crystal}-\S\ref{subsec-pris-crystal-proof} for the logarithmic prismatic site. 

\begin{lemma}\label{lem-log-nonemptyproduct}
\begin{enumerate}
    \item For $(A,I_A,M_A)^a, (B,I_B,M_B)^a\in (X,M_X)_{\Prism_{\log}}$ such that $M_A,M_B$ are integral and $(A,M_A)\to (A/I_A,\N)$ and $(B,M_B)\to (B/I_B,\N)$ are exact surjective, there is a prelog prism $(C,I_C,M_C)$ with integral log structure that is universal in the sense that the diagram 
    $$
    \begin{tikzcd}
    (A,I_A,M_A) \arrow[r] & (C,I_C,M_C) & (B,I_B,M_B) \arrow[l] 
    \end{tikzcd}
    $$
    is initial in the category of diagrams 
    $$
    \begin{tikzcd}
    (A,I_A,M_A) \arrow[r] & (D,I_D,M_D) & (B,I_B,M_B) \arrow[l] 
    \end{tikzcd}
    $$
    of prelog prisms over $(R,\N)$, and $(D,M_D)\to (D/I_D,\N)$ is an exact surjective.
    \item If $(C,I_C)$ in (1) is bounded, then $(C,I_C,M_C)^a$ is the product of $(A,I_A,M_A)^a$ and $(B,I_B,M_B)^a$ inside $ (X,M_X)_{\Prism_{\log}}$.
    \item If both $(A,I_A,M_A)^a$ and $ (B,I_B,M_B)^a$ in (1) are inside $ (X,M_X)_{\Prism_{\log}}^{\perf}$, and let $(C_{\perf},I_C)$ be the perfection of $(C,I_C)$ defined in (1). Let $(C_{\perf},I_C,M_C)$ be the prelog prism with prelog structure induced from $C$. Then $(C_{\perf},I_C,M_C)^a$ is isomorphic to the product of $(A,I_A,M_A)^a$ and $(B,I_B,M_B)^a$ in $(X,M_X)_{\Prism_{\log}}^{\perf}$. 
\end{enumerate}
\end{lemma}
\begin{proof}
Let $(A,I_A,M_A),(B, I_B, M_B)\in (X,M_X)_{\Prism_{\log}}$, define $C_0$ to be the $(p,I_A,I_B)$-adic completion of $A\otimes_{W(k)}B$ and let $J$ be the kernel of
$$
C_0 \to A/I_A \wh {\otimes}_R B/I_B.
$$
Then $(C_0,J,M_A\times M_B)$ is a $\deltalog$-triple over $(A,I_A,M_A)$. And we have $(C_0,J,M_A\times M_B) \to (C_0/J,\N)$ is surjective. Then we can apply \cite[Proposition 3.6]{Koshikawa2021log-prism} to get a universal prelog prism $(C,I_C,M_C)$ over $(A,I_A,M_A)$ and $(B, I_B, M_B)$ and satisfies $(C,M_C)\to (C/J,\N)$ is exact surjective. Just recall in the proof of \cite[Proposition 3.6]{Koshikawa2021log-prism}, we first construct a $\deltalog$-triple $(C',J',M_C')$ which is universal in the sense that it is a $\deltalog$-triple over both $(A,I_A,M_A)$ and $(B, I_B, M_B)$ satisfying $C'/J'$ is over $A/I_A$ and $B/I_B$ as $R$-algebra and $(C',M_C')\to (C'/J',\N)$ is exact surjective. Then we take the prismatic envelope with respect to $(A,I_A) \to (C',J')$ to get $(C,I_C)$. Then we can check such $(C,I_C,M_C)$ satisfies the universal property. For (2), when $(C,I_C)$ is bounded, the fact that $(C,I_C,M_C)^a$ is the product of $(A,I_A,M_A)^a$ and $(B,I_B,M_B)^a$ inside $ (X,M_X)_{\Prism_{\log}}$ follows from Proposition 3.7 of $loc.cit.$. For (3), we have $(C_{\perf},I_C)$ is automatically bounded, and one can check $(C_{\perf},I_C)$ is universal using exactly the same proof of Proposition 3.7 of $loc.cit.$.
\end{proof}

We thank Koji Shimizu for the following lemma about $A^{(2)}_{\st}$. 

\begin{lemma}\label{lem-Ast2islogprism}
Let $(A,I,\N)^a$ be the Breuil--Kisin prism defined in $(1)$ of Example~\ref{exa-logprism}, then the self-product (resp. self-triple product) of $(A,I,\N)^a$ in $(X,M_X)_{\Prism_{\log}}$ exist. Moreover, if we let $(A^{\langle 2 \rangle},I,M^2)^a$ (resp. $(A^{\langle 3 \rangle},I,M^3)^a$) be self-product (resp. self-triple product) of $(A,I,\N)^a$, then $A^{\langle i \rangle}\simeq A^{(i)}_{\st}$ for $i=2,3$.
\end{lemma}
\begin{proof}
By our construction in Lemma~\ref{lem-log-nonemptyproduct}, $(A^{\langle 2 \rangle},I,M)$ is the prelog prismatic envelope $(C,I_C,M_C)$ with respect to
    $$
    (A,(E),\N) \to (C_0,J,\N^2) \text{ and } (C_0/J,\N^2)\to (R,\N)
    $$
    where $C_0=W[\![u,v]\!]$, $J=(E(u),u-v)$ with the prelog structure given by $\beta: (1,0)\mapsto u, (0,1)\mapsto v$. The prelog prismatic envelope is constructed using the technique of exactification: consider $\pi: (C_0,\N^2)\to (R=C/J,\N)$ where the map between log structures is given by $\pi_{\log}: \N\times \N \to \N;(m,n)\mapsto m+n$, here $\pi_{\log}$ is surjective but not exact, so to construct the exactification of $\pi: (C,\N^2)\to (R,\N)$ (cf. \cite[Construction 2.18]{Koshikawa2021log-prism}), first we have the (complete) exactification of $\pi_{\log}$ is 
    $$
    \alpha: M^2 \to \N \quad \text{ given by } \quad (m,n) \mapsto m+n,
    $$
    where $M^2=\{(m,n)\in \Z\times \Z \,|\, m+n\in \N \}$. Since $M^2$ is generated by $(-1,1)$, $(1,-1)$, $(0,1)$ and $(1,0)$, one has the exactification of $\pi$ is 
    $$
    \Big( W(k)[\![u,v]\!]\big[\frac{v}{u},\frac{u}{v}\big]^\wedge_{(p,J')}, J',M^2; \alpha: (1,0)\mapsto {u}, (0,1)\mapsto v, (\pm 1,\mp1)\mapsto \pm\frac{u}{v} \Big)
    $$
    where $J':=\ker(W(k)[\![u,v]\!]\big[\frac{v}{u},\frac{u}{v}\big] \to R)$. 
    
    We have the $(p,J')$-adic completion of $W(k)[\![u,v]\!]\big[\frac{v}{u},\frac{u}{v}\big]$ is $W(k)[\![u,\frac{v}{u}-1]\!]$. Then take prismatic envelope of 
    $
    (A,(E))\to (W(k)[\![u,\frac{v}{u}-1]\!], (E,\frac{v}{u}-1)).
    $ One can check 
    $$ W(k)[\![u,\frac{v}{u}-1]\!]\big\{\frac{v/u-1}{E(u)}\big\}^\wedge_\delta \simeq A_{\st}^{(2)}$$ 
    directly from the definition of $A_{\st}^{(2)}$.
    
    Similarly, we can show $A^{\langle 3 \rangle}\simeq A^{(3)}_{\st}$ which is also bounded. 
\end{proof}

The following is one of our key observations.
\begin{lemma}\label{lem-perfectionofA2andAst2}
We have $(A^{\langle 2 \rangle})_{\perf} \simeq (A^{(2)})_{\perf}$.
\end{lemma}
\begin{proof}
Let $u_1,u_2$ be the image of $u$ under the two natural maps $i_{j}: A_{\perf} \to (A^{(2)})_{\perf}$ for $j=1,2$. We claim that $u_2/u_1$ is inside $(A^{(2)})_{\perf}$.

Firstly, we have already shown $A_{\perf}\simeq W(\widehat{\O}_{K_\infty}^\flat)$ and $u=[\varpi^\flat]$, here $\varpi^\flat=(\varpi_n)$ with $\{\varpi_n\}_{n\geq 0}$ being a compatible system of $p^n$-th roots of $\varpi$ inside $\O_{\widehat{K}_\infty}$, and $(\varpi_n) \in \O_{\wh K_\infty}^\flat$ via the identification $\O_{\wh K_\infty}^\flat \simeq \lim_{x \mapsto x^p} \O_{\wh K_\infty}$. Let $S=(A^{(2)})_{\perf}/(E)$, this is an integral perfectoid ring over $\O_K$ in the sense of \cite{BMS1}. We have $S^\flat\simeq (A^{(2)})_{\perf}/(p)$. For $j=1,2$, define $\varpi_j^\flat=u_j \mod (p) \in S^\flat$, then we have $u_j=[\varpi_j^\flat]$ for $j=1,2$.

Recall in \S~\ref{subsrc-construct-A2}, we have $z = \frac{y -x}{E(x)}$ in $\At$. Since  $ E(x) \equiv  x^e \mod p$, we have $ x  (1 + x^{e-1} z) \equiv  y \mod p$. If we denote $\iota : \At \to (\At)_{\perf} $ the natural map, then $\iota(x)=u_1$ and $\iota(y)=u_2$ in our definition, and $u _1 (1 + u_1^{e-1} \iota(z)) \equiv u _2 \mod p$ inside $S^\flat=\At_{\perf}/(p)$. This is the same as $\varpi_1^\flat \mu =  \varpi_2^\flat$ with $\mu = (1 + u_1^{e-1} \iota(z)) \mod p$ in $S ^\flat$. So we have $ [\mu] u _1 = [\mu] [\varpi_1^\flat] = [\varpi_2^\flat]=  u_2$, which proves our claim.

Now by symmetry, $u_1/u_2$ is also inside $(A^{(2)})_{\perf}$, so $u_1/u_2$ is a unit in $(A^{(2)})_{\perf}$. So we can give $(A^{(2)})_{\perf}$ a prelog structure
$$
\alpha: M^2 \to (A^{(2)})_{\perf} \text{ with } (1,-1)\mapsto \frac{u_1}{u_2}, (-1,1)\mapsto \frac{u_2}{u_1}, (1,0)\mapsto {u_1}, (0,1)\mapsto {u_2}
$$
with the monoid $M^2$ defined as in the proof of Lemma~\ref{lem-Ast2islogprism}, then $((A^{(2)})_{\perf},(E),M^2)^a$ is in $X_{\Prism_{\log}}^{\perf}$. 

One can check the maps $i_1,i_2: (A,(E)) \to (A^{(2)},(E)) \to ((A^{(2)})_{\perf},(E))$ induce $i_1,i_2: (A_{\perf},(E),\N) \to ((A^{(2)})_{\perf},(E),M^2)$ of prelog prisms. So by Lemma~\ref{lem-Ast2islogprism}, there is a unique map $(A^{\langle 2 \rangle},I,M^2)\to ((A^{(2)})_{\perf},(E),M^2)$, this map will factor through $((A^{\langle 2 \rangle})_{\perf},(E),M^2)$, let $((A^{\langle 2 \rangle})_{\perf},(E),M^2) \to ((A^{(2)})_{\perf},(E),M^2)$ be the induced map in $X_{\Prism_{\log}}^{\perf}$. On the other hand, by the universal property of $A^{(2)}$, we know there is a map $(A^{(2)})_{\perf} \to (A^{\langle 2 \rangle})_{\perf}$ fits into the coproduct diagram in $X_{\Prism}^{\perf}$, which is the full subcategory of $X_\Prism$ containing perfect prisms.

We have the composition $\eta: ((A^{(2)})_{\perf},(E)) \to ((A^{\langle 2 \rangle})_{\perf},(E)) \to ((A^{(2)})_{\perf},(E))$ satisfies $\eta\circ i_j= i_j \circ \eta$ for $i_1,i_2:(A_{\perf},(E))\to ((A^{(2)})_{\perf},(E))$. Such a map is unique inside $X_{\Prism}^{\perf}$, so $\eta=\id_{((A^{(2)})_{\perf},(E))}$. 

On the other hand, the composition 
$$
\eta': ((A^{\langle 2 \rangle})_{\perf},(E),M^2)^a \to ((A^{(2)})_{\perf},(E),M^2)^a \to ((A^{\langle 2 \rangle})_{\perf},(E),M^2)^a
$$ satisfies $\eta\circ i'_j= i'_j \circ \eta$ for $i'_1,i'_2:(A_{\perf},(E),\N)^a\to ((A^{\langle 2 \rangle})_{\perf},(E),M^2)^a$ induced from $i'_1,i'_2:(A,(E),\N)\to (A^{\langle 2\rangle},(E),M^2)$. Such map is also unique inside $X_{\Prism_{\log}}^{\perf}$, so $\eta'=\id_{((A^{\langle 2 \rangle})_{\perf},(E),M^2)^a}$. So in particular we have $(A^{\langle 2 \rangle})_{\perf}\simeq (A^{(2)})_{\perf}$.
\end{proof}

Combining the above Lemma and our discussions in \S\ref{subsec-phi-tau}, we get the following logarithmic variant of prismatic $(\varphi,\tau)$-module theory.

\begin{theorem}\label{thm-log-phitau}
The category of \'etale $\varphi$-module over $A[1/E]^\wedge_p$ with a descent data over $A_{\st}^{(2)}[1/E]^\wedge_p$ is equivalent to the category of lattice in representations of $G_K$. Moreover, for all $\gamma\in\hat{G}$, we can define the evaluation map
$$
e_\gamma: A_{\st}^{(2)}[1/E]^\wedge_p \to W(\hat{L}^\flat)
$$
such that Lemma~\ref{lem-evaluation-1} is still valid. Moreover, the $\Q$-isogeney version of this theorem also holds.
\end{theorem}

\begin{remark}
The above theorem should be related to the \'etale comparison theorem in the logarithmic prismatic settings, which has not been studied in \cite{Koshikawa2021log-prism} yet.
\end{remark}

Moreover, we have the logarithmic version of Lemma~\ref{lem-AEcoversfinal} holds. We thank Teruhisa Koshikawa for hints of the following result.

\begin{proposition}\label{prop-logcoverfinalobject}
The sheaf represented by $(A,(E),\N)^a$ covers the final object $\ast$ in $\Shv((X,M_X)_{\Prism_{\log}})$.
\end{proposition}
\begin{proof}
For any log prism $(B,J,M_B)$, by Lemma~\ref{lem-Nchart}, we can assume $(B,J,M_B)^a=(B,J,\N)^a$, with prelog structure defined by $n \mapsto u_B^n$ with $u_B \equiv \varpi \mod J$.

Using deformation theory, we have there is a unique $W(k)$-algebra structure for $B$, and we define $C=B[\![u]\!][\frac{u_B}{u},\frac{u}{u_B}]\{\frac{u_B/u-1}{J}\}^\wedge_\delta$, where the completion is taken for the $(p,J)$-adic topology. Similar to the proof of Lemma~\ref{lem-Ast2islogprism}, we have $(C,JC,\N)^a$ is the product of $(A,(E),\N)^a$ and $(B,J,\N)^a$ inside $(X,M_X)_{\Prism_{\log}}$. Moreover, we have $B \to C$ is $(p, J)$-complete flat by \cite[Proposition 3.13]{BS19}. It remains to show that $(B,J) \to (C,J)$ is a covering, i.e., $B \to C$ is $(p, J)$-complete faithfully flat. Let 
$$
C^{nc}:=B[\![u]\!][\frac{u_B}{u},\frac{u}{u_B}]\{\frac{u_B/u-1}{J}\}_{\delta}
$$
be the non-complete version of $C$ that we have the $(p,J)$-adic completion of $C^{nc}$ is $C$. Now we just need to show the flat ring map $B/(p,J) \to C/(p,J)=C^{nc}/(p,J)$ is also faithful. 

We claim that $C/ (p , J)$ is free over $B / (p , J)$. One has $JC=E(u)C$, and $(p , J)= (p , E) = (p , J , E)$ in $C$. So $C/ (p , J)=C^{nc}/(p,J)$ is equal to 
\[ B [\![u ]\!][\frac{u_B}{u}, \frac{u}{u _B}][\delta^i(z), i \geq 0 ]/ \left (p , J , E , \delta^i (\frac{u_B}{u }-1))= \delta^i (Ez), i \geq 0 \right ).\]
After modulo $(p,J)$, the above is the direct limit of
\[B / (p , J)[\delta ^i (z)]/ \left (\delta^i (\frac{u_B}{u }-1))= \delta^i (Ez) \mod (p, E , J) \right )\]
for $i\geq 0$.

Now we use Lemma \ref{lem-delta-n} to compute $\delta^i (\frac{u_B}{u }-1)= \delta^i (Ez) \mod (p, E , J)$. We keep the notations in Lemma \ref{lem-delta-n}, by induction, we have $b_n = 0 \mod (p , E)$. Using that $a_p^{(j)}\in A_0^\times$, $\delta^i (\frac{u_B}{u }-1)= \delta^i (Ez) \mod (p, E , J)$ gives a relation $ (z_{i -1}) ^ p =   \sum\limits_{j= 0}^{p -1} \tilde a_j^{(i)} (z_{i-1}) ^j$ where $z_i = \z_i \mod (p , J , E)$ and $\tilde a_j^{(i)} \in B / (p , J)[z_0, z_1, \dots, z_{i-2}]$. In summary, we have 
$$C/ (p , J) = B/(p, J)[z_i, i \geq 0]\Bigg/ \left ((z_{i}) ^ p -   \sum\limits_{j= 0}^{p -1} \tilde a_j^{(i)} (z_{i}) ^j, i\geq 1 \right )$$
which is free over $B /(p, J)$. 
\end{proof}

\begin{definition}\label{def-logFcrystal}
\begin{enumerate}
\item A prismatic crystal over $(X,M_X)_{\Prism_{\log}}$ (in finite locally free $\O_{\Prism}$-modules) is a finite locally free $\O_{\Prism}$-module $\MM_{\Prism}$ such that for all morphisms $f: (A, I,M_A) \to (B, J,M_B)$ of log prisms over $(X,M_X)_{\Prism_{\log}}$, it induces an isomorphism:
$$
f^\ast \MM_{\Prism,A} := \MM_{\Prism}((A,I,M_A))\otimes_A B \simeq \MM_{\Prism,B} := \MM_{\Prism}((B,J,M_B));
$$

\item A prismatic $F$-crystal over $(X,M_X)_{\Prism_{\log}}$ of height $h$ (in finite locally free $\O_{\Prism}$-modules) is a prismatic crystal $\MM_{\Prism}$ in finite locally free $\O_{\Prism}$-modules together with a $\varphi_{\Prism}$-semilinear endomorphism $\varphi_{\MM_{\Prism}}$ of $\MM_{\Prism}$ whose linearization $\varphi_\Prism^\ast \MM_{\Prism} \to \MM_{\Prism}$ has cokernel killed by $\II_\Prism^h$.

\end{enumerate}
\end{definition}

In particular, with the help of Theorem~\ref{thm-log-phitau} and Proposition~\ref{prop-logcoverfinalobject}, a direct translation of proofs in \S\ref{subsec-pris-crystal-proof} with $\At$ replaced by $\At_{\st}$ shows the following theorem.
\begin{theorem}\label{thm-log-main-1}
The category of prismatic $F$-crystals over $(X,M_X)_{\Prism_{\log}}$ of height $h$ is equivalent to the category of lattices in semi-stable representations of $G_K$ with Hodge-Tate weights between $0$ and $h$.
\end{theorem}

\appendix
\section{Some discussions on base rings}\label{subsec-baserings} In this appendix, we show that our base ring assumed at the beginning of \S \ref{sec-ring-strcuture} covers many situations of base rings used in \cite{Kim12} and \cite{Brinon}. 

Let $K$ be complete DVR with perfect residue field $k$, and let $K_0=W[\frac{1}{p}]$ with $W=W(k)$, fix a uniformizer $\varpi\in \O_K$ and $E(u)\in W[u]$ a minimal polynomial of $\varpi$ over $K_0$. Let $R$ be a normal domain and satisfies that $R$ is a $p$-complete flat $\O_K$-algebra that is complete with respect to $J$-adic topology, for an ideal  $J=(\varpi, {t_1},\ldots,{t_d})$ of $R$ containing $\varpi$. We also assume $\overline{R}=R/(\varpi)$ is a finite generated $k$-algebra with \emph{finite $p$-basis} discussed in \cite[\S 1.1]{deJong}.

\begin{lemma}[\cite{Kim12} Lemma 2.3.1 and lemma 2.3.4]\label{Kim-lemma}
\begin{enumerate}
    \item In the above setting, there is a $p$-adic formally smooth flat $W$-algebra $R_0$ equipped with a Frobenius lift $\varphi_0$ such that $\overline{R}: = R_0/(p)$. Moreover let $J_0$ be the preimage of $\overline{J}$ inside $R_0$, then $R_0$ is $J_0$-adically complete, and under this topology, $R_0$ is formally smooth. 
    \item $R_0/(p)\xrightarrow{\sim}R/(\varpi)$ lifts to a $W$-algebra morphism $R_0 \to R$ and the induced $\O_K$-algebra morphism $\O_K\otimes_W R_0 \to R$ is an isomorphism. Moreover this isomorphism is continuous with respect to the $J_0$-adic topology.
\end{enumerate}
\end{lemma}

Let $(R_0, \varphi_{R_0})$ denote a flat $W$-lift of $R/(\varpi)$ obtained from the above lemma. And we will have $J_0=(p, t_1, \ldots, t_d)\in R_0$, and we write $\overline{J}=(\overline{t_1},\ldots, \overline{t_d})\subset \overline{R}$. 

\begin{definition}\label{RAE}
Let $R_0$ be a $p$-complete $\Z_p$-algebra, we say $R_0$ satisfies the ``refined almost étalenes" assumption, or simply RAE assumption, if $\hat{\Omega}_{R_0}=\oplus_{i=1}^m R_0 dT_i$ with $T_i\in R_0^\times$. Where $\hat{\Omega}_{R_0}$ is the module of $p$-adically continuous K\"ahler differentials.
\end{definition}
The following are examples of $R_0$ and $R$ which satisfy assumptions of Lemma \ref{Kim-lemma} and RAE assumption. 
\begin{example}
\begin{enumerate}
    \item If $R/(\varpi)$ is a completed noetherian regular local ring with residue field $k$, then Cohen structure theorem implies
     $R/(\varpi)=k[\![\overline{x_1},\ldots,\overline{x_d}]\!]$. In this case, $R_0=W[\![x_1,\ldots, x_d]\!]$ and $J_0=(p,x_1,\ldots, x_d)$. Then $R=W[\![x_1,\ldots, x_d]\!][u]/E$, with $E\in W[u]$ is a Eisenstein polynomial.
    \item Let $R_0 = W(k) \langle t _1^{\pm 1} ,  \dots , t _m ^{\pm 1}\rangle$ and $J_0=(p)$, in this example, $\overline{R}=k[\overline{t}_1^{\pm 1} ,  \dots , \overline{t}_m ^{\pm 1}]$ is not local.
    \item An unramified complete DVR $(R_0 , p)$ with residue field $k$ so that $[k : k ^p]<\infty$. 
    \item Note the the Frobenius liftings in Lemma ~\ref{Kim-lemma} is not unique. In (2), we can choose $\varphi_{R_0}(t_i)=t_i^p$. In (1), we can choose the $\varphi_{R_0}(x_i)=x_i^p$ or $\varphi_{R_0}(x_i)=(x_i+1)^p-1$.
\end{enumerate}
\end{example}
Let $R_0$ be a $p$-complete algebra which satisfies the RAE assumption, we set $\breve R_0 = W\langle t_1 , \dots , t_m \rangle$ and $f : \breve R_0 \to R_0$ by sending $t_i$ to $T_i$. 
\begin{proposition}\label{prop-fetale} Assume that $R_0$ is a $p$-complete integral domain which admits finite $p$-basis and satisfies RAE assumption.  
Then $f$ is formally \'etale $p$-adically. 
\end{proposition}
\begin{proof} We thank Wansu Kim for providing the following proof. By standard technique using \cite[Ch.III, Corollaire 2.1.3.3]{Illusie1} (e.g., see the proof in \cite[Lem. 2.3.1]{Kim12}), it suffices to show that the cotangent complex 
$\mathbb L_{R_0 / \breve R_0}$ is acyclic. Since both $R_0$ and $\breve R_0$ are $\Z_p$-flat, it suffice to show that $\mathbb L_{R_1 / \breve R_1}$ is acyclic where $R_1 = R_ 0 / p R_0$ and $\breve R_1 = \breve R_0 / p \breve R_0$. Since $R_0$ has finite $p$-basis, by \cite[Lem. 1.1.2]{deJong}, $\mathbb L_{R_1 /k}\simeq \Omega_{R_1/k}$. Note that maps $k \to \breve R_1 \to R_1$ induces a fiber sequence 
\[ \mathbb L_{\breve R_1 /k}\otimes^{\mathbb L}_{\breve R_1} R_1 \to \mathbb L _{R_1 / k} \to \mathbb L_{R_1 / \breve R_1}\]
Since that $ \mathbb L_{\breve R_1 /k} \simeq \Omega_{\breve R_1/k}$ and $\Omega_{\breve R_1/k}\simeq \Omega_{R_1/k}$ by RAE condition, we conclude that $\mathbb L_{R_1/ \breve R_1}= 0$ as required. 
\end{proof}
Let us end with a discussion about our base rings and the base rings used in \cite{Brinon}. As explained at the beginning of \cite[Chap. 2]{Brinon}, his base ring $R_0$ in \cite{Brinon} is  obtained from $W\langle t_1^{\pm 1}, \ldots, t_m^{\pm 1}\rangle$ by a finite number of iterations of certain  operations and is also assumed to satisfy certain properties.  By Prop. 2.0.2 \emph{loc. cit.}, we see that $R_0$ has finite $p$-basis and satisfies RAE assumption. So the base ring $R_0$ in \cite{Brinon} also satisfies the requirement that $f:W\langle t_1, \ldots, t_m\rangle \to R_0$ is formally \'etale by Proposition \ref{prop-fetale}.

\end{document}